\documentclass{imanum}
\usepackage{color,graphics,graphicx}
\usepackage{url}
\usepackage{epstopdf}

\jno{drnxxx}
\received{October 21, 2020}
\revised{\today}
\accepted{}

\newcommand {\R}	{\mathbb{R}}
\newcommand {\N}	{\mathbb{N}}

\DeclareMathOperator{\Id}{Id}

\newcommand{\rW}{\mathrm{W}}
\newcommand{\rH}{\mathrm{H}}
\newcommand{\rL}{\mathrm{L}}

\newcommand{\eps}{\varepsilon}

\renewcommand{\epsilon}{\varepsilon}

\newcommand{\dbullet}{\partial^{\bullet}}

\newcommand{\GammaX}{\Gamma[X]}
\newcommand{\nablaGX}{\nabla_{\GammaX}}
\newcommand{\intGX}{\int_{\GammaX}}
\newcommand{\bx}{{\mathbf x}}
\newcommand{\by}{{\mathbf y}}

\newcommand{\bn}{{\mathbf n}}

\newcommand{\bu}{{\mathbf u}}
\newcommand{\bw}{{\mathbf w}}
\newcommand{\bz}{{\mathbf z}}
\newcommand{\bg}{{\mathbf g}}
\newcommand{\bM}{{\mathbf M}}
\newcommand{\bA}{{\mathbf A}}
\newcommand{\bK}{{\mathbf K}}

\newcommand{\Gammahx}{\Gamma_h[\bx]}

\newcommand{\Gammahtheta}{\Gamma_h^\theta}
\newcommand{\nablaGhx}{\nabla_{\Gamma_h[\bx]}}

\newcommand{\intGhx}{\int_{\Gamma_h[\bx]}}

\newcommand{\intGhtheta}{\int_{\Gamma_h^\theta}}
\newcommand{\ddtheta}{\frac{\d}{\d\theta}}

\newcommand\bfd{{\mathbf d}}
\newcommand\bfe{{\mathbf e}}
\newcommand\bff{{\mathbf f}}
\newcommand\bfg{{\mathbf g}}

\newcommand\bfn{{\mathbf n}}
\newcommand\bfr{{\mathbf r}}

\newcommand\bfu{{\mathbf u}}
\newcommand\bfv{{\mathbf v}}
\newcommand\bfw{{\mathbf w}}
\newcommand\bfx{{\mathbf x}}
\newcommand\bfy{{\mathbf y}}
\newcommand\bfz{{\mathbf z}}
\newcommand\bfA{{\mathbf A}}

\newcommand\bfH{{\mathbf H}}

\newcommand\bfK{{\mathbf K}}
\newcommand\bfM{{\mathbf M}}

\newcommand\bfV{{\mathbf V}}
\newcommand\bfW{{\mathbf W}}

\newcommand\qand{\quad\hbox{ and }\quad}

\renewcommand{\d}{\text{d}}

\newcommand{\Ga}{\Gamma}

\newcommand{\mat}{\partial^{\bullet}}

\newcommand{\diff}{\frac{\d}{\d t}}

\newcommand{\inv}{^{-1}}
\newcommand{\la}{\langle}

\newcommand{\nb}{\nabla}

\newcommand{\pa}{\partial}
\newcommand{\ra}{\rangle}

\newcommand{\spn}{\textnormal{span}}
\newcommand{\st}{such that}

\def \t {(t)}

\def \to {\rightarrow}
\newcommand{\vphi}{\varphi}

\def \nu {\textnormal{n}}
\newcommand{\phiv}{\varphi^v}
\newcommand{\phin}{\varphi^\n}
\newcommand{\phiV}{\varphi^V}

\newcommand{\Co}{C_0}

\newcommand{\wtu}{\widetilde{\bfu}}

\newcommand{\wtx}{\widetilde{\bfx}}
\newcommand{\us}{\bfu^\ast}
\newcommand{\vs}{\bfv^\ast}
\newcommand{\xs}{\bfx^\ast}

\newcommand{\uls}{\bfu_\ast}
\newcommand{\vls}{\bfv_\ast}
\newcommand{\xls}{\bfx_\ast}
\newcommand{\wtuls}{\widetilde{\bfu}_\ast}

\newcommand{\wtxls}{\widetilde{\bfx}_\ast}

\newcommand{\eu}{\bfe_\bfu}
\newcommand{\ev}{\bfe_\bfv}
\newcommand{\ex}{\bfe_\bfx}
\newcommand{\wteu}{{\bfe}_{\widetilde\bfu}}

\newcommand{\wtex}{{\bfe}_{\widetilde\bfx}}
\newcommand{\doteu}{\dot\bfe_\bfu}

\newcommand{\dotex}{\dot\bfe_\bfx}
\newcommand{\du}{\bfd_\bfu}
\newcommand{\dv}{\bfd_\bfv}

\newcommand{\dx}{\bfd_\bfx}
\newcommand{\ru}{\bfr_\bfu}

\newcommand{\n}{\nu}

\newcommand{\dof}{N}

\newcommand{\btb}{\color{black}}
\newcommand{\etb}{\color{black}}

\newcommand{\bbk}{\color{black}}
\newcommand{\ebk}{\color{black}}

\newcommand{\be}{\partial^\tau\!}

\newcommand{\Mxu}[1][n]{\bfM(\wtx^{#1},\wtu^{#1})}

\begin{document}

\title{A convergent finite element algorithm \\ for generalized mean curvature flows of closed surfaces}
\shorttitle{A convergent finite element algorithm for generalized mean curvature flows}

\author{%
	{\sc Tim Binz\thanks{Email: binz@mathematik.tu-darmstadt.de}} \\[2pt]
	Fachbereich Mathematik, Technische Universität Darmstadt, \\ 
	Schlossgartenstrasse 7, 64289 Darmstadt, Germany \\[6pt]
	{\sc and}\\[6pt]
	{\sc Bal\'{a}zs Kov\'{a}cs\thanks{Corresponding author. Email: kovacs@mathematik.uni-regensburg.de}} \\[2pt]
	Faculty of Mathematics, University of Regensburg,\\ 
	93040 Regensburg, Germany
}
\shortauthorlist{T.~Binz and B.~Kov\'{a}cs}

\maketitle

\begin{abstract}
	{An algorithm is proposed for generalized mean curvature flow of closed two-dimensional surfaces, which include inverse mean curvature flow, powers of mean and inverse mean curvature flow, etc. Error estimates are proven for semi- and full discretisations for the generalized flow.
	The algorithm proposed and studied here combines evolving surface finite elements, whose nodes determine the discrete surface, and linearly implicit backward difference formulae for time integration. The numerical method is based on a system coupling the surface evolution to non-linear second-order parabolic evolution equations for the normal velocity and normal vector.
	Convergence proof is presented in the case of finite elements of polynomial degree at least two and backward difference formulae of orders two to five. The error analysis combines stability estimates and consistency estimates to yield optimal-order $H^1$-norm error bounds for the computed surface position, velocity,
	normal vector, normal velocity, and therefore for the mean curvature. The stability analysis is performed in the matrix--vector formulation, and is independent of geometric arguments, which only enter the consistency analysis.
	Numerical experiments are presented to illustrate the convergence results, and also to report on monotone quantities, e.g.~Hawking mass for inverse mean curvature flow. Complemented by experiments for non-convex surfaces.}
{generalized mean curvature flow, inverse mean curvature flow, $H^\alpha$-flow, optimal-order convergence, evolving surface finite elements, linearly implicit BDF methods, energy estimates, monotone quantities}
\end{abstract}

\section{Introduction}

In this paper we propose and prove convergence of a numerical method for the evolution of a two-dimensional closed surface $\Ga\t$ evolving under \emph{generalized mean curvature flow}. The velocity of the surface $\Ga\t$ is given by the velocity law:
\begin{equation}
\label{intro:velocity law}
v = - V(H) \nu_{\Ga\t} ,
\end{equation}
here is $H$ the mean curvature of the surface $\Ga\t$, $\nu_{\Ga\t}$ denotes the outward unit normal vector, (using the convention that the mean curvature of a sphere is positive), and $V$ is a given function. 

\pagebreak

Many notable geometric flows fit into this framework, in particular
\begin{center}
	\emph{mean curvature flow, see \cite{Huisken1984}, \\ 
		inverse mean curvature flow, see \cite{HuiskenPolden,HuiskenIlmanen}, \\
		powers of mean curvature flow, see \cite{Schulze_1,Schulze_2,Schulze_3,Schulze_diss}, \\ 
		and powers of inverse mean curvature flow, see \cite{Gerhardt_pIMCF,Scheuer_pIMCF}, \\ 
		as well as a logarithmic mean curvature flow, see \cite{AlessandroniSinestrari_nhMCF,Espin_nhMCF}, \\ 
		etc.}
\end{center} 
There are a few papers which prove theoretical results for the general flow \eqref{intro:velocity law}, see, e.g.~\cite{HuiskenPolden,AlessandroniSinestrari_nhMCF,Espin_nhMCF}.

\smallskip

The solution of these flows, and their properties as well, are of theoretical and modelling interest: For example, to prove interesting geometric inequalities, most notably, the weak solvability theory of inverse mean curvature flow was used to prove: the positive mass conjecture \cite{SchoenYau}, and the Riemannian Penrose inequality from general relativity \cite{HuiskenIlmanen}. On the modelling side, mean curvature flow is used for various purposes, see the references in \cite{MCF}, various types of inverse mean curvature flows are utilised in image processing \cite[equation (23) and Section~8]{AlvarezGuichardLionsMorel}, \cite[equation (15)]{AngenentSapiroTannenbaum}, and \cite{SapiroTannenbaum}, see as well \cite[]{MalladiSethian}.

A number of numerical methods have been proposed for the above flows. 
For \emph{surfaces} finite volume algorithms were introduced by \cite{diss_Pasch} for mean curvature and inverse mean curvature flow.  
A surface finite element based algorithm for the general flow \eqref{intro:velocity law} was proposed by Barrett, Garcke, and N\"urnberg in \cite{BGN2008}. For \emph{curves and networks} evolving according the general flow they had also proposed a finite element algorithm in \cite{BGN2007}. 
For \emph{inverse mean curvature flow} many algorithms have been proposed, which use an equivalent formulation based on a \emph{non-linear singular elliptic equation on an unbounded domain} derived by \cite{HuiskenIlmanen}. This problem represents a level-set formulation for the inverse mean curvature
flow. A finite element method for a regularized flow, based on the regularisation of the singular elliptic problem, was proposed and analysed by Feng, Neilan, and Prohl in \cite{FengNeilanProhl}. They also proved error estimates for the regularized problem.
Using the same approach, level-set finite element method was proposed for \emph{powers of mean curvature flow} by \cite{Kroener_levelsetFEM_pMCF_arxiv}, the method was analysed by Kr\"oner, Kr\"oner, and Kr\"oner in \cite{KroenerKroenerKroener_levelsetFEM_pMCF}, while convergence rates for the semi-discretisation were proven by \cite{Kroener_levelset_pMCF}, similarly, for \emph{inverse mean curvature flow} by \cite{Kroener_iMCf_pMCF}. 

However, to our knowledge, no convergence results have been proved for evolving surface finite element algorithms for any of the above generalized mean curvature flows. Also, to our knowledge, no convergence results are available for any algorithm for the \emph{generalized} mean curvature flow \eqref{intro:velocity law}.

The main goals of the present paper are:
\begin{center}
	\emph{To propose a finite element algorithm for the generalized mean curvature flow \eqref{intro:velocity law} of closed two-dimensional surfaces, and to prove optimal-order error estimates for the proposed algorithm, under minor conditions on the function $V$.}
\end{center} 

To achieve these goals, the key idea is to derive \emph{non-linear} parabolic evolution equations for the normal velocity $V = V(H)$ and the surface normal $\nu$ along the generalized mean curvature flow, \bbk under natural assumptions on $V$. \ebk In recent previous works for mean curvature flow \cite{MCF} and Willmore flow \cite{Willmore}, which motivate this approach, the analogous evolution equations for $H$ and $\nu$ were used. 
In \cite{HuiskenPolden,HuiskenIlmanen,HuiskenIlmanan_highreg} the authors have already derived evolution equations for the mean curvature $H$ and the surface normal $\nu$. The evolution equation for the normal velocity $V$ is first derived here, and until the present work it was not evident that the evolution equations for $V$ and $\nu$ form a closed system that does not involve further geometric quantities.

A new non-linearity appears with the time derivatives in the evolution equations for $V$ and $\nu$, but otherwise they are very similar to the evolution equations for $H$ and $\nu$ for mean curvature flow \cite{MCF}. This structural similarity enables us to use many results from \cite{MCF}, but there are substantial parts that require a careful and extended analysis.

The system coupling the non-linear evolution equations for the geometric variables, the velocity law \eqref{intro:velocity law}, and an ordinary differential equation (ODE) for the surface evolution is  discretized using evolving surface finite elements (of degree at least 2) and using linearly implicit backward difference formulae (of order 2 to 5), \bbk under a mild step size restriction. \ebk 

We will prove optimal-order $H^1$-norm semi- and fully discrete error estimates for the surface position $X$ and all variables $v, \nu, V$ (and hence also for $H$). The fully discrete convergence proof clearly separates the issues of stability and consistency, and holds under sufficient regularity assumptions on the solution of generalized mean curvature flow, which excludes the formation of singularities.

For proving convergence of the full discretisation, the main issue is to prove \emph{stability}, that is to bound the errors in terms of consistency defects and errors in the initial values. The stability proof (as for \cite{MCF} and \cite{Willmore}) is performed in the matrix--vector formulation, where the similarity of the coupled system for generalized mean curvature flow and mean curvature flow \cite{MCF} will become more apparent. The main difference -- and difficulty -- is estimating the solution-dependent mass matrix terms. A key step in the stability proof is to establish uniform-in-time $W^{1,\infty}$-norm error bounds for all variables, shown via $H^1$-norm error bounds using inverse estimates. Additionally, in order to estimate the terms with the non-linear mass matrix, these $W^{1,\infty}$-bounds are used to prove $h$-uniform upper and lower bounds for the approximation of the mean curvature and some related variables. 
Due to the mentioned structural similarity, most of the stability proof uses the same techniques as \cite[Proposition~10.1]{MCF}, it is based on energy estimates testing with the time-derivative of the errors, via \cite{Dahlquist} and \cite{NevanlinnaOdeh}. The terms involving the solution-dependent mass matrix are estimated using a new technical lemma and a solution-dependent norm equivalence result based on the mentioned upper and lower bounds. 

Consistency estimates, i.e.~bounding the defects occurring upon inserting appropriate projections of the exact solution into the method, are analogous to the same result for mean curvature flow \cite[Lemma~8.1]{MCF}, and we mainly focus on the differences due to the non-linearity.

The paper is organized as follows.
Section~\ref{section:generalized MCF} introduces some basic notations and geometric concepts and is mainly devoted to deriving the evolution equations for the normal velocity and the normal vector along the generalized mean curvature flow. The coupled non-linear system and its weak formulation, which serves as the basis of the algorithm, is presented here. 
Section~\ref{section:ESFEM} describes the evolving surface finite element semi-discretization, the matrix--vector formulation. We also discuss here the similarity of the matrix--vector formulation to that of mean curvature flow.
Section~\ref{section:BDF} describes the linearly implicit time discretisation.
Section~\ref{section:main results} states the main results of the paper: optimal-order semi- and fully discrete error bounds in the $H^1$-norm for the errors in all variables. 
In Section~\ref{section:stability-full} we prove the fully discrete stability result after presenting the required auxiliary results in Section~\ref{section:relating surfaces}. Section~\ref{section:consistency - full} contains the consistency analysis. In Section~\ref{section:fully discrete proof completed} we combine the results of the previous two sections to prove the fully discrete convergence theorem.
Section~\ref{section:numerics} presents numerical experiments illustrating and complementing our theoretical results: reporting on convergence tests, on numerical solutions for various flows also with non-convex initial surfaces, and on the behaviour of monotone quantities, e.g.~Hawking mass.

\section{Evolution equations for generalized mean curvature flow}
\label{section:generalized MCF}

\subsection{Basic notions and notation}
\label{subsection:basic notions}

We start by introducing some basic concepts and notations, taking this description almost verbatim from \cite{MCF}. 

We consider the evolving two-dimensional closed surface $\Gamma[X]\subset\R^3$ as the image
$$
	\Ga[X] = \Ga[X(\cdot,t)] = \{ X(p,t) \,:\, p \in \Ga^0 \},
$$
of a smooth mapping $X:\Ga^0\times [0,T]\to \R^3$ such that $X(\cdot,t)$ is an embedding for every $t$. Here, $\Ga^0$ is a smooth closed initial surface, and $X(p,0)=p$. 
In view of the subsequent numerical discretization, it is convenient to think of $X(p,t)$ as the position at time $t$ of a moving particle with label $p$, and of $\Ga[X] $ as a collection of such particles. 

The {\it velocity} $v(x,t)\in\R^3$ at a point $x=X(p,t)\in\Ga[X(\cdot,t)]$  equals
\begin{equation}
\label{eq:velocity ODE}
	\partial_t X(p,t)= v(X(p,t),t).
\end{equation}
For a known velocity field  $v$,  the position $X(p,t)$ at time $t$ of the particle with label $p$ is obtained by solving the ordinary differential equation \eqref{eq:velocity ODE} from $0$ to $t$ for a fixed $p$.

For a function $u(x,t)$ ($x\in \Gamma[X]$, $0\le t \le T$) we denote the {\it material derivative} (with respect to the parametrization $X$) as
$$
	\mat u(x,t) = \frac \d{\d t} \,u(X(p,t),t) \quad\hbox{ for } \ x=X(p,t).
$$

On any regular surface $\Gamma\subset\R^3$, we denote by $\nabla_{\Ga}u:\Gamma\to\R^3$ the  {\it tangential gradient} of a function $u:\Gamma\to\R$, and in the case of a vector-valued function $u=(u_1,u_2,u_3)^T:\Gamma\to\R^3$, we let
$\nabla_{\Ga}u=
(\nabla_{\Ga}u_1,
\nabla_{\Ga}u_2,
\nabla_{\Ga}u_3)$. We thus use the convention that the gradient of $u$ has the gradient of the components as column vectors. 
We denote by $\nabla_{\Ga} \cdot f$ the {\it surface divergence} of a vector field $f$ on $\Gamma$, 
and by 
$\varDelta_{\Ga} u=\nabla_{\Ga}\cdot \nabla_{\Ga}u$ the {\it Laplace--Beltrami operator} applied to $u$; see the review \cite{DeckelnickDE2005} or \cite[Appendix~A]{Ecker2012} or any textbook on differential geometry for these notions. 

We denote the unit outer normal vector field to $\Gamma$ by $\n:\Gamma\to\R^3$. Its surface gradient contains the (extrinsic) curvature data of the surface $\Gamma$. At every $x\in\Gamma$, the matrix of the extended Weingarten map,
$$
	A(x)=\nabla_\Gamma \n(x),
$$ 
is a symmetric $3\times 3$ matrix (see, e.g., \cite[Proposition~20]{Walker2015}). Apart from the eigenvalue $0$ (with eigenvector $\n$), its other two eigenvalues are the principal curvatures $\kappa_1$ and $\kappa_2$. They determine the fundamental quantities
\begin{align}\label{Def-H-A2}
	H:={\rm tr}(A)=\kappa_1+\kappa_2, \qquad |A|^2 = \kappa_1^2 +\kappa_2^2 ,
\end{align}
where $|A|$ denotes the Frobenius norm of the matrix $A$.
Here, $H$ is called the {\it mean curvature} (as in most of the literature, we do not put a factor 1/2).

\subsection{Evolution equations for normal vector and normal velocity of a surface evolving under generalized mean curvature flow}

The velocity law of the \emph{generalized} mean curvature flow is given by
\begin{equation}
\label{eq:velocity law}
	v = - V(H) \nu ,
\end{equation}
where $V = V(H)$ denotes the normal velocity of the surface depending on the mean curvature $H$.

It is important to observe that \eqref{eq:velocity law} includes many classical surface flows (with non-exhaustive reference lists):
\begin{alignat*}{3}
	V(H) = &\ H , & \quad & \text{mean curvature flow, see \cite{Huisken1984}} , \\
	V(H) = &\ - \frac{1}{H} , & \quad & \text{inverse mean curvature flow,} \\[-3mm]
	&&&\text{see \cite{HuiskenPolden,HuiskenIlmanen}} , \\
	V(H) = &\ H^\alpha , & \quad & \text{powers of mean curvature ($\alpha > 0$), see  \cite{Schulze_diss,Schulze_1,Schulze_2,Schulze_3}} , \\
	V(H) = &\ -\frac{1}{H^\alpha} , & \quad & \text{powers of inverse mean curvature ($\alpha > 0$),} \\[-3mm]
	&&&\text{see  \cite{Gerhardt_pIMCF,Scheuer_pIMCF}} , \\
	V(H) = &\ \frac{H + \widetilde{H}}{\log{(H + \widetilde{H})}} , & \quad & \text{a non-homogeneous mean curvature flow ($\widetilde{H} > 0$),} \\[-4.5mm]
	&&&\text{see \cite{AlessandroniSinestrari_nhMCF,Espin_nhMCF}} , \\
	\text{etc.}&&&
\end{alignat*}


Throughout the paper we will assume that along the flow the mean curvature satisfies, for $0 \leq t \leq T$ and $x \in \Ga[X(\cdot,t)]$,
\begin{equation}
\label{eq:bounds on mean curvature}
	0 < H_0 \leq H(x,t) \leq H_1 \qquad \text{with constants} \quad H_0, H_1 > 0 .
\end{equation}

We assume that, \bbk for an interval $I$ (depending on the problem), \ebk 
\begin{equation}
\label{eq:V assumptions}
	\text{\bbk $V  : \R_+ \to I$ is a smooth and strictly monotone increasing bijection, \ebk }
\end{equation}
hence it has -- in particular -- the properties:
\begin{equation}
\label{eq:properties of V}
	\begin{aligned}
		& \, V \text{ is invertible such that its inverse } V^{-1}:I \to \R_+ \text{ is (locally) Lipschitz, and} \\
		& \, \text{its derivative } V' \text{ is everywhere positive and locally Lipschitz} . 	
	\end{aligned}
\end{equation}
\bbk We note here that the interval $I$ should suitably include the image of $[H_0,H_1]$, e.g.~$I = [V(H_0/c),V(c\,H_1)]$ with a suitable factor $c > 1$. \ebk 
Similar assumptions were made, e.g., in \cite{Espin_nhMCF}.

\begin{remark}
	\label{remark:strict parabolicity}
	We note here that a positive lower bound on the mean curvature of the initial surface is usually ensured by assuming its strict convexity (i.e.~the principal curvatures $\kappa_i$ are all positive). In general such an assumption is not restrictive, since assuming strict convexity of the initial surface is necessary to show the parabolicity of the evolution equations and, hence, existence results for the exemplary generalized mean curvature flows above. This property is conserved along the flow for these problems. See, e.g., \cite[Theorem~3.1]{Huisken1984} for mean curvature flow, \cite[Theorem~3.1]{HuiskenPolden} for inverse mean curvature flow, \cite[Theorem~1.1]{Schulze_diss}, \cite[Theorem~1.1]{Schulze_1} for the $H^\alpha$-flow, \cite[Theorem~1.2]{Gerhardt_pIMCF} for the $H^{-\alpha}$-flow, and \cite[Theorem~1]{AlessandroniSinestrari_nhMCF} for the non-homogeneous flow. (In some of these theorems only weak convexity ($\kappa_i \geq 0$) is needed.) 
	It is worth to note here that the role of the (sufficiently large) constant $\widetilde{H}$ is exactly to ensure that the function $(H + \widetilde{H}) / \log{(H + \widetilde{H})}$ preserves this property of the flow, in contrast to the flow with $H/\log(H)$, cf.~\cite{AlessandroniSinestrari_nhMCF}.
\end{remark}

The normal velocity $V = V(H)$ and the normal vector $\nu$ from \eqref{eq:velocity law} satisfy the following non-linear evolution equations along the generalized mean curvature flow \eqref{eq:velocity law}.


\begin{lemma}
\label{lemma:evolution equations}
	Assume that the function $V$ satisfies \eqref{eq:V assumptions}. 
	For a regular surface $\Ga[X]$ moving under \emph{generalized} mean curvature flow, satisfying \eqref{eq:bounds on mean curvature} , the normal vector $\nu$ and the normal velocity $V:=V(H)$ satisfy the following non-linear strictly parabolic evolution equations:
	\begin{align}
		\label{eq:evolution eq - normal}
		\frac{1}{V'(H)} \, \mat \nu &= \Delta_{\Gamma[X]} \nu + |A|^2 \nu, \\
		\label{eq:evolution eq - V}
		\frac{1}{V'(H)} \, \mat V &= \Delta_{\Gamma[X]} V  + |A|^2 V .
	\end{align} 
\end{lemma}
\begin{proof}
	By using the normal velocity in the proof of \cite[Lemma~3.3]{Huisken1984}, or see also \cite{Ecker2012}, \cite[Lemma~2.37]{BGN_survey}, the following evolution equation for the normal vector holds: 
	$$
	\mat \n = \nb_{\Ga[X]} \big( V(H) \big) =  V'(H) \nb_{\Ga[X]} H,  
	$$
	by the velocity law \eqref{eq:velocity law} and using the chain rule.
	On any surface $\Gamma$, it holds true that (see \cite[(A.9)]{Ecker2012} or \cite[Proposition~24]{Walker2015}):
	$$
	\nabla_{\Ga[X]} H = \Delta_{\Ga[X]} \n + |A|^2 \n .
	$$
	This, in combination with the previous equation and noting that, by \eqref {eq:properties of V}, via the monotonicity of $V$ it follows that $V'(H) > 0$. Dividing both sides by $V'(H)$ then gives the stated evolution equation for $\nu$.
	
	By revising the proof of \cite[Theorem~3.4 and Corollary~3.5]{Huisken1984}, or see \cite{Ecker2012}, \cite[Lemma~2.39]{BGN_survey}, with the normal velocity $V$ we obtain 
	\begin{equation*}
	\mat H = \Delta_{\Ga[X]} \big( V(H) \big) + |A|^2 V(H) ,
	\end{equation*}
	which, again by the chain rule for $\mat V = \mat\big(V(H)\big)$, and dividing by $V'(H) > 0$ again, finally yields the evolution equation for $V=V(H)$.
\end{proof}


It is instructive to relate the evolution equations for \emph{generalized} mean curvature flow, Lemma~\ref{lemma:evolution equations}, with those for \emph{standard} mean curvature flow, \cite{Huisken1984}, or \cite[Lemma~2.1]{MCF}. In particular, we point out that the right-hand sides of \eqref{eq:evolution eq - normal} and \eqref{eq:evolution eq - V} are formally the same (with $V$ instead of $H$) as the right-hand sides of \cite[equations~(2.4) and (2.5)]{MCF}:
\begin{align*}
\mat \nu &= \Delta_{\Gamma[X]} \nu + |A|^2 \nu, \\
\mat H &= \Delta_{\Gamma[X]} H  + |A|^2 H .
\end{align*}
The only differences are the non-linear factors on the left-hand sides. This structural similarity already suggests that the approach and numerical analysis presented in \cite{MCF} can be extended to the generalized mean curvature flow, but will require modifications treating the terms involving the non-linear weights in front of the material derivatives.

In this paper, we will address these modifications, and extend the stability and convergence analysis of \cite{MCF} for the generalized mean curvature flow \eqref{eq:velocity law}.

\subsection{The evolution equation system for generalized mean curvature flow}

The evolution of a surface under generalized mean curvature flow is then governed by the coupled system \eqref{eq:velocity law}, \eqref{eq:evolution eq - normal}--\eqref{eq:evolution eq - V} together with the ODE \eqref{eq:velocity ODE}. 
The numerical method is based on the weak form of the above coupled system which reads, denoting $|A|^2 = |\nablaGX \nu|^2$ and $V' = V'(H)$ (with $H$ obtained by inverting $V=V(H)$):
\begin{subequations}
	\label{evolutioneqs-weak}
	\begin{align}
	\label{evolutioneqs-weak a}
	\intGX \nablaGX v \cdot \nablaGX \vphi^v + \intGX v \cdot \vphi^v  
	&= - \intGX  \nablaGX (V\nu) \cdot \nablaGX \vphi^v - \intGX V \nu \cdot \vphi^v , \\
	\intGX \frac{1}{V'} \, \dbullet \nu \cdot \vphi^\nu 
	&= - \intGX \nablaGX \nu \cdot \nablaGX \vphi^\nu + \intGX |A|^2 \nu \cdot \vphi^\nu , \\
	\intGX \frac{1}{V'} \, \dbullet V \vphi^V 
	&= - \intGX \nablaGX V \cdot \nablaGX \vphi^V  
	+ \intGX |A|^2 V \vphi^V,
	\end{align}
\end{subequations}
for all test functions $\phiv \in H^1(\Ga[X])^3$ and $\phin \in H^1(\Ga[X])^3$, $\phiV \in H^1(\Ga[X])$, together with the ODE for the positions \eqref{eq:velocity ODE}. 
This system is complemented with the initial data $X^0$, $\n^0$ and $V^0 = V(H^0)$.

For simplicity, by $\, \cdot \,$ we denote both the Euclidean scalar product for vectors, and the Frobenius inner product for matrices (i.e., the Euclidean product with an arbitrary vectorisation).

\section{Evolving finite element semi-discretization}
\label{section:ESFEM}


\subsection{Evolving surface finite elements}
We formulate the evolving surface finite element (ESFEM) discretization for the velocity law coupled with evolution equations on the evolving surface, following (almost verbatim) the description in \cite{KLLP2017,MCF}, which is based on \cite{Dziuk88,DziukElliott_ESFEM,Demlow2009,Kovacs2018}. We use simplicial finite elements and continuous piecewise polynomial basis functions of degree~$k$, as defined in \cite[Section 2.5]{Demlow2009}.

We triangulate the given smooth initial surface $\Gamma^0$ by an admissible family of triangulations $\mathcal{T}_h$ of decreasing maximal element diameter $h$; see \cite{DziukElliott_ESFEM} for the notion of an admissible triangulation, which includes quasi-uniformity and shape regularity. For a momentarily fixed $h$, we denote by $\bfx^0 $  the vector in $\R^{3\dof}$ that collects all nodes $p_j$ $(j=1,\dots,\dof)$ of the initial triangulation. By piecewise polynomial interpolation of degree $k$, the nodal vector defines an approximate surface $\Gamma_h^0$ that interpolates $\Gamma^0$ in the nodes $p_j$. We will evolve the $j$th node in time, denoted $x_j(t)$ with $x_j(0)=p_j$, and collect the nodes at time $t$ in a column vector
$$
\bfx(t) \in \R^{3\dof}. 
$$
We just write $\bfx$ for $\bfx(t)$ when the dependence on $t$ is not important.

By piecewise polynomial interpolation on the  plane reference triangle that corresponds to every
curved triangle of the triangulation, the nodal vector $\bfx$ defines a closed surface denoted by $\Gamma_h[\bfx]$. We can then define globally continuous finite element {\it basis functions}
$$
\phi_i[\bfx]:\Gamma_h[\bfx]\to\R, \qquad i=1,\dotsc,\dof,
$$
which have the property that on every triangle their pullback to the reference triangle is polynomial of degree $k$, and which satisfy at the nodes $\phi_i[\bfx](x_j) = \delta_{ij}$ for all $i,j = 1,  \dotsc, \dof .$
These functions span the finite element space on $\Gamma_h[\bfx]$,
\begin{equation*}
S_h[\bfx] = S_h(\Gamma_h[\bfx])=\spn\big\{ \phi_1[\bfx], \phi_2[\bfx], \dotsc, \phi_\dof[\bfx] \big\} .
\end{equation*}
For a finite element function $u_h\in S_h[\bfx]$, the tangential gradient $\nabla_{\Gamma_h[\bfx]}u_h$ is defined piecewise on each element.

The discrete surface at time $t$ is parametrized by the initial discrete surface via the map $X_h(\cdot,t):\Gamma_h^0\to\Gamma_h[\bfx(t)]$ defined by
$$
X_h(p_h,t) = \sum_{j=1}^\dof x_j(t) \, \phi_j[\bfx(0)](p_h), \qquad p_h \in \Gamma_h^0,
$$
which has the properties that $X_h(p_j,t)=x_j(t)$ for $j=1,\dots,\dof$, that  $X_h(p_h,0) = p_h$ for all $p_h\in\Gamma_h^0$, and
$$
\Gamma_h[\bfx(t)]=\Gamma[X_h(\cdot,t)] = \{ X_h(p_h,t) \mid p_h \in \Ga_h^0 \}. 
$$

The {\it discrete velocity} $v_h(x,t) \in \R^3$ at a point $x=X_h(p_h,t) \in \Gamma[X_h(\cdot,t)]$ is given by
$$
\partial_t X_h(p_h,t) = v_h(X_h(p_h,t),t).
$$
In view of the transport property of the basis functions  \cite{DziukElliott_ESFEM},
$
\tfrac\d{\d t} \big( \phi_j[\bfx(t)](X_h(p_h,t)) \big) =0 ,
$
the discrete velocity equals, for $x \in \Gamma_h[\bfx(t)]$,
$$
v_h(x,t) = \sum_{j=1}^\dof v_j(t) \, \phi_j[\bfx(t)](x) \qquad \hbox{with } \ v_j(t)=\dot x_j(t),
$$
where the dot denotes the time derivative $\d/\d t$. 
Hence, the discrete velocity $v_h(\cdot,t)$ is in the finite element space $S_h[\bfx(t)]$, with nodal vector $\bfv(t)=\dot\bfx(t)$.
%
%

The {\it discrete material derivative} of a finite element function $u_h(x,t)$ with nodal values $u_j(t)$ is
$$
\mat_h u_h(x,t) = \frac{\d}{\d t} u_h(X_h(p_h,t)) = \sum_{j=1}^\dof \dot u_j(t)  \phi_j[\bfx(t)](x)  \quad\text{at}\quad x=X_h(p_h,t).
$$


\subsection{ESFEM spatial semi-discretizationss}
\label{subsection:semi-discretization}

Now we will describe the semi-discretization of the coupled system for generalized mean curvature flow.

The finite element spatial semi-discretization of the weak coupled parabolic system \eqref{evolutioneqs-weak} reads as follows: Find the unknown nodal vector $\bfx(t)\in \R^{3\dof}$ and the unknown finite element functions $v_h(\cdot,t)\in S_h[\bfx(t)]^3$ and $\nu_h(\cdot,t)\in S_h[\bfx(t)]^3$, and $V_h(\cdot,t) \in S_h[\bfx(t)]$ such that, by denoting $|A_h|^2 = | \nb_{\Ga_h[\bfx]} \n_h |^2$ and $V_h' = V'(H_h)$ (with $H_h$ obtained by inverting $V_h=V(H_h)$),
\begin{subequations}
	\label{semidiscretization}
	\begin{align}
	\label{semidiscretization: v}
		\intGhx \nablaGhx v_h \cdot \nablaGhx \vphi^v_h 
		+ \intGhx v_h \cdot \vphi^v_h 
		&= - \intGhx \nablaGhx (V_h \nu_h) \cdot \nablaGhx \vphi^v_h
		- \intGhx V_h \nu_h \cdot \vphi^v_h, \\
		\intGhx \frac{1}{V_h'} \, \dbullet_h \nu_h \cdot \vphi^\nu_h
		&= - \intGhx \nablaGhx \nu_h \cdot \nablaGhx \vphi^\nu_h 
		+ \intGhx |A_h|^2 \nu_h \cdot \vphi^\nu_h, \\
		\intGhx \frac{1}{V_h'} \, \dbullet_h V_h \vphi^V_h 
		&= - \intGhx \nablaGhx V_h \cdot \nablaGhx \vphi^V_h 
		+ \intGhx |A_h|^2 V_h \vphi^V_h .
	\end{align}
\end{subequations}
for all $\vphi^v_h \in S_h[\bfx(t)]^3$, $\vphi^\nu_h \in S_h[\bfx(t)]^3$, and $\vphi^V_h\in S_h[\bfx(t)]$, with the surface $\Gamma_h[\bfx(t)]=\Gamma[X_h(\cdot,t)] $ given by the differential equation
\begin{equation}
\label{eq:xh}
\partial_t X_h(p_h,t) = v_h(X_h(p_h,t),t), \qquad p_h\in\Ga_h^0.
\end{equation}
The initial values for the nodal vector $\bfx$ are taken as the positions of the nodes of the triangulation of the given initial surface $\Gamma^0$.
The initial data $\n_h^0$ and $V_h^0=V(H_h^0)$ are determined by Lagrange interpolation of $\n^0$ and $H^0$. 

\subsection{Matrix--vector formulation}
\label{section:matrix-vector form}

The nodal values of the unknown semi-discrete functions are collected into column vectors  $\bfv=(v_j) \in \R^{3N}$, $\bfn=(\n_j) \in \R^{3N}$, and $\bfV=(V_j)\in\R^N$. We furthermore collect
\begin{equation*}
	\bu := \begin{pmatrix} \bn \\ \bfV \end{pmatrix} \in \R^{4N} .
\end{equation*}

We define the surface-dependent mass matrix $\bfM(\bfx)$ and stiffness matrix $\bfA(\bfx)$, as well as the solution-dependent mass matrix $\bfM(\bfx,\bfu)$:
\begin{align*}
	\bfM(\bx)|_{ij} = \intGhx \phi_i[\bfx] \phi_j[\bfx] \qand \bA(\bx)|_{ij} = \intGhx \nablaGhx \phi_i[\bfx] \cdot \nablaGhx \phi_j[\bfx] , 
\end{align*}
\begin{equation}
\label{eq:solution-dependent mass matrix}
	\bM(\bx,\bfu)|_{ij} = \intGhx \frac{1}{V_h'} \, \phi_i[\bfx] \phi_j[\bfx] , 
\end{equation}
for $i,j = 1,  \dotsc,N$.
The non-linear terms $\bff(\bx,\bu) = \big(\bff_1(\bx,\bu) , \bff_2(\bx,\bu)\big)^T$ and $\bg(\bx,\bu)$ are defined by
\begin{align*}
	\bff_1(\bx,\bu)|_{j+(l-1)N} &= \intGhx |A_h|^2 \, (\nu_h)_l \phi_j[\bfx] , \\
	\bff_2(\bx,\bu)|_{j} &= \intGhx |A_h|^2 V_h \phi_j[\bfx] , \\
	\bfg(\bx,\bu)|_{j + (l-1)N} &= - \intGhx V_h \, (\nu_h)_l \phi_j[\bfx]
	- \intGhx \nablaGhx (V_h \, (\nu_h)_l) \cdot \nablaGhx \phi_j[\bfx]  ,
\end{align*}
for $j = 1, \dotsc, N$ and $\ell=1,2,3$. We recall that $|A_h|^2 = |\nablaGhx \nu_h|^2$ and $V_h' = V'(H_h)$, with $H_h$ obtained by inverting $V_h=V(H_h)$. 

We further let, for $d \in \N$ (with the identity matrices $I_d \in \R^{d \times d}$) 
$$
	\bfM^{[d]}(\bfx)= I_d \otimes \bfM(\bfx), \qquad
	\bfA^{[d]}(\bfx)= I_d \otimes \bfA(\bfx), \qquad
	\bfK^{[d]}(\bfx) = I_d \otimes \bigl( \bfM(\bfx) + \bfA(\bfx) \bigr) ,
$$
and similarly for $\bfM(\bfx,\bfu)$. 
When no confusion can arise, we will write $\bfM(\bfx)$ for $\bfM^{[d]}(\bfx)$, $\bfM(\bfx,\bfu)$ for $\bfM^{[d]}(\bfx,\bfu)$, $\bfA(\bfx)$ for $\bfA^{[d]}(\bfx)$, and $\bfK(\bfx)$ for $\bfK^{[d]}(\bfx)$.

Using these definitions \eqref{semidiscretization} with \eqref{eq:xh} can be written in the matrix--vector form:
\begin{subequations}
	\label{eq:matrix-vector form}
	\begin{align}
	\label{eq:matrix-vector form - velocity law}
	\bK^{[3]}(\bx) \bfv &= \bg(\bx,\bu), \\
	\bM^{[4]}(\bx,\bfu) \dot{\bu} + \bA^{[4]}(\bx) \bu &= \bff(\bx,\bu), \\
	\intertext{with \eqref{eq:xh} equivalent to}
	\dot \bfx &= \bfv .
	\end{align} 
\end{subequations}

We now compare the above matrix--vector formulation \eqref{eq:matrix-vector form} for \emph{generalized} mean curvature flow, to the same formulas for \emph{standard} mean curvature flow \cite[equation~(3.4)--(3.5)]{MCF}:
\begin{equation}
\label{eq:MCF matrix-vector form}
\begin{aligned}
\bK^{[3]}(\bx) \bfv &= \bg(\bx,\bu), \\
\bM^{[4]}(\bx) \dot{\bu} + \bA^{[4]}(\bx) \bu &= \bff(\bx,\bu), \\
\text{with} \qquad 
\dot \bfx &= \bfv .
\end{aligned} 
\end{equation}
The two formulations are formally the same, the only difference is the solution-dependent mass matrix $\bfM(\bfx,\bfu)$ in the term with a time derivative of $\bfu$, which in the case of mean curvature flow is simply $\bfM(\bfx)$. 
(Also note that here $\bfu$ collects $\bfn$ and $\bfV$, whereas for mean curvature flow $\bfu = (\bfn,\bfH)^T$.) 
The stability proof presented in \cite[Section~10]{MCF} will therefore be generalized below to accommodate the use of solution-dependent mass matrices, but we will exploit the similarities of the two problems as much as possible. We will also extend here the estimates of \cite[Section~7.1]{MCF}, relating different finite element surfaces, to the solution-dependent case.

\begin{remark}
	\label{remark:pointwise discrete velocity law}
	Instead of enforcing the velocity law \eqref{eq:velocity law} via the Ritz projection \eqref{semidiscretization: v}, in \cite{Willmore} the velocity law is enforced using the nodal finite element interpolation. That is \eqref{eq:matrix-vector form - velocity law} is replaced by
	\begin{equation}
	\label{eq:pointwise discrete velocity law}
	\bfv = - \bfV \bullet \bfn , \qquad \text{with} \qquad \big(\bfV \bullet \bfn\big)_j = \bfV_j \, \bfn_j \quad (j=1,\dotsc,N) .
	\end{equation}
	
	The stability proof requires an $H^1$-norm stability for the velocity law, which is rather straightforward for the Ritz projection \cite{MCF}, while for the above pointwise velocity law it is shown in \cite[Part~(B) of Proposition~5.1]{Willmore}.
\end{remark}

\subsection{Lifts}
\label{section:lifts}


As in \cite{KLLP2017} and \cite[Section~3.4]{MCF}, we compare functions on the {\it exact surface} $\Gamma[X(\cdot,t)]$ with functions on the {\it discrete surface} $\Gamma_h[\bfx(t)]$, via functions on the {\it interpolated surface} $\Gamma_h[\xs(t)]$, where
$\xs(t)$ denotes the nodal vector collecting the grid points $x_j^*(t)=X(p_j,t)$ on the exact surface, where $p_j$ are the nodes of the discrete initial triangulation $\Ga_h^0$.

Any finite element function $w_h:\Ga_h[\bfx] \to \R$ on the discrete surface, with nodal values $w_j$, is associated with a finite element function $\widehat w_h$ on the interpolated surface $\Ga_h[\xs]$ with the exact same nodal values. 
This can be further lifted to a function on the exact surface by using the \emph{lift operator} $\,^\ell$, mapping a function on the interpolated surface $\Ga_h[\xs]$ to a function on the exact surface $\Gamma$,
\bbk via the identity, for $x \in \Ga_h[\xs]$,
\begin{equation*}
	x^\ell = x - d(x,t) \nu_{\Ga[X]}(x^\ell,t) , \qquad \text{and setting} \qquad \widehat w_h^\ell(x^\ell) = \widehat w_h(x) ,
\end{equation*}
using the distance function $d$, provided that the two surfaces \ebk are sufficiently close, see \cite{Dziuk88,DziukElliott_ESFEM,Demlow2009}.

Then the composed lift $\,^L$ maps finite element functions on the discrete surface $\Gamma_h[\bfx(t)]$ to functions on the exact surface $\Gamma[X(\cdot,t)]$ via the interpolated surface $\Gamma_h[\xs(t)]$ is denoted by 
$$
w_h^L = (\widehat w_h)^\ell.
$$

\section{Linearly implicit full discretization}
\label{section:BDF}

Similarly as for mean curvature flow \cite{MCF}, for the time discretization of the system of ordinary differential equations \eqref{eq:matrix-vector form} we use a $q$-step linearly implicit backward difference formula (BDF). For a step size $\tau>0$, and with $t_n = n \tau \leq T$, we determine the approximations to all variables $\bfx^n$ to $\bfx(t_n)$, $\bfv^n$ to $\bfv(t_n)$, and $\bfu^n$ to $\bfu(t_n)$ by the fully discrete system of \emph{linear} equations
\begin{subequations}
	\label{eq:BDF}
	\begin{align}
	\bfK(\widetilde \bfx^n) \bfv^n &= \bfg(\widetilde \bfx^n,\widetilde \bfu^n) , \\
	\bfM(\widetilde \bfx^n,\widetilde \bfu^n) \dot \bfu^n + \bfA(\widetilde \bfx^n) \bfu^n &= \bff(\widetilde \bfx^n,\widetilde \bfu^n) ,  \\
	\dot \bfx^n &=  \bfv^n,
	\end{align}
\end{subequations}
where  the discretized time derivatives are given by
\begin{equation}
\label{eq:backward differences def}
\dot \bfx^n = \frac{1}{\tau} \sum_{j=0}^q \delta_j \bfx^{n-j} , \qquad	\dot \bfu^n = \frac{1}{\tau} \sum_{j=0}^q \delta_j \bfu^{n-j} , \qquad n \geq q .
\end{equation}
and where $\widetilde \bfx^n$ and $\widetilde \bfu^n$ are extrapolated values, approximating $\bfx(t_n)$ and $\bfu(t_n)$: 
\begin{equation}
\label{eq:extrapolation def}
\widetilde \bfx^n = \sum_{j=0}^{q-1} \gamma_j \bfx^{n-1-j} , \qquad	\widetilde \bfu^n = \sum_{j=0}^{q-1} \gamma_j \bfu^{n-1-j} , \qquad n \geq q .
\end{equation}
The starting values $\bfx^i$ and $\bfu^i$ ($i=0,\dotsc,q-1$) are assumed to be given. They can be precomputed using either a lower order method with smaller step sizes, or an implicit Runge--Kutta method.

The method is determined by its coefficients, given by $\delta(\zeta)=\sum_{j=0}^q \delta_j \zeta^j=\sum_{\ell=1}^q \frac{1}{\ell}(1-\zeta)^\ell$ and $\gamma(\zeta) = \sum_{j=0}^{q-1} \gamma_j \zeta^j = (1 - (1-\zeta)^q)/\zeta$. 
The classical BDF method is known to be zero-stable for $q\leq6$ and to have order $q$; see \cite[Chapter~V]{HairerWannerII}.
This order is retained by the linearly implicit variant using the above coefficients $\gamma_j$; 
cf.~\cite{AkrivisLubich_quasilinBDF,AkrivisLiLubich_quasilinBDF}.

The analogous linearly implicit backward difference methods were used for mean curvature flow \cite{MCF}. Theorem~6.1 in \cite{MCF} proves optimal-order error bounds for the combined ESFEM--BDF full discretization
of the mean curvature flow system, for finite elements of polynomial degree $k \geq 2$ and BDF methods of order $2 \leq q \leq 5$.

We note that in the $n$th time step, the method requires solving two linear systems  with the symmetric positive definite matrices $\bfK(\widetilde \bfx^n)$ and $\frac{\delta_0}\tau \bfM(\widetilde \bfx^n) + \bfA(\widetilde \bfx^n)$.

From the vectors $\bfx^n =(x_j^n)$, $\bfv^n = (v_j^n)$, and $\bfu^n=(u_j^n)$ with $u_j^n=(\n_j^n,V_j^n)\in\R^3\times \R$ we obtain position approximations to $X(\cdot,t_n)$,  ${\rm id}_{\Gamma[X(\cdot,t_n)]}$, velocity approximations to $v(\cdot,t_n)$, and approximations to the normal vector and the normal velocity, respectively, at time $t_n$ as
\begin{equation}
\label{eq:approx}
\begin{aligned}
X_h^n(p_h) &= \sum_{j=1}^N x_j^n \, \phi_j[\bfx(0)](p_h) \quad\hbox{ for } p_h \in \Gamma_h^0, \\   
x_h^n & = {\rm id}_{\Gamma[X_h^n]} , \\
v_h^n(x)      &= \sum_{j=1}^N v_j^n \, \phi_j[\bfx^n](x)  \qquad\hbox{ for } x \in \Gamma_h[\bfx^n], \\
\n_h^n(x)      &= \sum_{j=1}^N \n_j^n \, \phi_j[\bfx^n](x)  \qquad\hbox{ for } x \in \Gamma_h[\bfx^n], \\
V_h^n(x)      &= \sum_{j=1}^N V_j^n \, \phi_j[\bfx^n](x)  \qquad\hbox{ for } x \in \Gamma_h[\bfx^n].
\end{aligned}
\end{equation}
The approximation $H_h$ of the mean curvature is similarly given by the nodal values $\bfH^n = V^{-1}(\bfV^n)$. 

In the semi-discrete case, the approximations of the same quantities are given analogously.

\section{Main results: error estimates}
\label{section:main results}

We will now formulate the main results of this paper, which provide optimal-order error bounds for the finite element semi-discretisation, for finite elements of polynomial degree $k\ge 2$, and of the full discretisation with linearly implicit BDF methods, of order $2 \leq q \leq 5$.

We denote by $\Gamma(t)=\Gamma[X(\cdot,t)]$ the exact surface and by $\Ga_h[\bfx(t)]$ the discrete surface at time $t$. We introduce the notation
$$
x_h^L(x,t) =  X_h^L(p,t) \in \Ga_h[\bfx(t)] \qquad\hbox{for}\quad x=X(p,t)\in\Gamma(t).
$$

\subsection{Convergence of the semi-discretization}

\begin{theorem}
	\label{thm:semi-discrete convergence}
	Consider the semi-discretization \eqref{semidiscretization} of the coupled generalized mean curvature flow problem \eqref{evolutioneqs-weak} with \eqref{eq:velocity ODE}, using evolving surface finite elements of polynomial degree~$k \geq 2$. Let the function $V$ satisfy \eqref{eq:properties of V}. \bbk Suppose that \ebk the flow \eqref{eq:velocity law} has a sufficiently regular solution $(X,v,\nu,V)$ on some time interval $[0,T]$, and that the flow map $X(\cdot,t):\Gamma^0\to \Gamma(t)\subset\R^3$ is non-degenerate so that $\Gamma(t)$ is a regular surface, with mean curvature $0 < H_0 \leq H(\cdot,t) \leq H_1$, on the time interval $t\in[0,T]$.
	
	Then, there exists constants $h_0 > 0$ and $C>0$ such that 
	\begin{align*}
	\| x_h^L (\cdot,t) - \Id_{\Gamma(t)} \|_{H^1(\Gamma(t))} \leq C h^k, \\
	\| v_h^L (\cdot,t) - v(\cdot,t) \|_{H^1(\Gamma(t))} \leq C h^k, \\
	\| \nu_h^L (\cdot,t) - \nu(\cdot,t) \|_{H^1(\Gamma(t))} \leq C h^k, \\
	\| V_h^L (\cdot,t) - V(\cdot,t) \|_{H^1(\Gamma(t))} \leq C h^k, 
	\end{align*}
	and, since $V$ is a smooth and invertible function of the mean curvature $H$, we also obtain
	\begin{align*}
	\| H_h^L (\cdot,t) - H(\cdot,t) \|_{H^1(\Gamma(t))} \leq C h^k ,
	\end{align*}
	for all $h < h_0$. Furthermore, we obtain
	\begin{align}
	\| X_h^\ell (\cdot,t) - X(\cdot,t) \|_{H^1(\Gamma(t))} \leq C h^k
	\end{align}
	for all $h < h_0$. 
	The constant $C$ is independent of $h$, but depends on bounds of higher derivatives of the solution $(X,v,\nu,V)$ of the generalized mean curvature flow, and on the length $T$ of the time interval.
\end{theorem}

\subsection{Convergence of the full discretization}

\begin{theorem}
	\label{theorem:fully discrete convergence}
	Consider the ESFEM--BDF full discretization \eqref{eq:BDF} of the coupled generalized mean curvature flow problem \eqref{evolutioneqs-weak} with \eqref{eq:velocity ODE}, using evolving surface finite elements of polynomial degree~$k\ge 2$ and linearly implicit BDF time discretization of order $q$ with $2\le q\le 5$. 
	Let the function $V$ satisfy \eqref{eq:properties of V}. Suppose that the generalized flow admits an exact solution $(X,v,\nu,V)$ that is sufficiently smooth on some time interval $t\in[0,T]$, and that the flow map $X(\cdot,t):\Gamma^0\to \Gamma(t)\subset\R^3$ is non-degenerate so that $\Gamma(t)$ is a regular surface, with mean curvature $0 < H_0 \leq H(\cdot,t) \leq H_1$, on the time interval $t\in[0,T]$. 
	
	Then, there exist  $h_0>0$, $\tau_0>0$, and $C>0$ such that for all mesh sizes $h \leq h_0$  and time step sizes $\tau\le\tau_0$ satisfying the step size restriction 
	\begin{equation}
	\label{stepsize-restriction}
	\tau \le \Co h 
	\end{equation}
	(where $\Co >0$ can be chosen arbitrarily), 
	the following error bounds for the lifts of the discrete position, velocity, normal vector and normal velocity hold over the exact surface: provided that the starting values are sufficiently accurate in the $H^1$-norm at time $t_i=i\tau$ for $i=0,\dots,q-1$, we have at time $t_n=n\tau\le T$
	\begin{align*}
	\|(x_h^n)^L - \Id_{\Gamma(t_n)}\|_{H^1(\Ga(t_n))^3} &\leq C(h^k+\tau^q), \\
	\|(v_h^n)^L - v(\cdot,t_n)\|_{H^1(\Ga(t_n))^3} &\leq C(h^k+\tau^q), \\ 
	\|(\nu_h^n)^L - \nu(\cdot,t_n)\|_{H^1(\Ga(t_n))^3} &\leq C(h^k+\tau^q), \\ 
	\|(V_h^n)^L - V(\cdot,t_n)\|_{H^1(\Ga(t_n))} &\leq C(h^k+\tau^q), 
	\end{align*}
	and also
	\begin{align*}
	\|(X_h^n)^l - X(\cdot,t_n)\|_{H^1(\Ga^0)^3} & \leq C(h^k+\tau^q), \qquad \text{and} \\
	\|(H_h^n)^L - H(\cdot,t_n)\|_{H^1(\Ga(t_n))} &\leq C(h^k+\tau^q), 
	\end{align*}
	where the constant $C$ is independent of $h$, $\tau$ and $n$ with $n\tau\le T$, but depends on bounds of higher derivatives of the solution $(X,v,\nu,V)$ of the generalised mean curvature flow, and on the length $T$ of the time interval, and on $\Co$.
\end{theorem}

Sufficient regularity assumptions are the following: uniformly in $t\in[0,T]$ and for $j=1,\dotsc,q+1$,
\begin{align*}
&\ X(\cdot,t)  \in  H^{k+1}(\Ga^0), \ \pa_t^{j} X(\cdot,t) \in  H^{1}(\Ga^0)  , \\
&\ v(\cdot,t)  \in H^{k+1}(\Ga(X(\cdot,t))),\ {\mat}^j v(\cdot,t) \in H^{2}(\Ga(X(\cdot,t)))  , \\
\text{for } \ u=(\nu,V) , \quad &\ u(\cdot,t) , \mat u(\cdot,t) \in  W^{k+1,\infty}(\Ga(X(\cdot,t)))^4  , \\
&\ {\mat}^j u(\cdot,t) \in  H^{2}(\Ga(X(\cdot,t)))^4  .
\end{align*}

For the starting values, sufficient approximation conditions are the following: for $i=0,\dotsc,q-1$,
\begin{align*}
\|(x_h^i)^L - \Id_{\Gamma(t_i)}\|_{H^1(\Ga(t_i))^3} &\leq C(h^k+\tau^q), \\
\|(v_h^i)^L - v(\cdot,t_i)\|_{H^1(\Ga(t_i))^3} &\leq C(h^k+\tau^q), \\ 
\|(\n_h^i)^L - \n(\cdot,t_i)\|_{H^1(\Ga(t_i))^3} &\leq C(h^k+\tau^q), \\ 
\|(V_h^i)^L - V(\cdot,t_i)\|_{H^1(\Ga(t_i))} &\leq C(h^k+\tau^q), 
\end{align*}
and in addition, for $i=1,\dotsc,q-1$,
\begin{equation*}
\tau^{1/2} \Big\|\frac{1}{\tau}\big(X_h^i - X_h^{i-1} \big)^l - \frac{1}{\tau}\big(X(\cdot,t_i) -  X(\cdot,t_{i-1}) \big) \Big\|_{H^1(\Ga^0)^3} \leq C(h^k+\tau^q) .
\end{equation*}

In view of Remark~\ref{remark:pointwise discrete velocity law}, both of the above theorems hold verbatim if the discretized velocity law is enforced using the nodal finite element interpolation, cf.~\eqref{eq:pointwise discrete velocity law}, instead of the Ritz map.

It is important to note here that, since both of the above results are shown by extending the techniques of \cite{MCF} to generalized mean curvature flow, the observations (including preservation of mesh admissibility and non-degeneration under the assumed regularity) after Theorem~4.1 and 6.1 from \cite{MCF} hold analogously to Theorem~\ref{thm:semi-discrete convergence} and \ref{theorem:fully discrete convergence} here.

\bbk 
Since the proof of this result is almost identical to the proof of the semi-discrete convergence theorem, the latter is omitted for brevity.
\ebk


\section{Relating different surfaces}
\label{section:relating surfaces}

In our previous work \cite{KLLP2017,MCF} we proved some technical results relating different finite element surfaces. Here we use the same setting, and briefly (and almost verbatim) recapitulate it below.

The finite element matrices defined in Section~\ref{section:matrix-vector form} induce discrete versions of Sobolev norms on the discrete surface $\Gamma_h[\bfx]$. For any nodal vector $\bfw \in \R^{N}$, with the corresponding finite element function $w_h \in S_h[\bfx]$, we define the following \bbk (semi)-norms: \ebk 
\begin{equation}
\label{eq:norms}
\begin{aligned}
&  \|\bfw\|_{\bfM(\bfx)}^{2} = \bfw^T \bfM(\bfx) \bfw = \|w_h\|_{L^2(\Ga_h[\bfx])}^2 , \\
&  \|\bfw\|_{\bfA(\bfx)}^{2} = \bfw^T \bfA(\bfx) \bfw = \|\nb_{\Ga_h[\bfx]} w_h\|_{L^2(\Ga_h[\bfx])}^2 , \\
&  \|\bfw\|_{\bfK(\bfx)}^{2} = \bfw^T \bfK(\bfx) \bfw = \|w_h\|_{H^1(\Ga_h[\bfx])}^2 .
\end{aligned}
\end{equation}
We also note here that the matrix $\bM(\bx,\bfu)$ also generates a solution dependent norm:
\begin{equation*}
\|\bfw\|_{\bfM(\bfx,\bfu)}^{2} = \bfw^T \bfM(\bfx,\bfu) \bfw = \intGhx (V_h')^{-1} \, |w_h|^2 ,
\end{equation*}
equivalent to $\|\cdot\|_{\bfM(\bfx)}$.

\bbk Let arbitrary nodal vectors $\bx, \bfy \in \R^{3N}$ define \ebk the discrete surfaces $\Gammahx$ and $\Ga_h[\by]$, respectively. Their difference is denoted by
\begin{equation*}
\bfe = (e_j)_{j=1}^N = \bx - \by \in \R^{3N} . 
\end{equation*}
For $\theta \in [0,1]$ we consider the intermediate surface $\Gammahtheta = \Ga_h[\by+\theta\bfe]$, and for \bbk any vectors $\bw,\bz \in \R^N$ we consider the corresponding \ebk finite element functions on $\Ga_h^\theta$:
\begin{align*}
e_h^\theta = \sum_{j = 1}^N e_j \phi_j[\by+\theta\bfe] , \quad 
w_h^\theta = \sum_{j = 1}^N w_j \phi_j[\by+\theta\bfe] , \quad \text{and} \quad
z_h^\theta = \sum_{j = 1}^N z_j \phi_j[\by+\theta\bfe] .
\end{align*}
Figure~\ref{figure:relating different surfaces} illustrates the described construction.
\begin{figure}[htbp]
	\begin{center}
		\includegraphics[scale=1]{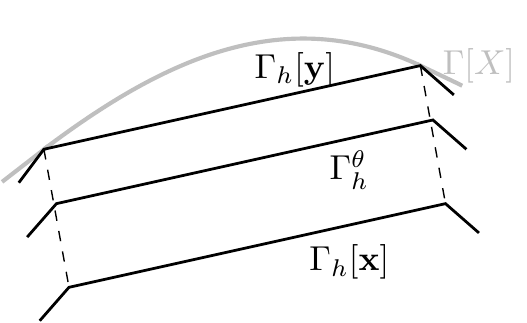}
		\caption{The construction of the intermediate surfaces $\Gamma_h^\theta$}
		\label{figure:relating different surfaces}
	\end{center}
\end{figure}

Similarly, for $\bfu = (\bfn,\bfV)^T \in \R^{4N}$, with $\bfV \in \R^N$ defining $V^{-1}(\bfV) = \bfH \in \R^N$ by inverting the function $V$, we consider the corresponding finite element function $(V_h')^\theta$ on $\Ga_h^\theta$, which appears in the solution-dependent mass matrix $\bfM(\bfx,\bfu)$:
\begin{equation}
\label{eq:def V_h prime}
	(V_h')^\theta = \sum_{j = 1}^N V'(\bfH_j) \phi_j[\by+\theta\bfe] .
\end{equation}

Analogous to \cite[Section~7]{MCF}, \bbk we will use the following results which relate \ebk quantities on different surfaces, in particular proving a new result which compares solution dependent matrices.

Assuming that $\| \nabla_{\Gamma_h[\bfy]} e_h^0 \|_{L^\infty(\Gamma_h[\bfy])}\le \frac14$ Lemma~7.2 of \cite{MCF} (with $p=2$) shows that 
\begin{equation}
\label{norm equivalence}
\begin{aligned}
&\text{the norms $\|\cdot\|_{\bfM(\bfy+\theta\bfe)}$ and the \bbk semi-norms \ebk $\|\cdot\|_{\bfA(\bfy+\theta\bfe)}$} 
\text{are $h$-uniformly equivalent for $0\le\theta\le 1$.}
\end{aligned}
\end{equation}

Under the condition that $\| \nabla_{\Gamma_h[\bfy]} e_h^0 \|_{L^\infty(\Gamma_h[\bfy])} \le \tfrac14$, using the definition of $e_h^\theta$ in Lemma~4.1 of \cite{KLLP2017} and applying the Cauchy--Schwarz inequality yields the bounds,
\begin{equation}
\label{matrix difference bounds}
\begin{aligned}
\bfw^T (\bfM(\bfx)-\bfM(\bfy)) \bfz \leq &\ c \, \bbk \| \nabla_{\Gamma_h[\bfy]} e_h^0 \|_{L^\infty(\Gamma_h[\bfy])} \ebk  \|\bfw\|_{\bfM(\bfy)} \|\bfz\|_{\bfM(\bfy)} , \\[1mm]
\bfw^T (\bfA(\bfx)-\bfA(\bfy)) \bfz \leq &\ c \, \bbk \| \nabla_{\Gamma_h[\bfy]} e_h^0 \|_{L^\infty(\Gamma_h[\bfy])} \ebk \|\bfw\|_{\bfA(\bfy)} \|\bfz\|_{\bfA(\bfy)} .
\end{aligned}
\end{equation}
We will also use the bounds with additionally assuming $z_h \in W^{1,\infty}(\Ga_h[\bfy])$:
\begin{equation}
\label{matrix difference bounds e_x}
\begin{aligned}
\bfw^T (\bfM(\bfx)-\bfM(\bfy)) \bfz \leq &\ c \, \|\bfw\|_{\bfM(\bfy)} \|\bfe\|_{\bfA(\bfy)} \bbk \| z_h \|_{W^{1,\infty}(\Gamma_h[\bfy])} \ebk , \\[1mm]
\bfw^T (\bfA(\bfx)-\bfA(\bfy)) \bfz \leq &\ c \, \|\bfw\|_{\bfA(\bfy)} \|\bfe\|_{\bfA(\bfy)} \bbk \| z_h \|_{W^{1,\infty}(\Gamma_h[\bfy])} \ebk .
\end{aligned}
\end{equation}

Consider now a continuously differentiable function $\bfx:[0,T]\to\R^{3N}$ that defines a finite element surface $\Gamma_h[\bfx(t)]$ for every $t\in[0,T]$, and assume that its time derivative $\bfv(t)=\dot\bfx(t)$ is the nodal vector of a finite element function $v_h(\cdot,t)$ that satisfies
\begin{equation}
\label{vh-bound}
\| \nabla_{\Gamma_h[\bfx(t)]}v_h(\cdot,t) \|_{L^{\infty}(\Gamma_h[\bfx(t)])} \le K, \qquad 0\le t \le T.
\end{equation}
With $\bfe=\bfx(t)-\bfx(s)=\int_s^t \bfv(r) \d r$, the bounds \eqref{matrix difference bounds} then yield the following bounds, which were first shown in Lemma~4.1 of \cite{DziukLubichMansour_rksurf}: for $0\le s, t \le T$ with $K|t-s| \le \tfrac14$, 
we have with $C=c K$
\begin{equation}
\label{matrix difference bounds-t}
\begin{aligned}
\bfw^T \bigl(\bfM(\bfx(t))  - \bfM(\bfx(s))\bigr)\bfz \leq&\ C \, |t-s| \, \|\bfw\|_{\bfM(\bfx(t))}\|\bfz\|_{\bfM(\bfx(t))} , \\[1mm]
\bfw^T \bigl(\bfA(\bfx(t))  - \bfA(\bfx(s))\bigr)\bfz \leq&\ C\,  |t-s| \, \|\bfw\|_{\bfA(\bfx(t))}\|\bfz\|_{\bfA(\bfx(t))}.   
\end{aligned}
\end{equation}
Letting $s\to t$, this implies the bounds stated in Lemma~4.6 of~\cite{KLLP2017}:
\begin{equation}
\label{matrix derivatives}
\begin{aligned}
\bfw^T \frac\d{\d t}\bfM(\bfx(t))  \bfz \leq &\ C \, \|\bfw\|_{\bfM(\bfx(t))}\|\bfz\|_{\bfM(\bfx(t))} , \\[1mm]
\bfw^T \frac\d{\d t}\bfA(\bfx(t))  \bfz \leq &\ C \, \|\bfw\|_{\bfA(\bfx(t))}\|\bfz\|_{\bfA(\bfx(t))} .
\end{aligned}
\end{equation}
Moreover, by patching together finitely many intervals over which $K|t-s| \le \tfrac14$, we obtain that
\begin{equation}
\label{norm-equiv-t}
\begin{aligned}
&\text{the norms $\|\cdot\|_{\bfM(\bfx(t))}$ and the \bbk semi-norms \ebk  $\|\cdot\|_{\bfA(\bfx(t))}$}
\text{are $h$-uniformly equivalent for $0\le t \le T$.}
\end{aligned}
\end{equation}

The following new result is a solution dependent variant of the estimates relating mass matrices on different surfaces and with different geometric variables. In both cases we establish the analogons of \eqref{matrix difference bounds}--\eqref{matrix difference bounds e_x}. These estimates will play a crucial role in the stability proofs. 
\begin{lemma}
	\label{lemma:nonlinear mass matrix diff}
	Let $\varepsilon := \| \nabla_{\Ga_h[y]} e_h^0 \|_{L^\infty(\Ga_h[\bfy])}$ and let $\bfu = (\bfn,\bfV)^T$ and $\us = (\bfn^*,\bfV^*)^T \in \R^{4N}$ such that the corresponding $(V_h')^{-1}$ and $((V_h')^*)^{-1}$, defined by \eqref{eq:def V_h prime}, 
	\btb
	have bounded positive upper and lower bounds. 
	\etb 
	If $\varepsilon \leq \frac{1}{4}$, then, in the above setting, the following bounds hold:
	\begin{align}
	\tag{i}
	\bfw^T \big( \bfM(\bfx,\bfu)-\bfM(\bfy,\bfu) \big) \bfz 
	&\leq c \| \nabla_{\Ga_h[y]} e_h^0 \|_{L^\infty(\Ga_h[\bfy])} \, \| \bfw \|_{\bfM(\bfy)} \, \| \bfz \|_{\bfM(\bfy)}, \\
	\tag{ii}
	\bfw^T \big( \bfM(\bfx,\bfu)-\bfM(\bfy,\bfu) \big) \bfz 
	&\leq c \|\bfe\|_{\bfA(\bfy)} \| \bfw \|_{\bfM(\bfy)} \, \| z_h \|_{L^\infty(\Ga_h[\bfy])} ,
	\end{align}
	and
	\begin{align}
	\tag{iii}
	\bfw^T \big( \bfM(\bfx,\bfu)-\bfM(\bfx,\us) \big) \bfz 
	&\leq c \| u_h - u_h^* \|_{L^\infty(\Ga_h[\bfy])} \, \| \bfw \|_{\bfM(\bfx)}  \, \| \bfz \|_{\bfM(\bfx)} , \\
	\tag{iv}
	\bfw^T \big( \bfM(\bfx,\bfu)-\bfM(\bfx,\us) \big) \bfz 
	&\leq c \| \bfu - \us \|_{\bfM(\bfx)} \, \| \bfw \|_{\bfM(\bfx)} \, \| z_h \|_{L^\infty(\Ga_h[\bfx])} .
	\end{align}
\end{lemma}

\begin{proof} The first step of the proof is similar to that of \cite[Lemma~4.1]{KLLP2017}. 	
	Using the fundamental theorem of calculus and the Leibniz formula \cite[Lemma~2.2]{DziukElliott_ESFEM} we obtain
	\begin{equation}
	\label{eq:estimate for first two est}
	\begin{aligned}
	&\ \bfw^T \big( \bfM(\bfx,\bfu)-\bfM(\bfy,\bfu) \big) \bfz \\
	= &\ \intGhx \frac{1}{(V_h')^1} \, w_h^1 z_h^1 
	- \int_{\Ga_h[\bfy]} \frac{1}{(V_h')^0} \, w_h^0 z_h^0 \\
	= &\ \int_0^1 \ddtheta \intGhtheta \frac{1}{(V_h')^\theta} \, w_h^\theta z_h^\theta \, \d\theta \\
	= &\ \int_0^1 \intGhtheta \frac{1}{(V_h')^\theta} \, w_h^\theta \, \big(\nablaGhx \cdot e_h^\theta\big) \, z_h^\theta \, \d \theta \\
	\leq &\ \int_0^1 \|1 / (V_h')^\theta \|_{L^\infty(\Ga_h^\theta)} \, \|w_h^\theta\|_{L^2(\Ga_h^\theta)} \, \|\big(\nablaGhx \cdot e_h^\theta\big)\|_{L^\infty(\Ga_h^\theta)} \, \|z_h^\theta\|_{L^2(\Ga_h^\theta)}  \, \d \theta , 
	\end{aligned}
	\end{equation}
	where we have used that the material derivatives of $w_h^\theta$, $z_h^\theta$ and $(V_h')^\theta$ are vanishing with respect to $\theta$. 
	By \cite[Lemma~7.2]{MCF} it follows then, that under the condition
	\begin{equation*}
	\| \nabla_{\Ga_h^0} e_h^0 \|_{L^\infty(\Ga_h^0)} \leq \frac{1}{2} ,
	\end{equation*}
	the $L^p(\Ga_h^\theta)$ norms and $W^{1,p}(\Ga_h^\theta)$ semi-norms are equivalent for all $0 \leq \theta \leq 1$. 
	This, together with the the previous estimates imply the first estimate:
	\begin{equation*}
	\bfw^T \big( \bfM(\bfx,\bfu)-\bfM(\bfy,\bfu) \big) \bfz 
	\leq c \, \bbk \| \nabla_{\Gamma_h[\bfy]} e_h^0 \|_{L^\infty(\Gamma_h[\bfy])} \ebk \| \bfw \|_{\bfM(\bfy)} \, \| \bfz \|_{\bfM(\bfy)} ,
	\end{equation*}
	where the constant $c$ depends on the $L^\infty(\Ga_h[\bfy])$ norm of $1/(V_h')$.
	By interchanging the roles of $\nablaGhx \cdot e_h^\theta$ and $z_h$ in the last estimate of \eqref{eq:estimate for first two est}, and then using the same argument from above, we obtain the second bound.
	
	\btb 
	For the third estimate note that the functions $V'(H_h)$ and $V'(H_h^*)$ are bounded from below by some constant $c_0 > 0$. Since $[c_0,\infty) \to \R \colon x \mapsto x^{-1}$ is Lipschitz continuous and $V'$ is locally Lipschitz continuous, the composition $1/(V')$ is also locally Lipschitz continuous, therefore
		\begin{equation}
		\label{eq:estimate for second two est}
		\begin{aligned}
			\bfw \big( \bM(\bx,\bfu)-\bM(\bx,\us) \big) \bfz 
			= &\ \intGhx w_h \bigg( \frac{1}{V'(H_h)}  - \frac{1}{V'(H_h^*)} \bigg) z_h \\
			\leq &\ \intGhx |w_h| \bigg| \frac{1}{V'(H_h)}  - \frac{1}{V'(H_h^*)} \bigg| |z_h| \\
			\leq &\ c_0^{-2} \intGhx |w_h|\big| V'(H_h) - V'(H_h^*) \big| |z_h| \\
			\leq &\ c \intGhx |w_h| \big| H_h - H_h^* \big| |z_h| \\
			\leq &\ c \| w_h \|_{L^2(\Ga_h[\bfx])} \|H_h - H_h^*\|_{L^2(\Ga_h[\bfx])} \|z_h\|_{L^\infty(\Ga_h[\bfx])} .
		\end{aligned}
	\end{equation}
	where the constant $c$ depends on $c_0$ and on the Lipschitz constant of $V'$.
	\etb
	
	Using the fact that the function $V$ is a smooth and invertible \eqref{eq:properties of V}, and recalling that $V_h = V(H_h)$ and $V_h^* = V(H_h^*)$, it follows
	\begin{equation*}
	\|H_h - H_h^*\|_{L^2(\Ga_h[\bfx])} \leq c \|V_h - V_h^*\|_{L^2(\Ga_h[\bfx])}
	\leq c \|u_h - u_h^*\|_{L^2(\Ga_h[\bfx])} = c \| \bfu - \us \|_{\bfM(\bfx)} .
	\end{equation*}
	Combining these inequalities we obtain
	\begin{equation*}
	\bfw^T \big( \bfM(\bfx,\bfu)-\bfM(\bfx,\us) \big) \bfz 
	\leq c \| \bfu - \us \|_{\bfM(\bfx)} \, \| \bfw \|_{\bfM(\bfx)} \|z_h^1\|_{L^\infty(\Ga_h[\bfx])} ,
	\end{equation*}
	where the constant depends on the $L^\infty(\Ga_h[\bfx])$ norm of $1 / (V_h')^*$. 
	Similarly as before, by interchanging the roles of $H_h - H_h^*$ and $z_h$ in the last estimate of \eqref{eq:estimate for second two est} we obtain the fourth estimate.
\end{proof}

\section{Stability of the  full discretization}
\label{section:stability-full}

In the following section we will prove a stability result for linearly implicit BDF discretisations. \bbk Due to the mentioned structural similarity between \eqref{eq:matrix-vector form} and \eqref{eq:MCF matrix-vector form} the proof is based on \cite[Proposition~10.1]{MCF}, with some substantial differences due to the solution-dependent mass matrix $\bfM(\bfx,\bfu)$. \ebk 

\subsection{Auxiliary results by Dahlquist and Nevanlinna \& Odeh}

We recall two important results that enable us to use energy estimates for BDF methods up to order 5: the first result is from Dahlquist's $G$-stability theory, and the second one from the multiplier technique of Nevanlinna and Odeh.

\begin{lemma}[\cite{Dahlquist}]
	\label{lemma: Dahlquist}
	Let $\delta(\zeta) = \sum_{j=0}^q \delta_j \zeta^j$ and $\mu(\zeta) = \sum_{j=0}^q \mu_j \zeta^j$ be polynomials of degree at most $q$ (at least one of them of degree $q$) that have no common divisor. Let $\la \, \cdot , \cdot \, \ra$ denote an inner product on $\R^N$. 
	If
	\begin{equation*}
	\textnormal{Re} \frac{\delta(\zeta)}{\mu(\zeta)} > 0 \qquad \textrm{for} \quad |\zeta|<1,
	\end{equation*}
	then there exists a symmetric positive definite matrix $G = (g_{ij}) \in \R^{q\times q}$ 
	such that for all $\bfw_0,\dotsc,\bfw_q\in\R^N$
	\begin{equation*}
	\Big\la \sum_{i=0}^q \delta_i \bfw_{q-i} , \sum_{i=0}^q \mu_i \bfw_{q-i}  \Big\ra \ge \sum_{i,j=1}^q g_{ij} \la \bfw_i , \bfw_j \ra - \sum_{i,j=1}^q g_{ij} \la \bfw_{i-1} , \bfw_{j-1} \ra .
	\end{equation*}
\end{lemma}

In view of the following result, the choice $\mu(\zeta)=1-\eta\zeta$ together with the polynomial $\delta(\zeta)$ of the BDF methods will play an important role later on.
\begin{lemma}[\cite{NevanlinnaOdeh}]
	\label{lemma: NevanlinnaOdeh multiplier}
	If $q\leq5$, then there exists $0\leq\eta<1$ \st\ for $\delta(\zeta)=\sum_{\ell=1}^q \frac{1}{\ell}(1-\zeta)^\ell$,
	\begin{equation*}
	\textnormal{Re} \,\frac{\delta(\zeta)}{1-\eta\zeta} > 0 \qquad \textrm{for} \quad |\zeta|<1.
	\end{equation*}
	The smallest possible values of $\eta$ are found to be 
	$\eta= 0,  0,  0.0836, 0.2878,  0.8160$ 
	for $q=1,\dotsc,5$, respectively.
\end{lemma}

These results have previously been applied in the error analysis of BDF methods, in particular for mean curvature flow \cite{MCF}, and also for various parabolic problems in
\cite{AkrivisLiLubich_quasilinBDF,AkrivisLubich_quasilinBDF,KL2018,KovacsPower_quasilinear,LubichMansourVenkataraman_bdsurf,LLG}, where they were used when testing the error equation with the error. Similarly as in \cite{MCF} and \cite{LLG}, and in contrast to the other references above, here these results are used for testing the error equation with the discretized time derivative of the error. 

For the six-step BDF method a new and intriguing energy approach was recently introduced in \cite{Akrivis_etal_BDF6}.

\subsection{Errors and defects}

We define the nodal vectors $\xs(t) \in \R^{3N}$ and $\vs\t \in \R^{3N}$ by collecting the values of the exact solution $X(\cdot,t)$ and $v(\cdot,t)$, respectively, at the finite element nodes.
The vector $\us(t)$ contains the nodal values of the finite element function $u_h^*(\cdot,t) := \widetilde{R}_h u(\cdot,t) \in S_h[\xs(t)]^4$ that is defined by a Ritz map on the interpolated surface $\Gamma_h[\xs(t)]$:
\begin{equation}
\label{eq:Ritz map def}
	\int_{\Ga_h[\xs]} \!\! \!\!\!\! \nb_{\Gamma_h[\xs]} u_h^* \cdot \nabla_{\Gamma_h[\xs]} \varphi_h 
	+ \int_{\Ga_h[\xs]} \!\!\!\!  u_h^* \cdot \varphi_h 
	= \int_{\Ga[X]} \!\! \!\!\!\! \nabla_{\Gamma[X]} u \cdot \nabla_{\Gamma[X]} \varphi_h^\ell 
	+ \int_{\Ga[X]} \!\!\!\! u \cdot  \varphi_h^\ell
\end{equation}
for all $\varphi_h \in S_h[\xs]^4$, where again $\varphi_h^\ell$ denotes the lift \bbk of the function $\varphi_h$ onto $\Gamma[X]$. \ebk 


We insert these values into the numerical scheme, and obtain defects $\dv^n$, $\du^n$, $\dx^n$: for $n\ge q$,
\begin{subequations}
	\label{eq:BDF defect definitions}
	\begin{align}
	\bfK(\widetilde \bfx_*^n) \bfv_*^n &= \bfg(\widetilde \bfx_*^n,\widetilde \bfu_*^n) +  \bfM(\wtxls^n)  \dv^n, \\
	\bfM(\widetilde \bfx_*^n,\widetilde \bfu_*^n)\dot \bfu_*^{n} + \bfA(\widetilde \bfx_*^n) \bfu_*^n &= \bff(\widetilde \bfx_*^n,\widetilde \bfu_*^n) +  \bfM(\wtxls^n)  \du^n,  \\
	\dot \bfx_*^{n} &= \bfv_*^n + \dx^n ,
	\end{align}
\end{subequations}
where the backward difference time derivatives and the extrapolated values are given by \eqref{eq:backward differences def} and \eqref{eq:extrapolation def}.

\subsubsection{An $L^\infty$ error estimates and bound for the Ritz map}
\label{section:Ritz map bounds - semidiscrete}

In the upcoming stability proof, we will need some $L^\infty$-norm estimates for the Ritz map of $V$. This preparatory section is devoted to the proof of these estimates.

Recalling, form \eqref{eq:Ritz map def}, that $V_h^*(\cdot,t) = \widetilde{R}_h V(\cdot,t) \in S_h[\xs\t]$ is the Ritz map of $V(\cdot,t)$, and we denote, (omitting the argument $t$), $(V_h^*)^\ell = (\widetilde{R}_h V)^\ell = R_h V \in H^1(\Ga[X])$. Then using multiple triangle inequalities and an inverse estimate \cite[Theorem~4.5.11]{BrennerScott} \bbk with dimension $d = 2$, \ebk and norm equivalences, we obtain, for $t \in [0,T]$ and $h \leq h_0$,
\begin{equation}
\label{eq:Linfty error of Ritz map of Vh*}
	\begin{aligned}
		\|(V_h^*)^\ell - V\|_{L^\infty(\Ga[X])} 
		\leq &\ c \|\widetilde{R}_h V - \widetilde{I}_h V\|_{L^\infty(\Ga_h[\xs])} + \|I_h V - V\|_{L^\infty(\Ga[X])} \\
		\leq &\ c h^{-d/2} \|\widetilde{R}_h V - \widetilde{I}_h V\|_{L^2(\Ga_h[\xs])} + \|I_h V - V\|_{L^\infty(\Ga[X])} \\
		\leq &\ c h^{-d/2} \Big( \|R_h V - V\|_{L^2(\Ga[X])} + \|V - I_h V\|_{L^2(\Ga[X])} \Big) + \|I_h V - V\|_{L^\infty(\Ga[X])} \\
		\leq &\ c h^{2-d/2} \|V\|_{H^{2}(\Ga[X])} + c h^2 \|V\|_{W^{2,\infty}(\Ga[X])} ,
	\end{aligned}
\end{equation}
where for the last estimate, we have used the (sup-optimal) error bounds for the Ritz map in the $H^1$ norm, see \cite[Theorem~6.2]{Kovacs2018}, and error estimates for the interpolation in the $H^1$ and the in the $L^\infty$ norm, see Proposition~2.7 in \cite{Demlow2009}, with $m=2$ for $p=2$ and $p=\infty$ therein, respectively.

Using the equivalence of the $L^\infty$ norms on $\Ga_h[\xs]$ and $\Ga[X]$, via \cite[Lemma~7.2]{MCF}, the error bound \eqref{eq:Linfty error of Ritz map of Vh*} and a reverse triangle inequality, we then immediately obtain, for $t \in [0,T]$ (again omitted as an argument), $h \leq h_0$,  \bbk and with dimension $d = 2$, \ebk
\begin{equation}
\label{eq:Linfty bound for Vh*}
	\begin{aligned}
		\|V_h^*\|_{L^\infty(\Ga_h[\xs])} 
		\leq &\ c \|(V_h^*)^\ell\|_{L^\infty(\Ga[X])} \\
		\leq &\ c\|(V_h^*)^\ell - V\|_{L^\infty(\Ga[X])} + c\|V\|_{L^\infty(\Ga[X])} \\
		\leq &\ c h^{2-d/2} \|V\|_{H^{2}(\Ga[X])} + (c h^2 + c) \|V\|_{W^{2,\infty}(\Ga[X])} 
		\leq M , 
	\end{aligned}
\end{equation}
with an $M>0$ independent of $h$ and $t$.

\subsubsection{Two estimates for the extrapolation}

Analogously to the time-continuous stability proof, in the proof of fully discrete stability we will need $L^\infty$-norm estimates (now) for the extrapolation of the Ritz map of $u$. These preparatory estimates will play analogous roles as those in Section~\ref{section:Ritz map bounds - semidiscrete}, and are proved below.

We first derive an estimate for the error in the extrapolation of the Ritz map of the exact normal velocity. Using the Peano kernel representation (see \cite[Section~3.2.6]{Gautschi}) of the extrapolation error, analogously to the proof of Lemma~4.3 in \cite{KL2018}, we obtain
\begin{equation}
\label{eq:Linfty error of Ritz map of extrapolated Vh*}
	\begin{aligned}
		\|((V_h^*)^n)^\ell - ((\widetilde{V}_h^*)^n)^\ell\|_{L^\infty(\Ga[X(\cdot,t_n)])} 
		= \Big\|((V_h^*)^n)^\ell - \Big(\sum_{j=0}^{q-1} \gamma_j (V_h^*)^{n-1-j}\Big)^\ell \Big\|_{L^\infty(\Ga[X(\cdot,t_n)])} 
		\leq &\ c \tau^q .
	\end{aligned}
\end{equation}

Furthermore, we prove that the extrapolations at subsequent times $t_n$ and $t_{n-1}$ of the Ritz map of $u(\cdot,t)$ are $L^\infty$-norm bounded by $O(\tau)$, similarly as in the proof of Lemma~8.1 in \cite{LLG}. Namely, the following estimate holds:
\begin{equation}
\label{eq:comparing Ritz maps in time}
	\begin{aligned}
		\|\widetilde{u}_*^n - \widetilde{u}_*^{n-1}\|_{L^{\infty}(\Ga_h[\wtx^{n-1}])}
		\leq &\ \sum_{j=0}^{k-1} |\gamma_j| \int_{t_{n-j-2}}^{t_{n-j-1}} \|\mat_h \widetilde{R}_h u(\cdot,s)  \|_{L^{\infty}(\Ga_h[\wtx^{n-1}])} \d s 
		\leq c \tau ,
	\end{aligned}
\end{equation}
where in the last estimate to show the boundedness of $\mat_h \widetilde{R}_h u$, we have used a similar argument as \eqref{eq:Linfty error of Ritz map of Vh*}, but here using the error estimates in the material derivative of the Ritz map \cite[Theorem~6.4]{Kovacs2018}.

\subsubsection{Fully discrete error equations}
The errors of the numerical solution $\bfx^n$, $\bfv^n$ and $\bfu^n = (\bfn^n,\bfV^n)^T$ are denoted by
\begin{align*}
\ex^n = \bfx^n - \xls^n , \qquad \ev^n = \bfv^n - \vls^n, \qquad \eu^n = \bfu^n - \uls^n,
\end{align*}
and we abbreviate
\begin{equation}
\label{dotexun}
\dotex^n = \frac{1}{\tau} \sum_{j=0}^q \delta_j \ex^{n-j}, \qquad 
\doteu^n = \frac{1}{\tau} \sum_{j=0}^q \delta_j \eu^{n-j} .
\end{equation}
Subtracting \eqref{eq:BDF defect definitions} from \eqref{eq:BDF}, we obtain the following error equations:
\begin{subequations} 
	\label{eq:error equations - full}
	\begin{align}
	\label{eq:error eq - v - full}
	\bfK(\widetilde \bfx^n)  \ev^n &= \bfr_\bfv^n ,	\\
	\label{eq:error eq - u - full}
	\bfM(\widetilde \bfx^n,\widetilde \bfu^n)\doteu^n+ \bfA(\widetilde \bfx^n)\eu^n &= \bfr_\bfu^n , \\
	\label{eq:error eq - x - full}
	\dotex^n &= \ev^n - \dx^n , 
	\end{align}
\end{subequations}
where the right-hand side terms denote
\begin{subequations} 
	\label{eq:error terms - full}
	\begin{align}
	\label{eq:error term - v - full}
	\bfr_\bfv^n
	=&\ -  \big( \bfK(\widetilde \bfx^n)-\bfK(\widetilde \bfx_*^n) \big) \vls^n \\
	\nonumber 
	&\ + \big(\bfg(\widetilde \bfx^n,\widetilde \bfu^n) - \bfg(\widetilde \bfx_*^n,\widetilde \bfu_*^n)\big) - \bfM(\wtxls^n) \dv^n , \\
	\label{eq:error term - u - full}
	\bfr_\bfu^n
	=&\ -  \big( \bfM(\widetilde \bfx^n,\widetilde \bfu^n)-\bfM(\widetilde \bfx^n,\widetilde \bfu_*^n) \big) \dot\bfu_*^n  \\
	\nonumber
	&\ -  \big( \bfM(\widetilde \bfx^n,\widetilde \bfu_*^n)-\bfM(\widetilde \bfx_*^n,\widetilde \bfu_*^n) \big) \dot\bfu_*^n  \\
	\nonumber
	&\ - \big( \bfA(\widetilde \bfx^n)-\bfA(\widetilde \bfx_*^n) \big) \bfu_*^n \\
	\nonumber &\ + \big(\bff(\widetilde \bfx^n,\widetilde \bfu^n) - \bff(\widetilde \bfx_*^n,\widetilde \bfu_*^n)\big) -  \bfM(\wtxls^n)  \du^n .
	\end{align}
\end{subequations}  

In the sequel we need the following discrete dual norm
\begin{equation*}
	\| \bfd \|_{\star,\xls^n}^2 := \bfd^T \bfM(\xls^n) \bfK(\xls^n)^{-1} \bfM(\xls^n) \bfd. 
\end{equation*}

\subsection{Stability estimate}

The following fully discrete stability result holds for the errors in the positions $\ex^n$, in the velocity $\ev^n$, and in the geometric variables $\eu^n = (\bfe_\bfn^n,\bfe_\bfV^n)^T$, provided that the defects are small enough.

The basic idea of the proof of this result is the same as for the stability result for mean curvature flow \cite[Proposition~10.1]{MCF}, however there are substantial differences due to the solution-dependent mass matrix which have to be addressed carefully.


\begin{proposition}
	\label{proposition:stability - full} 
	Assume that the function $V$ satisfies the assumptions in \eqref{eq:properties of V}.
	Consider the \bbk full discretisation of generalized mean curvature flow using evolving surface finite elements of degree $k \geq 2$ in space, \ebk and linearly implicit BDF time discretization \eqref{eq:BDF} of order $q$ with $2 \leq q \leq 5$ \bbk in time. \ebk  
	Assume that, for step sizes restricted by $\tau \le \Co h$ (where $\Co > 0$ is arbitrary), there exists $\kappa$ with $2 \leq \kappa \leq k$ such that the defects are bounded by 
	\begin{equation}
	\label{eq:assume small defects}
		\|\dx^n\|_{\bfK(\xls^n)} \leq c h^\kappa , \quad
		\|\dv^n\|_{\star,\xls^n} \leq c h^\kappa , \quad 
		\|\du^n\|_{\bfM(\xls^n)} \leq c h^\kappa ,
	\end{equation} 
	for $q\tau \leq n\tau \leq T$, and that also the errors of the starting values are bounded by
	\begin{equation}
	\label{eq:assume small initial values}
		\|\ex^i\|_{\bfK(\xls^i)} \leq c h^\kappa , \quad
		\|\ev^i\|_{\bfK(\xls^i)} \leq c h^\kappa , \quad 
		\|\eu^i\|_{\bfK(\xls^i)} \leq c h^\kappa , \quad
	\end{equation} 
	for $i=0,\dots,q-1$, and, with the notation $\be \ex^i = (\ex^i-\ex^{i-1})/\tau$, for $i = 1,\dotsc,q-1$, 
	\begin{equation}
	\label{eq:assume small initial values 2}
	\tau^{1/2} \| \be \ex^i \|_{\bfK(\bfx^{i})}  \leq c h^\kappa .
	\end{equation} 
	
	Then, there exist  $h_0>0$ and $\tau_0>0$ such that the following stability estimate holds for all $h\leq h_0$, $\tau\le\tau_0$, and $n$ with $n\tau \le T$, satisfying $\tau \leq \Co h$,
	\begin{equation}
	\label{eq:stability bound - full}
	\begin{aligned}
	& \| \ex^n \|_{\bfK(\xls^n)}^2 +  \| \ev^n \|_{\bfK(\xls^n)}^2 +  \| \eu^n \|_{\bfK(\xls^n)}^2 \\
	& \leq \ C  \sum_{i=0}^{q-1}\Bigl( \| \ex^i \|_{\bfK(\xls^i)}^2 + \| \ev^i \|_{\bfK(\xls^i)}^2 + \| \eu^i \|_{\bfK(\xls^i)}^2\bigr) + C \tau \sum_{i=1}^{q-1} \| \be \ex^i \|_{\bfK(\bfx^{i})}^2 \\
	& \quad +C \max_{0\le j \le n}\|\dv^j\|_{\star,\xls^j}^2  
	+ C \tau\sum_{j=q}^{n} \|\du^j\|_{\bfM(\xls^j)}^2 + C  \tau\sum_{j=q}^{n} \|\dx^j\|_{\bfK(\xls^j)}^2,
	\end{aligned}
	\end{equation}
	where the constant $C > 0$ is independent of $h$, $\tau$ and $n$ with $n\tau \leq T$, but depends on the final time $T$.
\end{proposition}

\bbk 
In Section~\ref{section:consistency - full} we will prove that the defects satisfy a $O(h^k + \tau^q)$ bound (with $k \geq 2$ and $2 \leq q \leq 5$), and in view of the mild step size restriction $\tau \leq \Co h$ the assumed bounds \eqref{eq:assumed bounds - n} indeed hold with $\kappa = \min\{k,q\} \geq 2$.
\ebk 

\begin{proof}
	\bbk Similarly, as in \cite{MCF}, the stability proof uses energy estimates in the matrix--vector formulation, and relies mostly on the preparatory results of Section~\ref{section:relating surfaces} (with $\bfx$ and $\xs$ in the role of $\bfx$ and $\bfy$).  
	The proof is based on energy estimates testing with the discrete time-derivative $\doteu^n$, relying on the $G$-stability theory of Dahlquist (Lemma~\ref{lemma: Dahlquist}), and the multiplier techniques of Nevanlinna and Odeh (Lemma~\ref{lemma: NevanlinnaOdeh multiplier}).
	By this we obtain uniform-in-time $H^1$-norm error bounds, which allow a control in the $W^{1,\infty}$-norm of the errors via an inverse estimate.
	
	As we have noted before algorithm \eqref{eq:BDF} for the generalized mean curvature flow is very similar to that of mean curvature flow \cite{MCF}. In particular, the error equations (\eqref{eq:error equations - full} with \eqref{eq:error terms - full}) are the same except the solution-dependent mass terms in \eqref{eq:error eq - u - full}. 
	Therefore, the proof of this result is also closely related to that of Proposition~10.1 in \cite{MCF}. Repeating estimates for identical terms would not yield any extra insight compared to the original proof in \cite{MCF}, and hence are only recalled therefrom. 
	Due to the mentioned similarities, the proof below has the same structure as the proof of Proposition~10.1 in \cite{MCF}. Corresponding estimates are carried out in corresponding parts. 
	
	A key difference in the proofs is that the present proof works with the solution dependent norm $\|\cdot\|_{\Mxu}$ instead of $\|\cdot\|_{\bfM(\wtx^n)}$. This, simple looking, yet crucial difference requires extra care during the stability analysis. 
	In particular the estimates for the solution-dependent mass terms require $L^\infty$-norm bounds on the weight function $V_h'$ in \eqref{eq:solution-dependent mass matrix}.
	
	Throughout the proof we will use the following conventions: References to the proof Proposition~10.1 in \cite{MCF} are abbreviated to \cite{MCF}, unless a specific reference therein is given. For example, (i) in part (A) of the proof of Proposition~10.1 of \cite{MCF} is referenced as \cite[(A.i)]{MCF}. 
	By $c$ and $C$ we will denote generic $h$- and $\tau$-independent constants, which might take different values on different occurrences. \ebk 
	
	\emph{Preparations:} Let $t^*$ with $0 < t^* \leq T$ (which {\it a priori} might depend on $\tau$ and $h$) be the maximal time such that the following inequalities hold: 
	\begin{equation}
	\label{eq:assumed bounds - n}
	\begin{aligned}
	\| e_x^n \|_{W^{1,\infty}(\Ga_h[\xls^n])} &\leq  h^{(\kappa-1)/2} , \\
	\| e_v^n \|_{W^{1,\infty}(\Ga_h[\xls^n])} &\leq  h^{(\kappa-1)/2} , \\
	\| e_u^n \|_{W^{1,\infty}(\Ga_h[\xls^n])} &\leq  h^{(\kappa-1)/2} , 
	\end{aligned} \qquad \textrm{ for } \quad n \tau \leq t^* .
	\end{equation}
	Note that the by the smallness condition for the errors in the initial data we have, at least, $t^* = q \tau$.
	At the end of the proof we will show that in fact $t^* = T$.

	Through a series of bounds we now show $L^\infty$-norm bounds for $((\widetilde V_h')^\ast)^n$ and $(\widetilde V_h')^n$, which are crucial to estimate the solution-dependent mass terms.
	A similar argument was necessary in the proof of Proposition~7.1 and 7.2 in \cite{LLG}.
	
	In particular, the third bound \bbk from \eqref{eq:assumed bounds - n} \ebk implies, using the equivalence of the $L^\infty$ norms, for $h \leq h_0$ sufficiently small and $n \tau \leq t^*$,
	\begin{equation}
	\label{eq:V Linfty difference - fully discrete}
		\begin{aligned}
			\| \widetilde{V}_h^n - (\widetilde{V}_h^*)^n \|_{\rL^\infty(\Ga_h[\xls^n])}
			\leq &\ \sum_{j=0}^{q-1} |\gamma_j| \| \widetilde{V}_h^{n-1-j} - (\widetilde{V}^*_h)^{n-1-j} \|_{\rL^\infty(\Ga_h[\xls^{n-1-j}])} \\
			\leq &\ \sum_{j=0}^{q-1} |\gamma_j| \| e_u^{n-1-j} \|_{\rW^{1,\infty}(\Ga_h[\xls^{n-1-j}])} \leq C_\gamma h^{(\kappa-1)/2} ,
		\end{aligned}
	\end{equation}
	with $C_\gamma = \sum_{j=0}^{q-1} |\gamma_j| = 2^q - 1$.
	
	The $L^\infty(\Ga_h[\xls^n])$ boundedness of $\widetilde{V}_h^n$ is obtained by combining \eqref{eq:Linfty bound for Vh*} and \eqref{eq:V Linfty difference - fully discrete}: for $h \leq h_0$ and for $n \tau \leq t^*$,
	\begin{equation}
	\label{eq:Linfty bound for Vh - fully discrete}
		\begin{aligned}
			\|\widetilde{V}_h^n\|_{L^\infty(\Ga_h[\xls^n])} \leq &\ 
			\|(\widetilde{V}_h^*)^n\|_{L^\infty(\Ga_h[\xls^n])} + \|\widetilde{V}_h^n - (\widetilde{V}_h^*)^n\|_{L^\infty(\Ga_h[\xls^n])} \\
			\leq &\ C_\gamma M + C_\gamma h^{(\kappa-1)/2} \leq 
			(C_\gamma+1) M .
		\end{aligned}
	\end{equation}
	
	Recall that $\widetilde{V}_h^n = V(\widetilde{H}_h^n)$, which defines the curvature data $\widetilde{H}_h^n$ by inverting the function $V$, \bbk since $\widetilde{V}_h^n \in I$ (see \eqref{eq:properties of V}). \ebk Note, however, that $\widetilde{H}_h^n$ is not an extrapolation for the mean curvature, but merely a suggestive notation, expressing its relation to $\widetilde{V}_h^n$.
	
	Then, by the local Lipschitz continuity of the function $V^{-1}$ (see~\eqref{eq:properties of V}), we have, for $n \leq n^*$,
	\begin{equation}
	\label{eq:H Linfty difference - fully discrete}
	\begin{aligned}
	C_\gamma h^{(\kappa-1)/2} \geq &\ 
	\| \widetilde{V}_h^n - (\widetilde{V}_h^*)^n \|_{\rL^\infty(\Ga_h[\xls^n])} 
	= \| V(\widetilde{H}_h^n) - V((\widetilde{H}_h^*)^n) \|_{\rL^\infty(\Ga_h[\xls^n])} \\
	\geq &\ \frac{1}{L_{2cM}^{V^{-1}}} \big\| V^{-1}\big(V(\widetilde{H}_h^n)\big) - V^{-1}\big(V((\widetilde{H}_h^*)^n)\big) \big\|_{\rL^\infty(\Ga_h[\xls^n])} \\
	= &\ \frac{1}{L_{2cM}^{V^{-1}}} \big\| \widetilde{H}_h^n - (\widetilde{H}_h^*)^n \big\|_{\rL^\infty(\Ga_h[\xls^n])} 
	= 
	c \| e_{\widetilde{H}}^n\|_{\rL^\infty(\Ga_h[\xls^n])} .
	\end{aligned}
	\end{equation}
	where $L_{2cM}^{V^{-1}}$ depends on the $L^\infty(\Ga_h[\xls^n])$ norms of $\widetilde{V}_h^*$ and $\widetilde{V}_h$, \eqref{eq:Linfty bound for Vh*} and \eqref{eq:Linfty bound for Vh - fully discrete}.
	
	As for the semi-discrete case, since the mean curvature $H$ is assumed to be time-uniformly bounded from above and below~\eqref{eq:bounds on mean curvature}, using the local Lipschitz continuity of $V^{-1}$ as in \eqref{eq:H Linfty difference - fully discrete} we derive $h$-uniform bounds on the Ritz map of the mean curvature.
	By the definition of the lift map, we have the equality $\widetilde{H}_h^*(x,t) = (\widetilde{H}_h^*)^\ell(x^\ell,t)$, for any time $t$ and for any $x \in \Ga_h[\xs\t]$. Then, by the triangle inequality, we obtain \bbk (with dimension $d = 2$) \ebk 
	\begin{align*}
	|(\widetilde{H}_h^*)^n| \leq &\ |H(\cdot,t_n)| + |H(\cdot,t_n) - ((\widetilde{H}_h^*)^n)^\ell| \\
	\leq &\ |H(\cdot,t_n)| + \|H(\cdot,t_n) - ((\widetilde{H}_h^*)^n)^\ell\|_{L^\infty(\Ga[X(\cdot,t_n)])} \\
	\leq &\ H_1 + c \|V(\cdot,t_n) - ((V_h^*)^n)^\ell\|_{L^\infty(\Ga[X(\cdot,t_n)])} + c \|((V_h^*)^n)^\ell - ((\widetilde{V}_h^*)^n)^\ell\|_{L^\infty(\Ga[X(\cdot,t_n)])} \\
	\leq &\ H_1 + c h^{2-d/2} + c \tau^q ,  
	\end{align*}
	and similarly 
	\begin{align*}
	|(\widetilde{H}_h^*)^n| \geq &\ |H(\cdot,t_n)| - |H(\cdot,t_n) - ((\widetilde{H}_h^*)^n)^\ell| \\
	\geq &\ |H(\cdot,t_n)| - \|H(\cdot,t_n) - ((\widetilde{H}_h^*)^n)^\ell\|_{L^\infty(\Ga[X(\cdot,t_n)])} \\
	\geq &\ H_1 - c \|V(\cdot,t_n) - ((V_h^*)^n)^\ell\|_{L^\infty(\Ga[X(\cdot,t_n)])} - c \|((V_h^*)^n)^\ell - ((\widetilde{V}_h^*)^n)^\ell\|_{L^\infty(\Ga[X(\cdot,t_n)])} \\
	\geq &\ H_1 - c h^{2-d/2} - c \tau^q ,
	\end{align*} 
	where in both estimates we have used \eqref{eq:Linfty error of Ritz map of Vh*} and \eqref{eq:Linfty error of Ritz map of extrapolated Vh*}. 
	The argument is now repeated for $\widetilde{H}_h^n$, now comparing $\widetilde{H}_h^n$ with $(\widetilde{H}_h^*)^n$, and using \eqref{eq:H Linfty difference - fully discrete} instead of \eqref{eq:Linfty error of Ritz map of Vh*}, and \eqref{eq:Linfty error of Ritz map of extrapolated Vh*}:
	\begin{equation*}
	\begin{aligned}
	|\widetilde{H}_h^n| \leq &\ |(\widetilde{H}_h^*)^n| + |\widetilde{H}_h^n - (\widetilde{H}_h^*)^n| \\
	\leq &\ |(\widetilde{H}_h^*)^n| + \|\widetilde{H}_h^n - (\widetilde{H}_h^*)^n\|_{L^\infty(\Ga_h[\xls^n])} \\
	\leq &\ |(\widetilde{H}_h^*)^n| + c \| e_{\widetilde{H}}^n\|_{\rL^\infty(\Ga_h[\xls^n])} \\
	\leq &\ |(\widetilde{H}_h^*)^n| + c h^{(\kappa-1)/2} , 
	\end{aligned}
	\qquad 
	\begin{aligned}
	|\widetilde{H}_h^n| \geq &\ |(\widetilde{H}_h^*)^n| - |\widetilde{H}_h^n - (\widetilde{H}_h^*)^n| \\
	\geq &\ |(\widetilde{H}_h^*)^n| - \|\widetilde{H}_h^n - (\widetilde{H}_h^*)^n\|_{L^\infty(\Ga_h[\xls^n])} \\
	\geq &\ |(\widetilde{H}_h^*)^n| - c \| e_{\widetilde{H}}^n\|_{\rL^\infty(\Ga_h[\xls^n])} \\
	\geq &\ |(\widetilde{H}_h^*)^n| - c h^{(\kappa-1)/2} . 
	\end{aligned} 
	\end{equation*}
	Altogether, recalling that $d = 2$ and $\kappa \geq 2$, we obtain the bounds, for $h \leq h_0$ and $\tau \leq \tau_0$,
	\begin{align}
	\label{eq:uniform bounds for Hh* - fully discrete}
	0 \, < \, \tfrac23 H_0 \, \leq \, \|(\widetilde{H}_h^*)^n\|_{L^\infty(\Ga_h[\xls^n])} \, \leq \, \tfrac32 H_1, \qquad \text{$h$- and $\tau$-uniformly for $n\tau \leq T$,}
	\intertext{and}
	\label{eq:uniform bounds for Hh - fully discrete}
	0 \, < \, \tfrac12 H_0 \, \leq \, \|\widetilde{H}_h^n\|_{L^\infty(\Ga_h[\xls^n])} \, \leq \, 2 H_1, \qquad \text{$h$- and $\tau$-uniformly for $n \tau \leq t^*$.}
	\end{align}
	
	By \eqref{eq:uniform bounds for Hh* - fully discrete} and \eqref{eq:uniform bounds for Hh - fully discrete}, and using that the function $V':\R \to \R$ is positive everywhere, the functions $(\widetilde{V}_h')^n = V'(\widetilde{H}_h(\cdot,t_n))$ and $((\widetilde{V}_h')^*)^n = V'(\widetilde{H}_h^*(\cdot,t_n))$ satisfy that
	\begin{equation}
	\label{bounds V-s - fully discrete}
	\text{$(\widetilde{V}_h')^\ast(\cdot,t_n)$ and $\widetilde{V}_h'(\cdot,t_n)$ have $h$- and $\tau$-uniform positive upper and lower bounds for $n \tau \leq t^*$.} 
	\end{equation}
	\bbk These bounds are required by Lemma~\ref{lemma:nonlinear mass matrix diff}. \ebk

	\emph{Norm equivalences:} 
	Throughout the stability proof we will additionally need some norm equivalence results.
	The fact that 
	\begin{equation}
	\label{norm-equiv-BDF}
	\text{the \bbk (semi)-norms \ebk  $\|\cdot\|_{\bfM(\wtx^n)}$ and $\|\cdot\|_{\bfA(\wtx^n)}$ are $h$- and $\tau$-uniformly equivalent for $q \tau \leq n \tau \leq t^*$,}
	\end{equation}
	is proven by the same techniques as (10.12) in \cite{MCF}.
	
	As we have already pointed out, the function $\widetilde{H}_h^n$ is given by $\widetilde{H}_h^n = V^{-1}(\widetilde{V}_h^n)$ (and not an extrapolation). This then defines $(\widetilde{V}_h')^n =V'(\widetilde{H}_h^n)$, which appears in the solution-dependent norm generated by $\Mxu$:
	\begin{align*}
	\| \bfw \|_{\Mxu}^2 := \bfw^T \Mxu \bfw = &\ \int_{\Ga_h[\wtx^n]} \frac{1}{(\widetilde{V}_h')^n } \, |w_h|^2  .
	\end{align*}
	In view of \eqref{bounds V-s - fully discrete}, i.e.~$(\widetilde{V}'_h)^n$ is uniformly bounded away from zero and bounded from above, the matrix $\Mxu$ is indeed generates a norm.
	
	Furthermore, the bounds \eqref{bounds V-s - fully discrete} additionally yield that 
	\begin{equation}
	\label{norm and norm-sol equivalence}
	\text{the norm $\|\cdot\|_{\Mxu[n]}$ is equivalent to $\|\cdot\|_{\bfM(\wtx^n)}$ uniformly in $h$ and $\tau$ for any $q \tau \leq n \tau \leq t^*$}.
	\end{equation}
	
	As a final preparatory result, we prove the norm equivalence of the solution-dependent norm at different times, i.e.~the solution-dependent analogue of the norm equivalence \eqref{norm-equiv-BDF}. To this end, we start by rewriting
	\begin{align*}
	\| \bfw \|_{\Mxu}^2 - \| \bfw \|_{\Mxu[n-1]}^2 
	= &\ \bfw^T \big( \Mxu - \Mxu[n-1] \big) \bfw \\
	= &\ \bfw^T \big( \Mxu - \bfM(\wtx^{n-1},\wtu^{n}) \big) \bfw \\
	&\ + \bfw^T \big( \bfM(\wtx^{n-1},\wtu^{n}) - \Mxu[n-1] \big) \bfw .
	\end{align*}
	The first term on the right-hand side is estimated using Lemma~\ref{lemma:nonlinear mass matrix diff} (i), using \eqref{eq:assumed bounds - n} to ensure the $W^{1,\infty}$ norm boundedness of $\widetilde{u}_h^n$, together with an estimate for $\wtx^n - \wtx^{n-1}$. In (10.11) of \cite{MCF} it was shown that $\|\widetilde{W}_h^n\|_{W^{1,\infty}(\Ga_h[\wtx^n])} \leq K$, where the coefficient functions of $\widetilde{W}_h^n$ are given by $\widetilde{\bfW}^n = ( \wtx^n - \wtx^{n-1} ) / \tau$. Hence, we altogether obtain
	\begin{align*}
	\bfw^T \big( \Mxu - \bfM(\wtx^{n-1},\wtu^{n}) \big) \bfw \leq &\ c \tau \|\widetilde{W}_h^n\|_{W^{1,\infty}(\Ga_h[\wtx^n])} \|\bfw\|_{\bfM(\wtx^{n})}^2 \\
	\leq &\ c \tau \| \bfw \|_{\Mxu}^2 ,
	\end{align*}
	where for the second inequality we have used \eqref{norm and norm-sol equivalence}, i.e.~the equivalence between the norms $\|\cdot\|_{\Mxu}$ and $\|\cdot\|_{\bfM(\wtx^{n})}$.
	
	The second term on the right-hand side is estimated using Lemma~\ref{lemma:nonlinear mass matrix diff} (iii). Using the inequality \eqref{eq:comparing Ritz maps in time}, we estimate as
	\begin{align*}
	&\ \bfw^T \big( \bfM(\wtx^{n-1},\wtu^{n}) - \Mxu[n-1] \big) \bfw \\
	= &\ \bfw^T \big( \bfM(\wtx^{n-1},\wtu^{n}) - \bfM(\wtx^{n-1},\wtuls^{n}) \big) \bfw 
	+ \bfw^T \big( \bfM(\wtx^{n-1},\wtuls^{n}) - \bfM(\wtx^{n-1},\wtuls^{n-1}) \big) \bfw \\
	&\ + \bfw^T \big( \bfM(\wtx^{n-1},\wtuls^{n}) - \bfM(\wtx^{n-1},\wtu^{n-1}) \big) \bfw \\
	\leq &\ c \Big( \|\widetilde{e}_u^n\|_{L^{\infty}(\Ga_h[\wtx^{n-1}])}
	+ \|\widetilde{u}_*^n - \widetilde{u}_*^{n-1}\|_{L^{\infty}(\Ga_h[\wtx^{n-1}])} 
	+ \|\widetilde{e}_u^{n-1}\|_{L^{\infty}(\Ga_h[\wtx^{n-1}])} \Big) \|\bfw\|_{\bfM(\wtx^{n-1})}^2 \\ 
	\leq &\ c \Big( h^{(\kappa-1)/2} + c \tau \Big) \|\bfw\|_{\Mxu[n-1]}^2 ,
	\end{align*}
	where in the last inequality we have again used the norm equivalence \eqref{norm and norm-sol equivalence}.
	
	The combination of the two above estimates, together with the mild restriction $\tau \leq C_0 h$, yields
	\begin{align*}
	\| \bfw \|_{\Mxu}^2 \leq &\ \| \bfw \|_{\Mxu[n-1]}^2  + c C_0 h \| \bfw \|_{\Mxu}^2 
	+ c \Big( h^{(\kappa-1)/2} + c C_0 h \Big) \|\bfw\|_{\Mxu[n-1]}^2 .
	\end{align*}
	Then absorbing the second term to the right-hand side yields
	\begin{align*}
	\| \bfw \|_{\Mxu}^2 \leq &\  \frac{ 1 + c  h^{(\kappa-1)/2} + c h }{1 - c h} \|\bfw\|_{\Mxu[n-1]}^2 
	\leq  \big(1 + C h^r \big) \|\bfw\|_{\Mxu[n-1]}^2 ,
	\end{align*}
	with $1/2 \leq r := \min\{1,(\kappa-1)/2\}$ and for $h \leq h_0$. For the last inequality here we have used that $(1-ch)^{-1} \leq 1 + C h$ for some constant $C > 0$.
	
	By reversing the roles of the arguments, we obtain
	\begin{equation}
	\label{eq:Mxu - n-1 to n}
	\big(1 - C h^r \big) \| \bfw \|_{\Mxu[n-1]}^2 \leq \|\bfw\|_{\Mxu}^2 .
	\end{equation}
	
	Therefore, for sufficiently small $h \leq h_0$ and $\tau \leq \tau_0$ (subject to $\tau \leq C_0 h$), we have that
	\begin{equation}
	\label{norm-sol-equiv-BDF}
	\text{the norms $\|\cdot\|_{\Mxu[n]}$ are $h$- and $\tau$-uniformly equivalent for $q \tau \leq n \tau \leq t^*$.}
	\end{equation}

	\medskip
	(A) \emph{Estimates for the surface PDE:} \bbk We test the error equation for $\eu$ with the time derivative $\doteu$. In order to obtain a Nevanlinna--Odeh multiplier term, \ebk we form the difference of equation \eqref{eq:error eq - u - full} for $n$ with $\eta$ times this equation for $n-1$, for $\eta\in[0,1)$ of 
	Lemma~\ref{lemma: NevanlinnaOdeh multiplier}, and then we test this difference
	with the discrete time derivative $\doteu^n$ defined by \eqref{dotexun}. This yields, for $(q+1) \tau \leq n \tau \leq t^*$,
	\begin{equation} 
	\label{eq:error eq tested - u - full}
	\begin{aligned}
	& (\doteu^n)^T\bfM(\wtx^n,\wtu^n)\doteu^n - \eta (\doteu^n)^T\bfM(\wtx^{n-1},\wtu^{n-1})\doteu^{n-1} + 
	(\doteu^n)^T\bfA(\wtx^n)(\eu^n - \eta \eu^{n-1}) 
	\\
	& = 
	-  \eta (\doteu^n)^T\big( \bfA(\wtx^n) -  \bfA(\wtx^{n-1}) \big)  \eu^{n-1} 
	+  (\doteu^n)^T (\ru^n-\eta\ru^{n-1}).
	\end{aligned}
	\end{equation}
	
	(i) On the left-hand side of \eqref{eq:error eq tested - u - full}, the first term is 
	\begin{equation*}
	(\doteu^n)^T\bfM(\wtx^n,\wtu^n)\doteu^n =  \|\doteu^n\|_{\Mxu}^2 .
	\end{equation*}
	The second term is bounded by
	\begin{align*}
	(\doteu^n)^T \Mxu[n-1] \doteu^{n-1}
	\leq &\ \|\doteu^n\|_{\Mxu[n-1]} \|\doteu^{n-1}\|_{\Mxu[n-1]} \\
	\leq &\ \tfrac{1}{2}   \|\doteu^n\|_{\Mxu[n-1]}^2 +  \tfrac{1}{2}   \|\doteu^{n-1}\|_{\Mxu[n-1]}^2 \\
	\leq &\ \tfrac{1}{2}  (1+ch^r) \|\doteu^n\|_{\Mxu}^2 +  \tfrac{1}{2}  \|\doteu^{n-1}\|_{\Mxu[n-1]}^2,
	\end{align*} 
	where for the last inequality we used the bound \eqref{eq:Mxu - n-1 to n} (to raise the superscript from $n-1$ to $n$).
	This yields
	\begin{equation}
	\label{M-lower bound}
	\begin{aligned}
	& (\doteu^n)^T \Mxu \doteu^n - \eta (\doteu^n)^T \Mxu[n-1] \doteu^{n-1} 
	\\
	& \qquad\geq  \big( 1 - \tfrac{1}{2} \eta (1+ ch^r) \big)  \|\doteu^n\|_{\Mxu}^2 - \tfrac{1}{2} \eta \|\doteu^{n-1}\|_{\Mxu[n-1]}^2.
	\end{aligned} 
	\end{equation}
	
	The above terms pose the requirement to work with the solution-dependent norm $\|\cdot\|_{\Mxu}$, and not with $\|\cdot\|_{\bfM(\wtx)}$-norms used for mean curvature flow.

	(ii) The terms involving the stiffness matrix are estimated exactly as in \cite[(A.ii)]{MCF}. We first estimate the third term on the left-hand side of \eqref{eq:error eq tested - u - full}: the combination of Lemma~\ref{lemma: Dahlquist} and \ref{lemma: NevanlinnaOdeh multiplier} yields
	\begin{equation}
	\label{eq:DNO term}
	\begin{aligned}
	& (\doteu^{n})^T \bfA(\wtx^n)(\eu^{n} - \eta \eu^{n-1}) = \Bigl(\frac{1}{\tau} \sum_{i=0}^q \delta_i \eu^{n-i}\Bigr)^T\bfA(\wtx^n)(\eu^{n} - \eta \eu^{n-1})
	\\
	&\geq \frac{1}{\tau} \bigg( \sum_{i,j=1}^q g_{ij} (\eu^{n-q+i})^T \bfA(\wtx^n) \eu^{n-q+j} - 
	\sum_{i,j=1}^q g_{ij} (\eu^{n-q+i-1})^T \bfA(\wtx^n) \eu^{n-q+j-1} \bigg) .
	\end{aligned}
	\end{equation}
	While for the first term on the right-hand side of \eqref{eq:error eq tested - u - full}, cf.~equation (10.18) in \cite{MCF}, we have
	\begin{equation}
	\label{eq:A-diff}
	\begin{aligned}
	- (\doteu^n)^T\bigl( \bfA(\wtx^n) -  \bfA(\wtx^{n-1}) \bigr)  \eu^{n-1} 
	\leq &\ c \rho \|\doteu^n\|_{\bfM(\wtx^n)}^2 +  c \rho^{-1}\|\eu^{n-1}\|_{\bfA(\wtx^{n-1})}^2 \\
	\leq &\ c \rho \|\doteu^n\|_{\Mxu}^2 +  c \rho^{-1}\|\eu^{n-1}\|_{\bfA(\wtx^{n-1})}^2 ,
	\end{aligned}
	\end{equation}
	where for the last estimate we have used the norm equivalence \eqref{norm and norm-sol equivalence} to measure $\doteu^n$ in the solution dependent norm $\|\cdot\|_{\Mxu}$.
	
	We now estimate the remaining terms. 
	Recalling \eqref{eq:error terms - full}, the last term on the right-hand side of \eqref{eq:error eq tested - u - full} altogether reads:
	\begin{equation}
	\label{eq:remainder terms}
	\begin{alignedat}{3}
	&\ (\doteu^n)^T (\ru^n-\eta\ru^{n-1}) \\
	= &\ - (\doteu^n)^T  \big( \bfM(\widetilde \bfx^n,\widetilde \bfu^n)-\bfM(\widetilde \bfx^n,\widetilde \bfu_*^n) \big) \dot\bfu_*^n 
	+ \eta (\doteu^n)^T  \big( \bfM(\widetilde \bfx^{n-1},\widetilde \bfu^{n-1})-\bfM(\widetilde \bfx^{n-1},\widetilde \bfu_*^{n-1}) \big) \dot\bfu_*^{n-1} \\
	&\ - (\doteu^n)^T  \big( \bfM(\widetilde \bfx^n,\widetilde \bfu_*^n)-\bfM(\widetilde \bfx_*^n,\widetilde \bfu_*^n) \big) \dot\bfu_*^n 
	+ \eta (\doteu^n)^T  \big( \bfM(\widetilde \bfx^{n-1},\widetilde \bfu_*^{n-1})-\bfM(\widetilde \bfx_*^{n-1},\widetilde \bfu_*^{n-1}) \big) \dot\bfu_*^{n-1} \\ 
	&\ - (\doteu^n)^T  \big( \bfA(\widetilde \bfx^n)-\bfA(\widetilde \bfx_*^n) \big) \bfu_*^n 
	+ \eta (\doteu^n)^T  \big( \bfA(\widetilde \bfx^{n-1})-\bfA(\widetilde \bfx_*^{n-1}) \big) \bfu_*^{n-1} \\ 
	&\ + (\doteu^n)^T  \big(\bff(\widetilde \bfx^n,\widetilde \bfu^n) - \bff(\widetilde \bfx_*^n,\widetilde \bfu_*^n)\big) 
	- \eta (\doteu^n)^T  \big(\bff(\widetilde \bfx^{n-1},\widetilde \bfu^{n-1}) - \bff(\widetilde \bfx_*^{n-1},\widetilde \bfu_*^{n-1})\big) \\ 
	&\ - (\doteu^n)^T  \bfM(\wtxls^n)  \du^n 
	+ \eta (\doteu^n)^T  \bfM(\wtxls^{n-1})  \du^{n-1} . 
	\end{alignedat}
	\end{equation}
	The terms not involving the solution-dependent mass matrix are estimated by the exact same techniques as the corresponding terms in \cite{MCF} (see (iv), (v), and (vi) below). On the other hand, the terms in the first two lines require new estimates compared to \cite[(A.ii)]{MCF}.
	
	(iii) In the two terms in the first line of \eqref{eq:remainder terms} the position vectors are fixed. Hence, using Lemma~\ref{lemma:nonlinear mass matrix diff} (iv) we estimate them (similarly to the time-continuous case) by
	\begin{equation}
	\label{eq:rhs estimate - 1st dot e vs solution}
	\begin{aligned}
	& - (\doteu^n)^T  \big( \bfM(\widetilde \bfx^n,\widetilde \bfu^n)-\bfM(\widetilde \bfx^n,\widetilde \bfu_*^n) \big) \dot\bfu_*^n 
	+ \eta (\doteu^n)^T  \big( \bfM(\widetilde \bfx^{n-1},\widetilde \bfu^{n-1})-\bfM(\widetilde \bfx^{n-1},\widetilde \bfu_*^{n-1}) \big) \dot\bfu_*^{n-1} \\
	&\leq  c \| \doteu^n\|_{\bfM(\wtx^n)} \|\wteu^n\|_{\bfM(\wtx^n)}
	+ c \| \doteu^n\|_{\bfM(\wtx^{n-1})} \|\wteu^{n-1}\|_{\bfM(\wtx^{n-1})} \\
	&\leq  \rho \| \doteu^n\|_{\Mxu}^2  + c  \rho^{-1}   \|\wteu^n\|_{\bfM(\wtx^n)}^2
	+ c  \rho^{-1} \|\wteu^{n-1}\|_{\bfM(\wtx^{n-1})}^2 ,
	\end{aligned}
	\end{equation}
	where we have used the norm equivalence \eqref{norm and norm-sol equivalence}, and then Young's inequality with a small $\rho > 0$, independent of $h$, $\tau$, and $n \leq n^*$, which will be chosen later on.
	
	Analogously, using Lemma~\ref{lemma:nonlinear mass matrix diff} (ii), for the terms in the second line of \eqref{eq:remainder terms} we obtain, with a small $\rho > 0$,
	\begin{equation}
	\label{eq:rhs estimate - 2nd dot e vs solution}
	\begin{aligned}
	& - (\doteu^n)^T  \big( \bfM(\widetilde \bfx^n,\widetilde \bfu_*^n)-\bfM(\widetilde \bfx_*^n,\widetilde \bfu_*^n) \big) \dot\bfu_*^n 
	+ \eta (\doteu^n)^T  \big( \bfM(\widetilde \bfx^{n-1},\widetilde \bfu_*^{n-1})-\bfM(\widetilde \bfx_*^{n-1},\widetilde \bfu_*^{n-1}) \big) \dot\bfu_*^{n-1} \\
	& \leq  c \| \doteu^n\|_{\bfM(\wtx^n)} \|\wtex^n\|_{\bfK(\wtx^n)}
	+ c \| \doteu^n\|_{\bfM(\wtx^{n-1})} \|\wtex^{n-1}\|_{\bfK(\wtx^{n-1})} \\
	& \leq  \rho \| \doteu^n\|_{\Mxu}^2  + c \rho^{-1} \|\wtex^n\|_{\bfK(\wtx^n)}^2 + c \rho^{-1} \|\wtex^{n-1}\|_{\bfK(\wtx^{n-1})}^2 ,
	\end{aligned}
	\end{equation}
	where we have again used the norm equivalence \eqref{norm and norm-sol equivalence}.
	
	(iv) For the stiffness matrix terms in \eqref{eq:remainder terms}, the rather complicated estimates of \cite[(A.iv)]{MCF} can be used verbatim, and they yield the bound, see (10.24)--(10.27) in \cite{MCF}:
	\begin{equation}
	\label{eq:A terms}
	\begin{aligned}
	& - (\doteu^n)^T  \big( \bfA(\widetilde \bfx^n)-\bfA(\widetilde \bfx_*^n) \big) \bfu_*^n 
	+ \eta (\doteu^n)^T  \big( \bfA(\widetilde \bfx^{n-1})-\bfA(\widetilde \bfx_*^{n-1}) \big) \bfu_*^{n-1} \\
	& \leq \bbk - \sum_{j=0}^{q-1} \sigma_j \be \Big( (\eu^{n-j})^T \big( \bfA(\wtx^n)-\bfA(\wtxls^n) \big) \bfu_*^n \Big)
	+ \eta \sum_{j=0}^{q-1} \sigma_j \be \Big( (\eu^{n-j})^T \big( \bfA(\wtx^{n-1})-\bfA(\wtxls^{n-1}) \big) \bfu_*^{n-1} \Big) \ebk \\
	& \quad c \sum_{j=0}^{q} \|\eu^{n-j}\|_{\bfA(\wtx^{n-j})}^2  + c \sum_{j=0}^{q} \|\ex^{n-1-j}\|_{\bfK(\wtx^{n-1-j})}^2 
	+ c \sum_{j=0}^{q-1} \| \be \ex^{n-1-j}\|_{\bfK(\wtx^{n-1-j})}^2
	\end{aligned}
	\end{equation}

	(v) The non-linear terms in \eqref{eq:remainder terms} are estimated as \cite[(A.v)]{MCF}, with a small $\rho > 0$, by
	\begin{equation}
	\label{eq:rhs estimate - nonlinearities}
	\begin{aligned}
	& (\doteu^n)^T \big(\bff(\wtx^n,\widetilde \bfu^n) - \bff(\wtxls^n,\widetilde \bfu_*^n)\big)  
	- \eta (\doteu^n)^T \big(\bff(\wtx^{n-1},\widetilde \bfu^{n-1}) - \bff(\wtxls^{n-1},\widetilde \bfu_*^{n-1})\big) \\
	& \leq  \rho \|\doteu^n\|_{\bfM(\wtx^n)}^2 + c \rho^{-1} \|\wteu^n\|_{\bfK(\wtx^n)}^2 + c\rho^{-1}  \|\wteu^{n-1}\|_{\bfK(\wtx^{n-1})}^2 
	+ c \rho^{-1} \|\wtex^n\|_{\bfK(\wtx^n)}^2 + c\rho^{-1} \|\wtex^{n-1}\|_{\bfK(\wtx^{n-1})}^2  \\
	& \leq c \rho \|\doteu^n\|_{\Mxu}^2 + c \rho^{-1} \|\wteu^n\|_{\bfK(\wtx^n)}^2 + c \|\wteu^{n-1}\|_{\bfK(\wtx^{n-1})}^2 
	+ c \rho^{-1} \|\wtex^n\|_{\bfK(\wtx^n)}^2 + c\rho^{-1} \|\wtex^{n-1}\|_{\bfK(\wtx^{n-1})}^2 .
	\end{aligned}
	\end{equation}
	where for the final estimate we have used the norm equivalence \eqref{norm and norm-sol equivalence}.
	
	(vi) Finally, the defect terms are bounded, exactly as \cite[(A.v)]{MCF} but additionally using the norm equivalence \eqref{norm and norm-sol equivalence}, with a small $\rho > 0$, by
	\begin{equation}
	\label{eq:rhs estimate - defects}
	\begin{aligned}
	& - (\doteu^n)^T \bfM(\wtxls^n) \du^n 
	+ \eta (\doteu^n)^T \bfM(\wtxls^{n-1}) \du^{n-1} \\
	& \leq  \rho \|\doteu^n\|_{\Mxu}^2 + c \rho^{-1} \|\du^n\|_{\bfM(\wtxls^n) }^2 + c\rho^{-1}  \|\du^{n-1}\|_{\bfM(\wtxls^{n-1} )}^2 .
	\end{aligned}
	\end{equation}

	We substitute the estimates from (i)--(vi) into \eqref{eq:error eq tested - u - full}, which altogether yields the inequality, for $q+1 \leq n \leq n^*$,
	\begin{equation}
	\label{eq:stability estimate - pre sum}
	\begin{aligned}
	& \Big( 1 -  \frac{1}{2} \eta (1+ ch^r) - c\rho \Big) \|\doteu^n\|_{\Mxu}^2 - \frac{1}{2} \eta  \|\doteu^{n-1}\|_{\Mxu[{n-1}]}^2 \\
	&\quad + \frac{1}{\tau} \bigg( \sum_{i,j=1}^q g_{ij} (\eu^{n-q+i})^T \bfA(\wtx^n) \eu^{n-q+j} 
	- \sum_{i,j=1}^q g_{ij} (\eu^{n-q+i-1})^T \bfA(\wtx^n) \eu^{n-q+j-1} \bigg) \\
	&\leq 
	- \sum_{j=0}^{q-1} \sigma_j \be \Big( (\eu^{n-j})^T \big( \bfA(\wtx^n)-\bfA(\wtxls^n) \big) \bfu_*^n \Big)
	\\
	&\quad	+ \eta \sum_{j=0}^{q-1} \sigma_j \be \Big( (\eu^{n-j})^T \big( \bfA(\wtx^{n-1})-\bfA(\wtxls^{n-1}) \big) \bfu_*^{n-1} \Big)
	+ c \eps_n
	\end{aligned}
	\end{equation}
	with 
	\begin{equation}
	\label{eq:eps-n}
	\begin{aligned}
	\eps_n &=  \sum_{j=0}^{q} \|\eu^{n-j}\|_{\bfA(\wtx^{n-j})}^2  + \sum_{j=0}^{q} \|\ex^{n-1-j}\|_{\bfK(\wtx^{n-1-j})}^2 
	+   \sum_{j=0}^{q-1} \| \be  \ex^{n-1-j}\|_{\bfK(\wtx^{n-1-j})}^2
	\\ &\quad 
	+   \|\du^n\|_{\bfM(\wtxls^n)}^2 + \|\du^{n-1}\|_{\bfM(\wtxls^{n-1})}^2 .
	\end{aligned}
	\end{equation}
	
	This is exactly the same formula as (10.30)--(10.31) in \cite{MCF}, except the solution-dependent norms (instead of $\|\cdot\|_{\bfM(\wtx)}$) on the $\doteu$ terms. Therefore, the proof can be finished by the exact same arguments, but using the norm equivalence \eqref{norm-sol-equiv-BDF}. 
	

	\bbk 
	(B) \emph{Estimates for the velocity equation:} Since the velocity equation is formally the same here and in \cite{MCF}: \eqref{eq:error eq - v - full} and \cite[equation~(10.3a)]{MCF} coincide, the analysis in \cite[Part~(B)]{MCF} and the obtained result applies to the present situation as well, and yields the estimate:
	\begin{equation}
	\label{eq:energy estimate - v - full}
		\|\ev^n\|_{\bfK(\wtxls^n)}^2 \leq c \sum_{i=0}^{q-1} \Big( \|\ex^{n-1-i}\|_{\bfK(\wtxls^{n-1-i})}^2 + c \|\eu^{n-1-i}\|_{\bfK(\wtxls^{n-1-i})}^2 \Big) + c \|\dv\t\|_{\star,\xls^n}^2  .
	\end{equation} 
	
	(C) \emph{Combination:} Since the final estimates from Part (A) and (B) are formally the same as the two corresponding estimates in \cite[Part~(A) and (B)]{MCF}, the proof can be finished exactly as it was done in \cite[Part~(C)]{MCF} (using Gronwall's inequality), by which we obtain the stability estimate \eqref{eq:stability bound - full} for $n \tau \leq t^*$.
	
	\ebk 
	
	It remains to show that $t^\ast = T$ for $h$ and $\tau$ sufficiently small. 
	Upon noting that by the assumed defect bounds \eqref{eq:assume small defects} and \eqref{eq:assume small initial values}--\eqref{eq:assume small initial values 2}, the obtained stability bound \eqref{eq:stability bound - full} implies
	\begin{equation*}
		\| \ex^n \|_{\bfK(\xls^n)}
		+ \| \ev^n \|_{\bfK(\xls^n)}
		+ \| \eu^n \|_{\bfK(\xls^n)}
		\leq C h^\kappa, 
	\end{equation*}
	and therefore by an inverse inequality we obtain, for $n \tau \leq t^*$,
	\begin{equation}
	\label{eq:showing assumed bound beyond t*}
		\begin{aligned}
			&\ \| e_x(\cdot,t_n) \|_{W^{1,\infty}(\Ga_h(\xls^n))} 
			+ \| e_v(\cdot,t_n) \|_{W^{1,\infty}(\Ga_h(\xls^n))}
			+ \| e_u(\cdot,t_n) \|_{W^{1,\infty}(\Ga_h(\xls^n))} \\
			&\ \ \leq \frac{c}{h} \Big( \| \ex^n \|_{\bfK(\xls^n)}
			+\| \ev^n \|_{\bfK(\xls^n)}
			+\| \eu^n \|_{\bfK(\xls^n)} \Big) \\
			&\ \ \leq c C h^{\kappa-1} \leq \frac{1}{2} h^{(\kappa-1)/2} ,
		\end{aligned}
	\end{equation}
	for sufficiently small $h$. This means that the bounds \eqref{eq:assumed bounds - n} can be extended beyond $t^*$, contradicting the maximality of $t^\ast$, unless $t^\ast = T$ already. Therefore we have shown the stability bound \eqref{eq:stability bound - full} for all $n \tau \leq T$.
\end{proof}

\section{Consistency estimates for the full discretisation}
\label{section:consistency - full}
The following estimates for the defects \eqref{eq:BDF defect definitions} are proved by approximation result for the interpolation and the Ritz map and by geometric approximation errors \cite{Kovacs2018}.

\begin{lemma}\label{lemma:consistency - fully-discrete}
	Assume that the surface $\Ga[X]$ evolving under generalized mean curvature flow is sufficiently regular on the time interval $[0,T]$. Then, there exists constants $h_0 > 0$, $\tau_0 > 0$, and $c = c(T) > 0$ such that for all $h \leq h_0$ and $\tau \leq \tau_0$, satisfying $0 \leq n \tau \leq T$, the defects $d_x^n \in S_h(\Ga_h[\xls^n])^3$, $d_v^n \in S_h(\Ga_h[\xls^n])^3$ and $d_u^n \in S_h(\Ga_h[\xls^n])^4$ of the $k$th-degree finite elements and the $q$-step backward difference formula are bounded as
	\begin{alignat*}{3}
	\|\dx^n\|_{\bfK(\xls^n)} = &\ \|d_x^n\|_{H^1(\Ga_h[\xs(t_n)])} & \leq &\ c\tau^q, \\
	\|\dv^n\|_{\star,\xls^n} = &\ \|d_v^n\|_{H_h\inv(\Ga_h[\xs(t_n)])} & \leq &\ c (h^k+\tau^q) , \\
	\|\du^n\|_{\bfM(\xls^n)} = &\ \|d_u^n\|_{L^2(\Ga_h[\xs(t_n)])} & \leq &\ c (h^k +\tau^q) .
	\end{alignat*}
	The constant $c$ is independent of $h$, $\tau$ and $n$ with $n\tau\le T$.
\end{lemma}
\bbk 
\begin{proof}	
	(a) Since the first and third equation of \eqref{eq:BDF defect definitions} is (formally) the same as the corresponding equations in (7.14) in \cite{MCF} the proof of the defect bound of $\dv^n$ and $\dx^n$ in \cite[Lemma~11.1]{MCF} hold in the present case as well. 
	
	(b) We now decompose the defects $d_u^n$ into spatial and temporal parts, cf.~\cite[Lemma~6.2]{KL2018}:
	\begin{equation*}
		d_u^n = d_{\tau,u}^n + d_{h,u}(t_n) .
	\end{equation*}

	We first prove bounds for the semi-discrete defects $d_{h,u}$. 	
	In general the proof is similar to \cite[Lemma~8.1]{MCF}.
	By \cite[Theorem~6.3]{Kovacs2018}, the error in the Ritz map \eqref{eq:Ritz map def} and in its material derivative there holds:
	\begin{equation}
		\begin{aligned}
			\| (u_h^\ast)^\ell(\cdot,t)-u(\cdot,t) \|_{\rH^1(\Ga[X(\cdot,t)])} \leq c h^k, \\
			\| (\dbullet u_h^\ast)^\ell(\cdot,t)- \dbullet u(\cdot,t) \|_{\rH^1(\Ga[X(\cdot,t)])} \leq c h^k.
		\end{aligned}\label{eq:8.1}
	\end{equation}
	The Ritz map error bound directly implies
	\begin{equation}
		\| (V^\ast_h)^\ell (\cdot,t) - V(\cdot,t) \|_{\rH^1(\Ga[X(\cdot,t)])} \leq c h^k .
		\label{eq:8.2}
	\end{equation}
	Using the function $f(u, \nabla_\Gamma u) = |A|^2 u$ we rewrite as
	\begin{align*}
	&\int_{\Gamma_h[\xs]} \!((V'_h)^*)^{-1} \partial^\bullet_h u_h^* \cdot \varphi_h
	+
	\int_{\Gamma_h[\xs]} \!\nabla_{\Gamma_h[\xs]} u_h^* \cdot \nabla_{\Gamma_h[\xs]} \varphi_h
	= 
	\int_{\Gamma_h[\xs]} \! f(u_h^*, \nabla_{\Gamma_h[\xs]}u_h^*) \cdot  \varphi_h
	+ \int_{\Gamma_h[\xs]} d_u \cdot \varphi_h 
	\end{align*}
	for all $\varphi_h \in S_h[\xs]^4$.
	Subtracting the weak formulation for the exact solution \eqref{evolutioneqs-weak} from this equation, we obtain 
	\begin{align*}
	\int_{\Gamma_h[\xs]} d_u \cdot  \varphi_h = &\ \biggl[ \int_{\Gamma_h[\xs]} \!((V'_h)^*)^{-1} \partial^\bullet_h u_h^* \cdot \varphi_h -
	\int_{\Gamma[X]} (V')^{-1} \partial^\bullet u \cdot   \varphi_h^\ell \biggr]
	\\
	&\ + \biggl( \int_{\Gamma_h[\xs]} \!\nabla_{\Gamma_h[\xs]} u_h^* \cdot \nabla_{\Gamma_h[\xs]} \varphi_h 
	- \int_{\Gamma[X]} \nabla_{\Gamma[X]} u \cdot  \nabla_{\Gamma[X]} \varphi_h^\ell \biggr)
	\\ 
	&\ - \biggl( \int_{\Gamma_h[\xs]} \! f(u_h^*, \nabla_{\Gamma_h[\xs]}u_h^*) \cdot  \varphi_h - 
	\int_{\Gamma[X]}  f(u,\nabla_{\Gamma[X]}u )\cdot   \varphi_h^\ell \biggr)
	\end{align*}
	for all $\varphi_h \in S_h[\xs]^4$. The second and the third term on the right-hand side can be estimated exactly as in \cite[Lemma~8.1]{MCF}. 
	The critical term in the square brackets can be rewritten as
	\begin{align*}
	&\ \int_{\Gamma_h[\xs]} \!((V'_h)^*)^{-1} \partial^\bullet_h u_h^* \cdot \varphi_h -
	\int_{\Gamma[X]} (V')^{-1} \partial^\bullet u \cdot \varphi_h^\ell \\
	= &\ 
	\biggl(
	\int_{\Gamma_h[\xs]} \!((V'_h)^*)^{-1} \partial^\bullet_h u_h^* \cdot \varphi_h -
	\int_{\Gamma[X]} (((V'_h)^*)^{-1})^\ell (\partial^\bullet_h u_h^*)^\ell \cdot \varphi_h^\ell
	\biggr) \\
	&\ + 
	\biggl(
	\int_{\Gamma[X]} (((V'_h)^*)^{-1})^\ell (\partial^\bullet_h u_h^*)^\ell \cdot \varphi_h^\ell
	-
	\int_{\Gamma[X]} (V')^{-1} (\partial^\bullet_h u_h^*)^\ell \cdot \varphi_h^\ell
	\biggr) \\
	&\ + 
	\biggl(
	\int_{\Gamma[X]} (V')^{-1} (\partial^\bullet_h u_h^*)^\ell \cdot \varphi_h^\ell
	-
	\int_{\Gamma[X]} (V')^{-1} \partial^\bullet u \cdot \varphi_h^\ell
	\biggr) 
	\end{align*}
	Using \eqref{eq:properties of V} and \eqref{bounds V-s - fully discrete}, \cite[Lemma~7.4]{KLLP2017} implies that the first term is bounded by $c h^k \| \varphi_h^\ell \|_{L^2(\Gamma[X])}$. 
	Using Cauchy-Schwarz inequality, \eqref{bounds V-s - fully discrete} and \eqref{eq:8.2}, the second term is bounded by $c h^k \| \varphi_h^\ell \|_{L^2(\Gamma[X])}$.
	
	Finally, by Cauchy-Schwarz inequality the last term is bounded by $c h^k \| \varphi_h^\ell \|_{L^2(\Gamma[X])}$ using \eqref{eq:8.1}.
	From here on $\|d_{h,u}(t_n)\|_{L^2\Ga_h} = O(h^k)$ is shown exactly as the proof of \cite[Lemma~8.1]{KLLP2017}. 
	
	The temporal defect $d_{h,u}^n$ is bounded by a straightforward combination of the above techniques and those of \cite[Lemma~11.1]{MCF}, and is therefore omitted.
\end{proof}

\ebk 

\section{Proof of Theorem~\ref{theorem:fully discrete convergence}}
\label{section:fully discrete proof completed}

With the stability estimate of Proposition~\ref{proposition:stability - full} and the defect bounds of Lemma \ref{lemma:consistency - fully-discrete} at hand, the proof is a usual decomposition of the errors, cf.~\cite[Section~12]{MCF}.

The errors are decomposed using finite element interpolations of $X$ and $v$ and the Ritz map \eqref{eq:Ritz map def} for $u$. The decomposed parts are then estimated, as $O(h^k+\tau^q)$ in the $H^1(\Gamma[X])$ norm, either using the stability estimate of Proposition~\ref{proposition:stability - full} together with the defect bounds of Lemma~\ref{lemma:consistency - fully-discrete}, or using the interpolation and Ritz map error bounds of \cite{Kovacs2018}. Altogether proving the stated theorem.

\section{Numerical examples}
\label{section:numerics}

We performed the following numerical experiments for various generalized mean curvature flows: 
\begin{itemize}
	\item[-] Convergence tests for spheres where the exact solutions of generalized mean curvature flows are known, i.e.~for inverse mean curvature flow, and generalized mean curvature and generalized inverse mean curvature flow. 
	\item[-] We report on numerical solutions for various flows, and also on geometric quantities which are known to be monotone along their respective flows, e.g.~Hawking mass for inverse mean curvature flow.
	\item[-] We have performed some numerical experiments for some non-convex initial surfaces.
\end{itemize}

All our numerical experiments use quadratic evolving surface finite elements, and linearly implicit backward difference time discretisation of various orders. The numerical computations were carried out in Matlab. The initial meshes for all surfaces were generated using DistMesh by \cite{distmesh}, without exploiting any symmetry of the surfaces.

\subsection{Convergence tests}

Using the algorithm \eqref{eq:BDF}, i.e.~using quadratic evolving surface finite elements for spatial discretisation in combination with a $q$-step linearly implicit BDF method for time integration, we computed approximations to various generalized mean curvature flows in two dimensions over the time interval $[0,T] = [0,1]$. 
The computations are carried out for a sphere (with initial radius $R_0 = 3$), since in these cases the exact solutions are known. Under all these flows spheres remain spherical and only change their radius, expressed by the mean curvature as $R\t = d / H(\cdot,t)$ \bbk with dimension $d = 2$. \ebk 

For all experiments we have started the time integration from the nodal interpolations of the exact initial values $\nu(\cdot,0)$ and $H(\cdot,0)$. 

\bbk We will derive exact solutions of various flows starting from a $d$-dimensional sphere of radius $R_0$, since the derived formulas give insight into the higher dimensional case as well. In our numerical experiments we always consider $d = 2$. \ebk 

\textit{Inverse mean curvature flow.} 
For an $d$-dimensional sphere $\Ga[X(\cdot,t)] = \{R\t x \mid x \in \Ga^0 \}$, with $\Ga^0 = \{ |x| = 1 \mid x \in \R^{d+1} \}$, from the velocity law \eqref{eq:velocity law} and the ODE \eqref{eq:velocity ODE} we obtain that the radius $R\t$ of the sphere satisfies the ODE
\begin{alignat}{3}
\label{eq:ODE - iMCF}
\diff R\t = &\ \frac{R\t}{d}, & \qquad & \text{with initial value} \quad R(0) = R_0 ,
\intertext{whose solution is given by }
\label{eq:ODE - iMCF solution}
R\t = &\ R_0 e^{t / d}, & \qquad & \text{for all} \quad 0 \leq t < \infty .
\end{alignat}

\textit{Generalized inverse mean curvature flow  $\alpha > 1$.} 
For a sphere, from the velocity law \eqref{eq:velocity law} and the ODE \eqref{eq:velocity ODE}, we now obtain that the radius $R\t$ satisfies the ODE, with $\alpha > 1$,
\begin{alignat}{3}
\label{eq:ODE - iMCFgen}
\diff R\t = &\ \Big( \frac{R\t}{d} \Big)^\alpha , & \qquad & \text{with initial value} \quad R(0) = R_0 ,
\end{alignat}
whose solution is given by
\begin{equation}
\label{eq:ODE - iMCFgen solution}
\begin{aligned}
R\t = &\ \big( R_0^{1-\alpha} - (\alpha-1)  d^{-\alpha}  t \big)^{1-\alpha} ,  \\
& \ \text{with maximal existence time} \quad T_{\max} = \frac{R_0^{1-\alpha}}{(\alpha-1)  d^{-\alpha}} . 
\end{aligned}
\end{equation}

We note here that for $0 < \alpha < 1$ a solution of the same form exists for all $0 \leq t < \infty$.

\textit{Generalized mean curvature flow  $\alpha > 0$.} 
For a sphere, from the velocity law \eqref{eq:velocity law} and the ODE \eqref{eq:velocity ODE}, we now obtain that the radius $R\t$ satisfies the ODE, with $\alpha > 0$,
\begin{alignat}{3}
\label{eq:ODE - MCFgen}
\diff R\t = &\ - \Big( \frac{d}{R\t} \Big)^\alpha , & \qquad & \text{with initial value} \quad R(0) = R_0 ,
\end{alignat}
whose solution is given by
\begin{equation}
\label{eq:ODE - MCFgen solution}
\begin{aligned}
R\t = &\ \big( R_0^{1+\alpha} - (1+\alpha)  d^{\alpha}  t \big)^{1/(1+\alpha)} ,  \\
& \ \text{with maximal existence time} \quad T_{\max} = \frac{R_0^{1+\alpha}}{(1+\alpha)  d^{\alpha}} . 
\end{aligned}
\end{equation}
See, e.g., \cite[Beispiel~2.13]{Schulze_diss}.

\bbk In Figures~\ref{fig:conv_iMCF}, \ref{fig:conv_iMCFgen} and \ref{fig:conv_MCFgen} we report on convergence tests, respectively, for inverse mean curvature flow, generalised inverse mean curvature flow and generalised mean curvature flow both with $\alpha = 2$. The logarithmic plots show the $L^\infty(H^1)$ norm of the errors between the numerical and exact solutions for position, normal vector, and mean curvature (columns left to right). 
Our computations were carried out with initial radius $R_0 = 3$ on the time interval $[0,1]$, using a sequence of time step sizes $\tau_k = \tau_{k-1}/2$ with $\tau_0 = 0.2$, and a sequence of meshes with mesh widths $h_k \approx 2^{-1/2} h_{k-1}$. 
The top rows reporting on the spatial convergence rate, while the bottom rows are reporting on the temporal convergence. The observed convergence $O(h^2 + \tau^2)$ (note the reference lines) match the theoretical results of Theorem~\ref{theorem:fully discrete convergence}.

\ebk 

\begin{figure}[htbp]
	\includegraphics[width=0.9\textwidth]{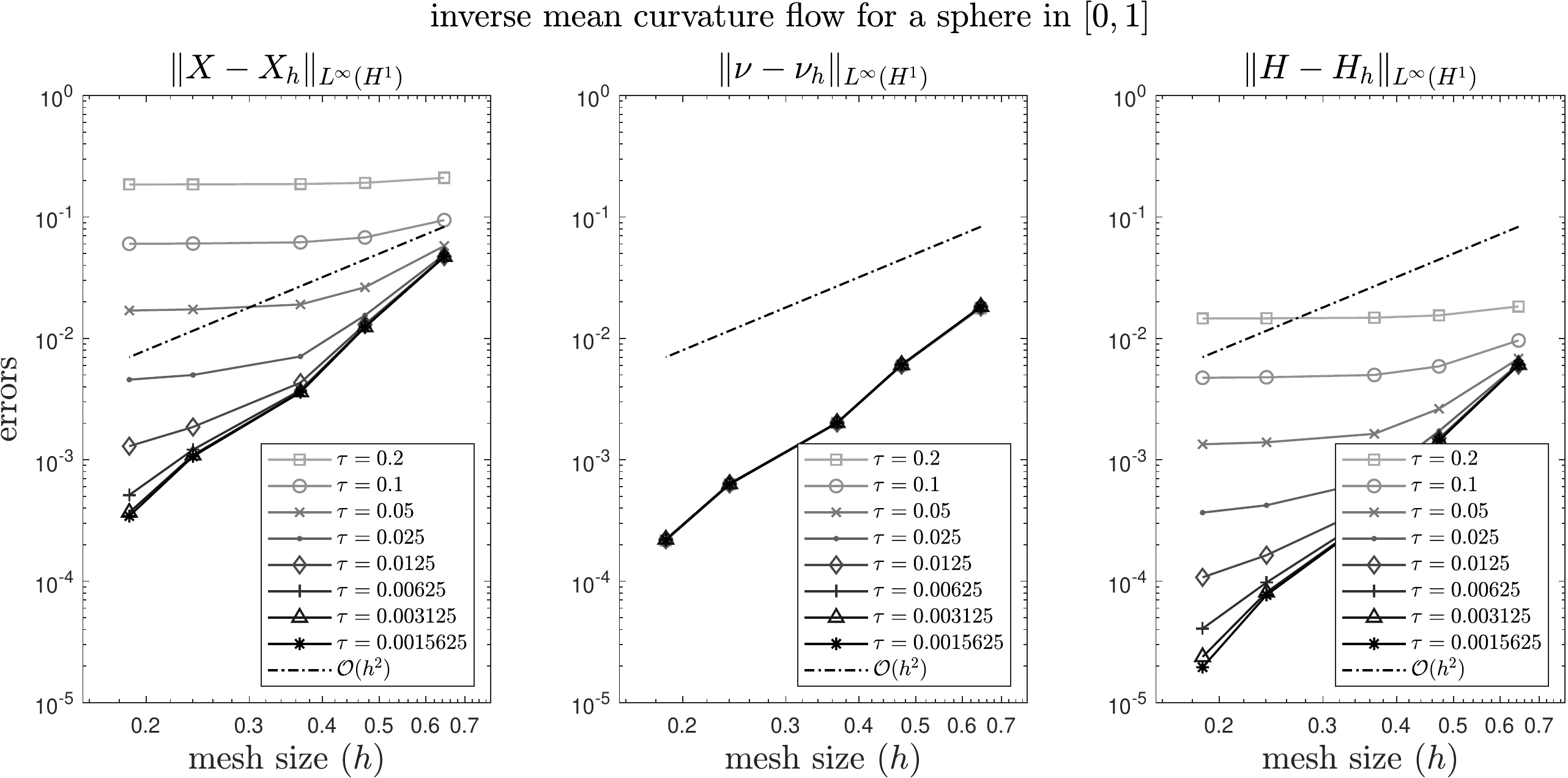}
	\\ \vspace{3mm}
	\includegraphics[trim={0 0 0 25},clip, width=0.9\textwidth]{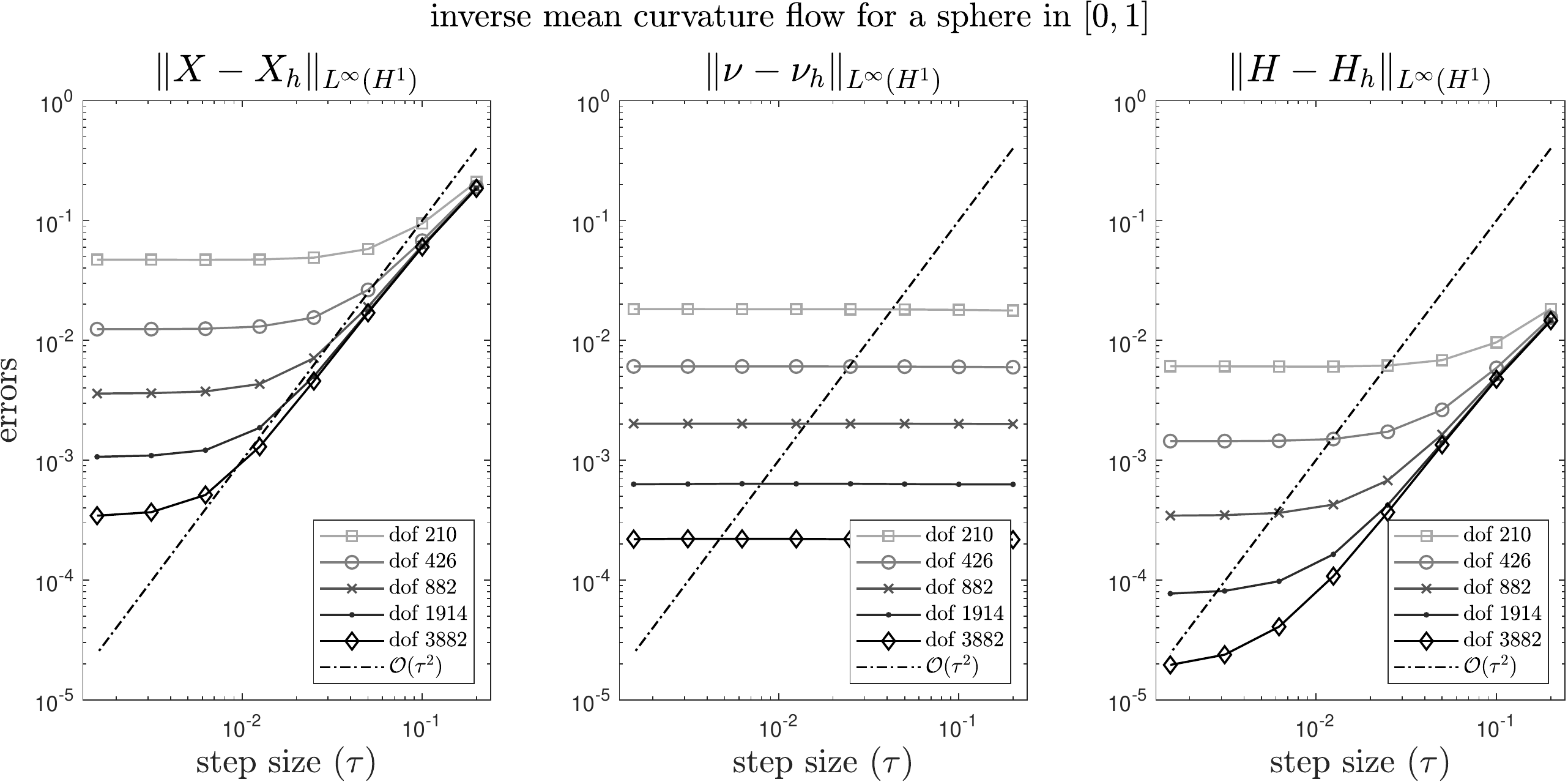}
	\caption{Spatial and temporal convergence of the BDF2 / quadratic ESFEM discretisation for the inverse mean curvature flow of a sphere for $T = 1$.}
	\label{fig:conv_iMCF}
\end{figure}

\begin{figure}[htbp]
	\includegraphics[width=0.9\textwidth]{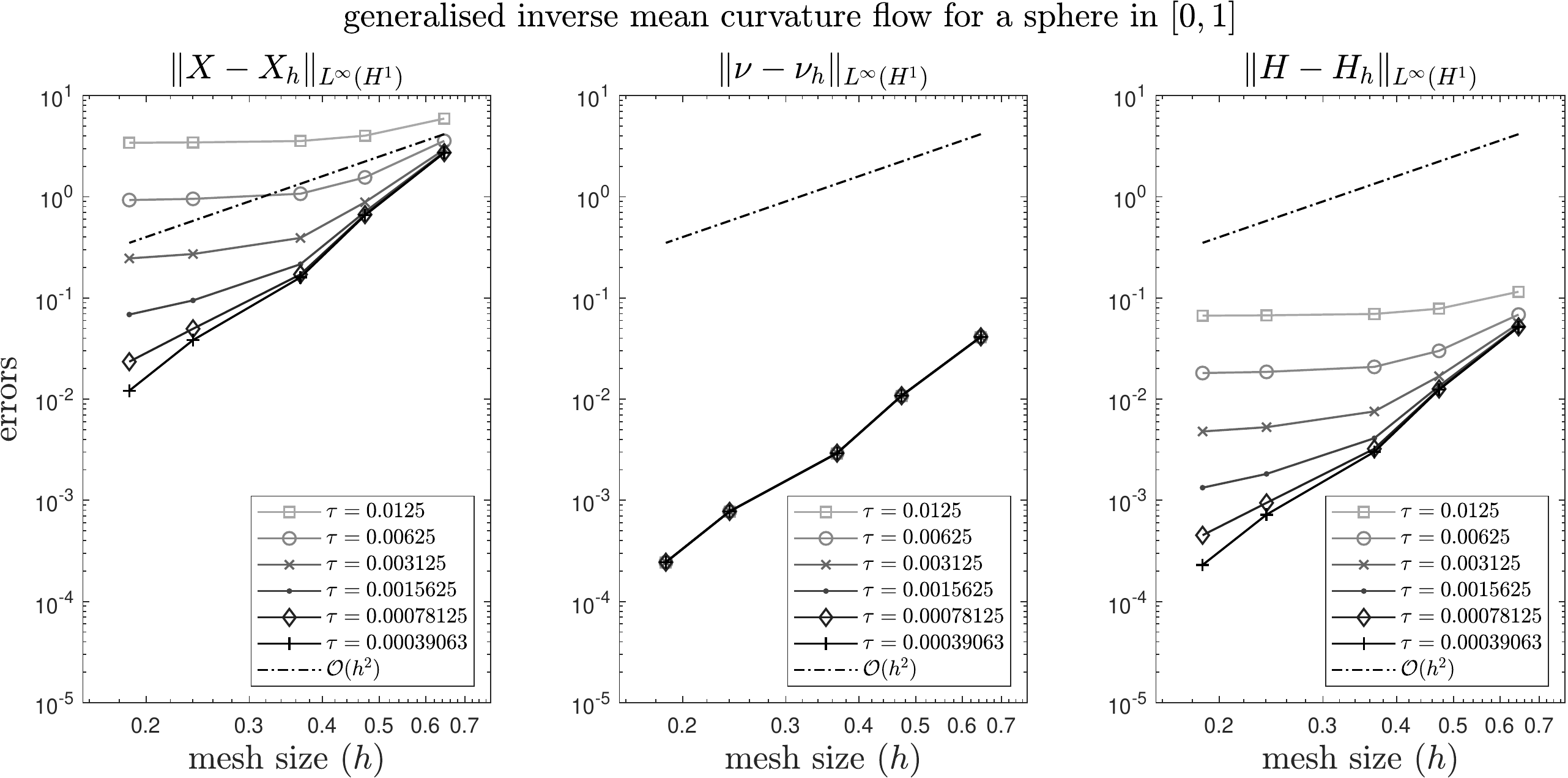}
	\\ \vspace{3mm}
	\includegraphics[trim={0 0 0 25},clip, width=0.9\textwidth]{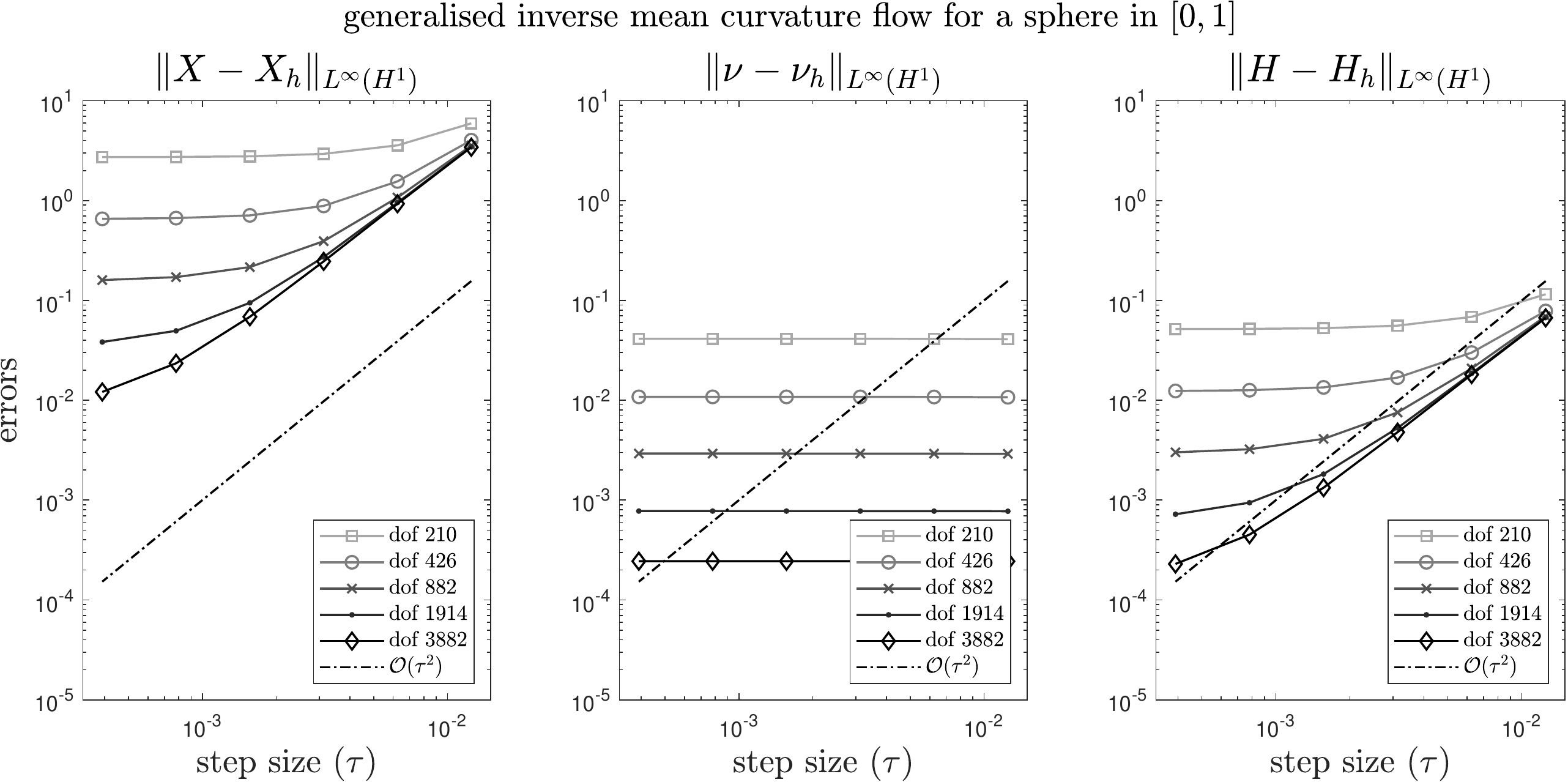}
	\caption{Spatial and temporal convergence of the BDF2 / quadratic ESFEM discretisation for generalized inverse mean curvature flow \bbk $\alpha=2$ \ebk of a sphere for $T = 1$.}
	\label{fig:conv_iMCFgen}
\end{figure}

\begin{figure}[htbp]
	\includegraphics[width=0.9\textwidth]{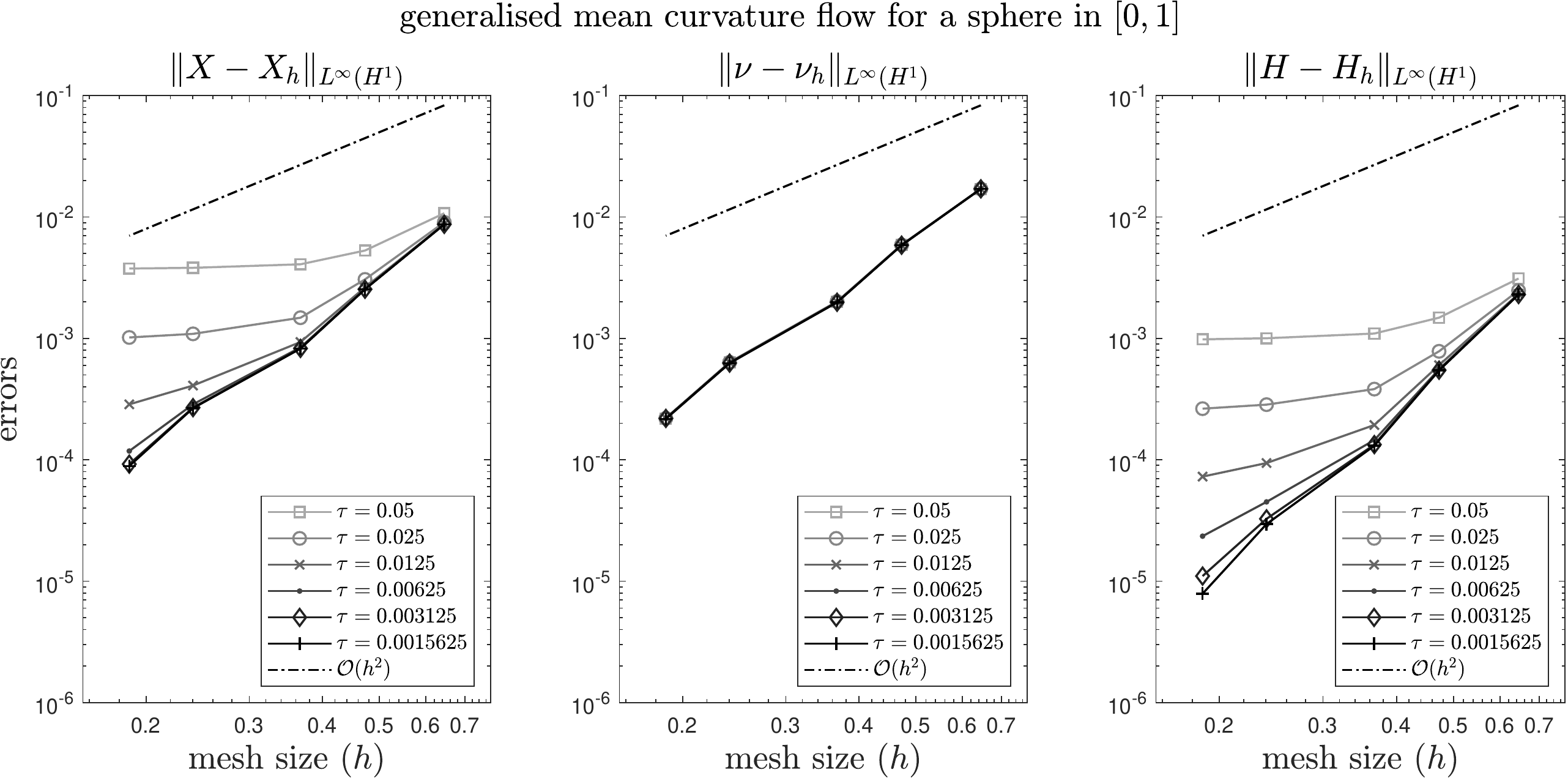}
	\\ \vspace{3mm}
	\includegraphics[trim={0 0 0 25},clip, width=0.9\textwidth]{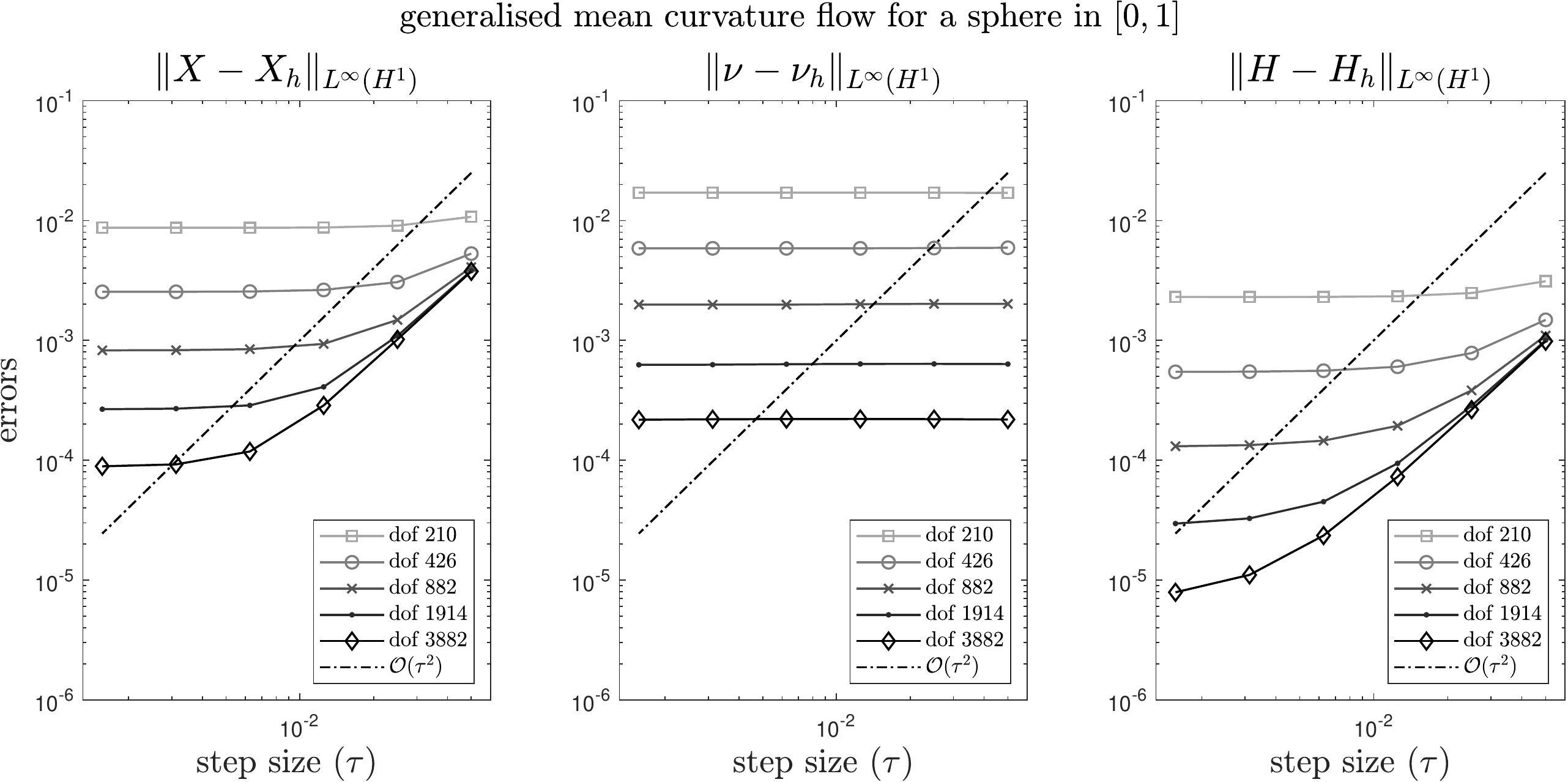}
	\caption{Spatial and temporal convergence of the BDF2 / quadratic ESFEM discretisation for generalized mean curvature flow \bbk $\alpha=2$ \ebk of a sphere for $T = 1$.}
	\label{fig:conv_MCFgen}
\end{figure}

\subsection{Monotone geometric quantities}

We report on surface evolutions under various generalized mean curvature flows (in two dimensions), and on the time-evolution of corresponding monotone geometric quantities for surfaces with \emph{non-negative} mean curvature. Such quantities and their monotonicity are often used in analysis to prove various results (e.g.~convergence to a round point for $H^\alpha$-flows), see~\cite[Section~6]{HuiskenPolden}, \cite[Appendix~A]{Schulze_2}, and in particular \cite{Schnurer}, where such quantities are found by the aid of a randomized algorithmic test.

The fact that the algorithm analysed in this paper preserves the monotonicity of such quantities is of interest both from an analytical and a numerical viewpoint.

\medskip
For \emph{inverse mean curvature flow} an important monotone quantity is the Hawking mass, see \cite{Hawking}, or \cite[Section~6]{HuiskenPolden}:
\begin{equation}
\label{eq:Hawking mass}
m_{\text{H}}(\Ga[X]) = \sqrt{\frac{\text{Area}(\Ga[X])}{16\pi}} \bigg( 1 - \frac{1}{16\pi} \int_{\Ga[X]} H^2 \bigg) ,
\end{equation}
which is non-decreasing in time, i.e.~$\tfrac{\d}{\d t} m_{\text{H}}(\Ga[X]) \geq 0$.

For \emph{generalized mean curvature flow} ($H^\alpha$-flow), for $1 \leq \alpha \leq 5$, an important monotone quantity is the following \cite[Appendix~A]{Schulze_2}, (also expressed using the principal curvatures $\kappa_j$):
\begin{equation}
\label{eq:MCFgen monotone quantity}
\begin{aligned}
m_{\text{$H^\alpha$-flow}}(\Ga[X]) = &\ \max_{\Ga[X]} \frac{ H^{2\alpha} ( 2|A|^2 - H^2 ) }{ ( H^2 - |A|^2 )^2 } \\
= &\ \max_{\Ga[X]} \frac{ (\kappa_1 + \kappa_2)^{2\alpha} (\kappa_1 - \kappa_2)^2 }{ 4 \kappa_1^2 \kappa_2^2 } ,
\end{aligned}
\end{equation}
which is non-increasing in time, i.e.~$\tfrac{\d}{\d t} m_{\text{$H^\alpha$-flow}}(\Ga[X]) \leq 0$.

Let us point out that both monotone quantities can be very easily computed from the geometric quantities obtained from our algorithm.

In Figure~\ref{fig:monotone} we report on the numerical solution and the corresponding monotone quantity for three different generalized flows: inverse mean curvature flow (iMCF), powers of mean curvature flow ($H^\alpha$-flow) with $\alpha = 2$ and $6$ (plotted left to right in the figure). In the experiments we have used time step size $\tau = 0.0125$, and surfaces with degrees of freedom $6242$, $4210$, and $4242$ respectively.

For inverse mean curvature flow and for the $H^\alpha$-flow with $\alpha = 2$ the  quantities \eqref{eq:Hawking mass} and \eqref{eq:MCFgen monotone quantity}, respectively, are known to be monotone (non-decreasing and non-increasing). For $\alpha = 6$, to our knowledge, such a result is an open question. This numerical evidence suggest that it is not monotone. However, we strongly note here, that the small jump at $t \approx 0.25$ is \bbk probably \ebk a numerical artefact, which (in our experience) is due to violations of discrete maximum principles \bbk (the computed normal velocity becomes non-positive, and hence inverting $V_h = H_h^6$ is problematic, in such cases $V_h$ was appropriately truncated). The same maximum principle violation occurs for smaller time step sizes. The phenomena should be investigated further in subsequent numerical experiments. \ebk The algorithm could benefit from applying the techniques of \cite{ESFEM_DMP}.

\begin{figure}[htbp]
	\includegraphics[trim={5 40 50 70},clip,width=0.31\textwidth]{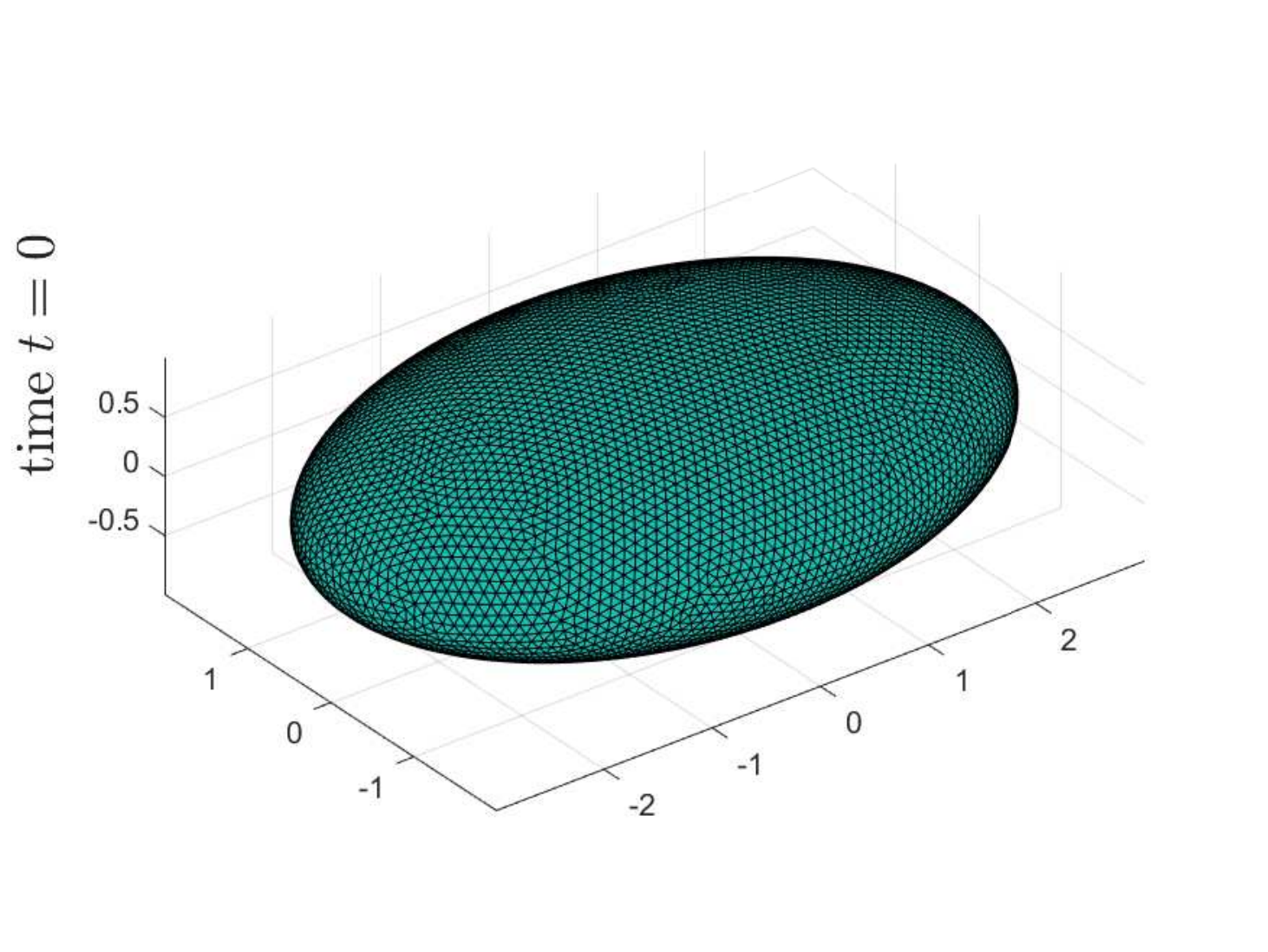} 
	\includegraphics[trim={12 37 40 70},clip,width=0.31\textwidth]{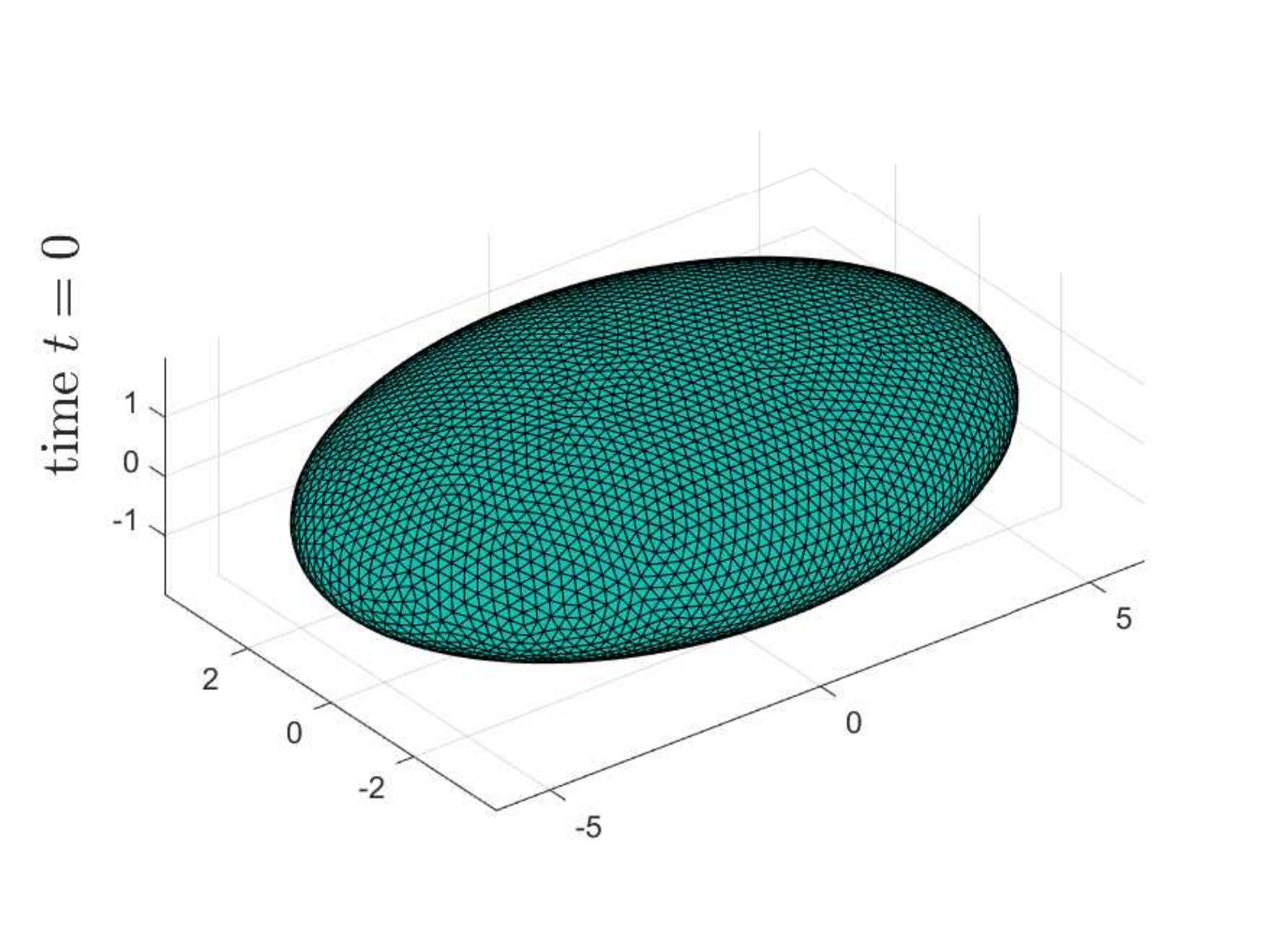} 
	\includegraphics[trim={50 20 49 50},clip,width=0.32\textwidth]{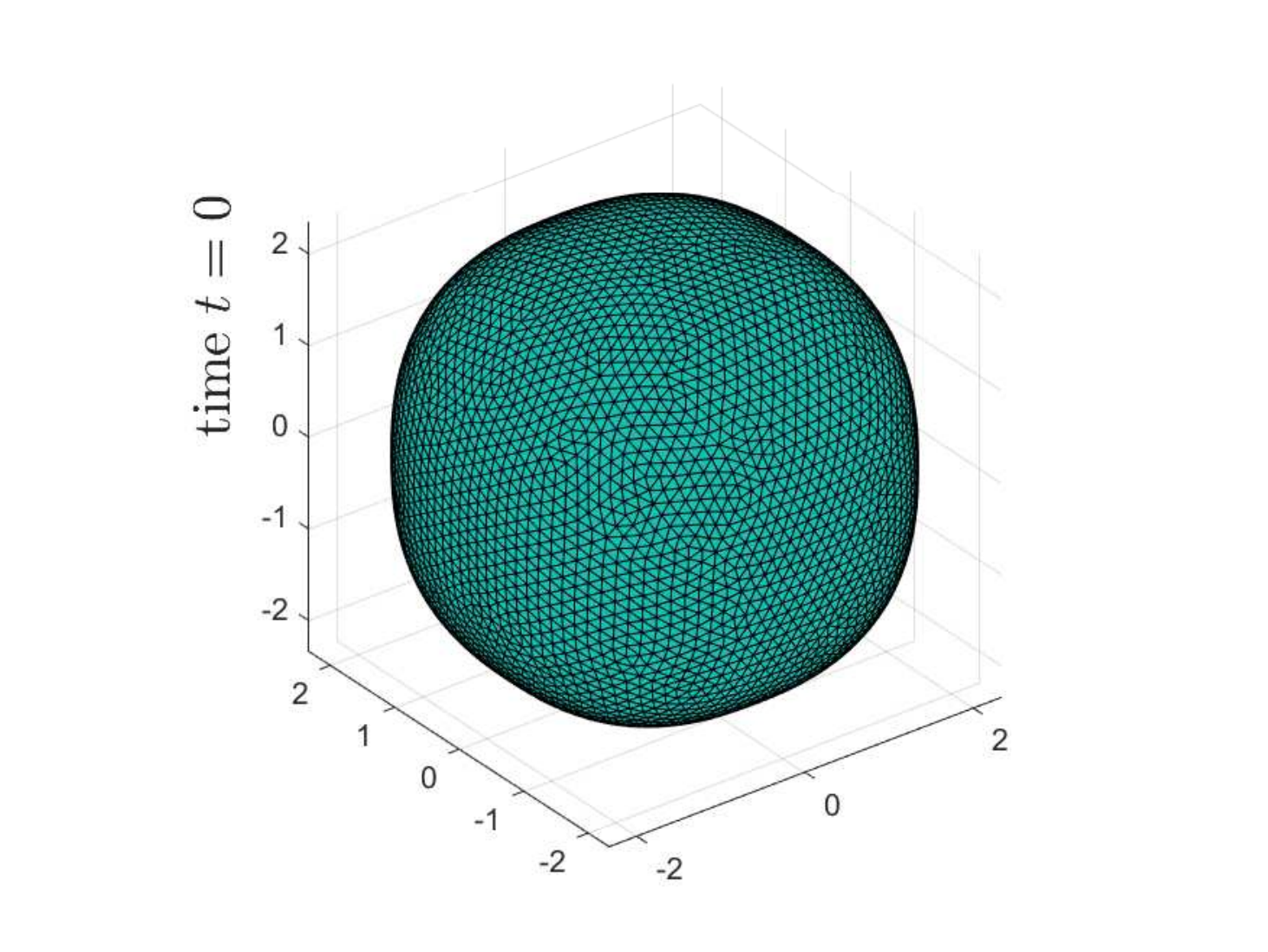} 
	\\
	\includegraphics[trim={10 20 50 60},clip,width=0.32\textwidth]{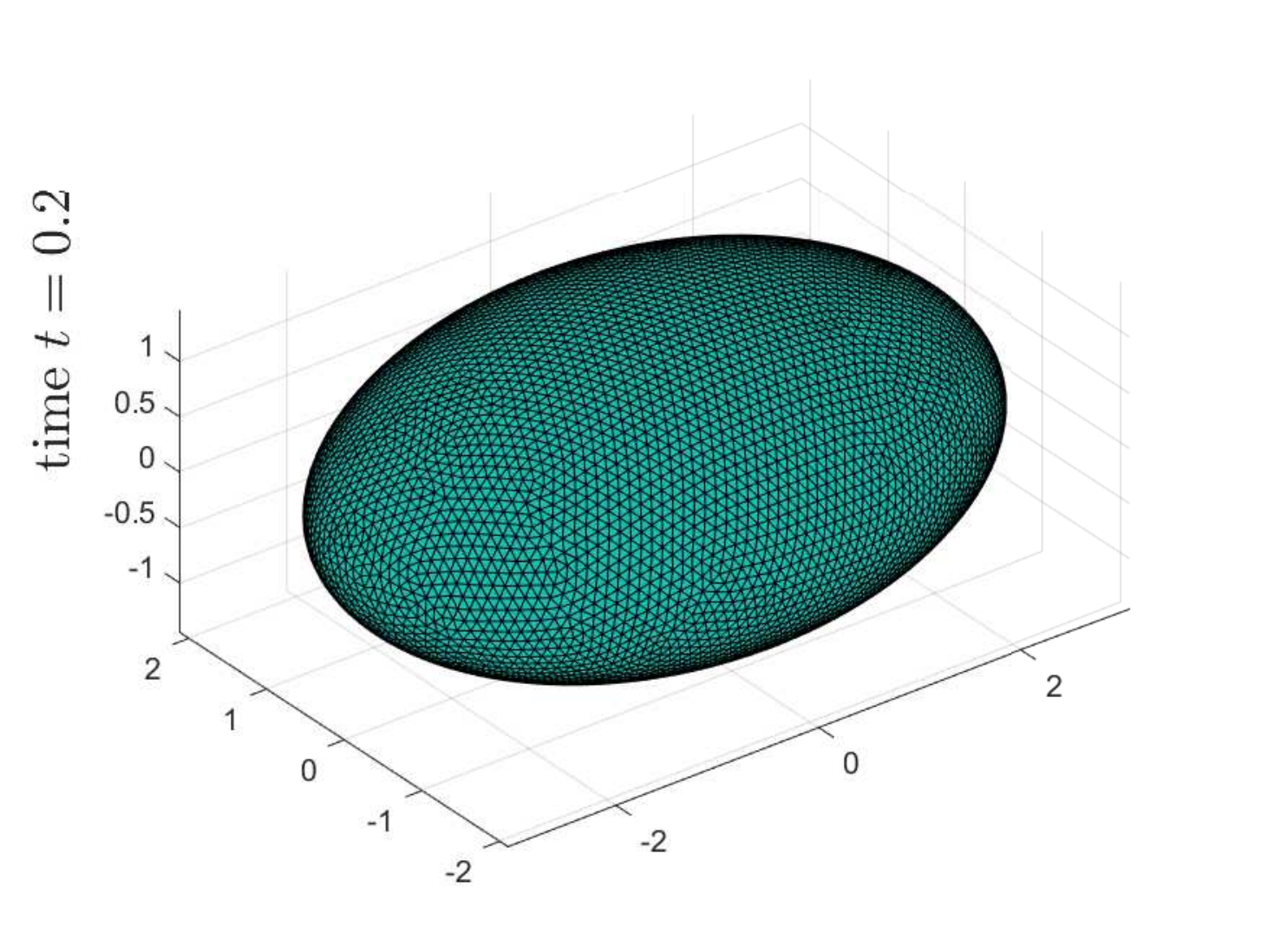} 
	\includegraphics[trim={12 30 35 60},clip,width=0.32\textwidth]{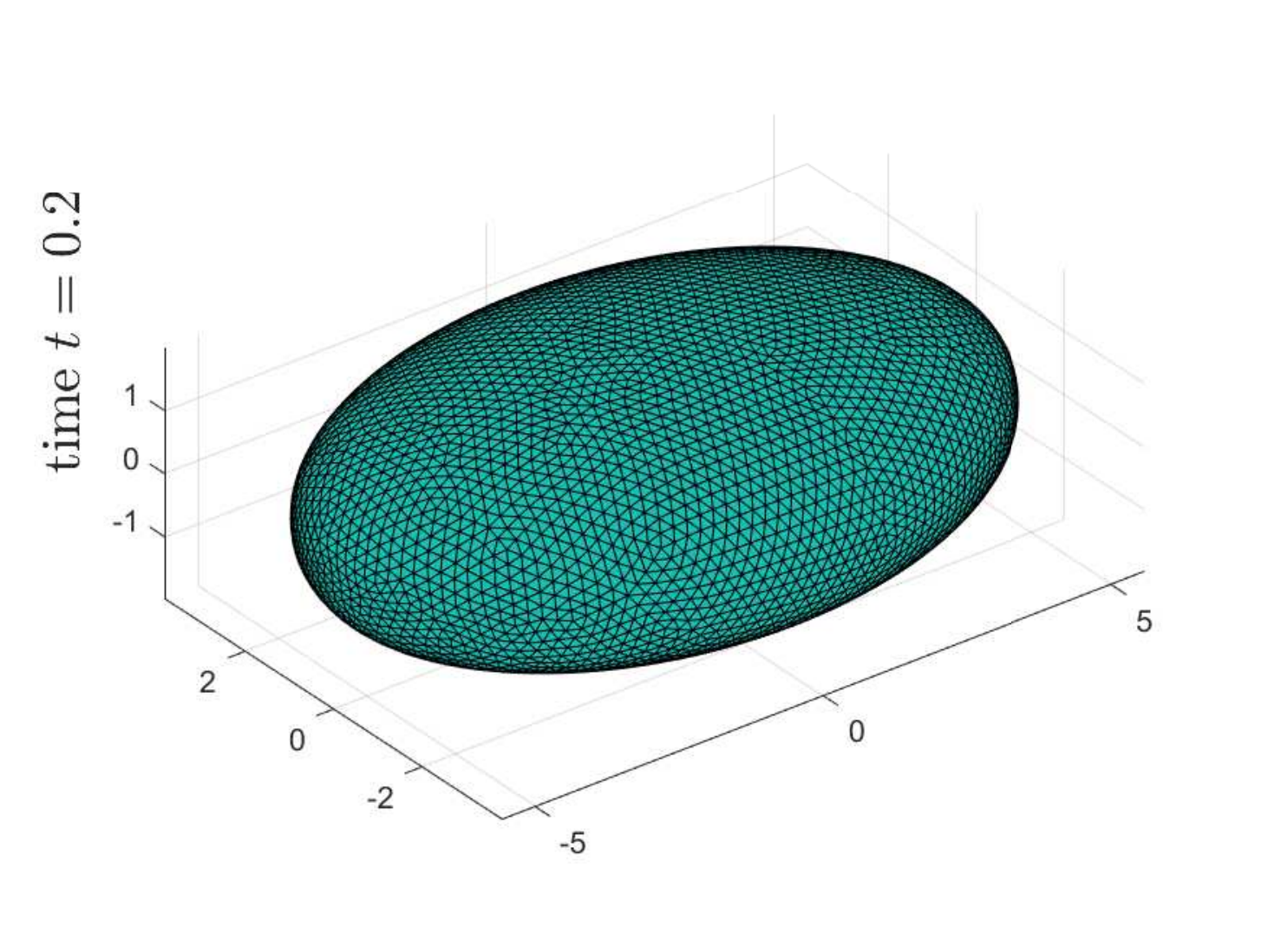} 
	\includegraphics[trim={50 20 50 50},clip,width=0.32\textwidth]{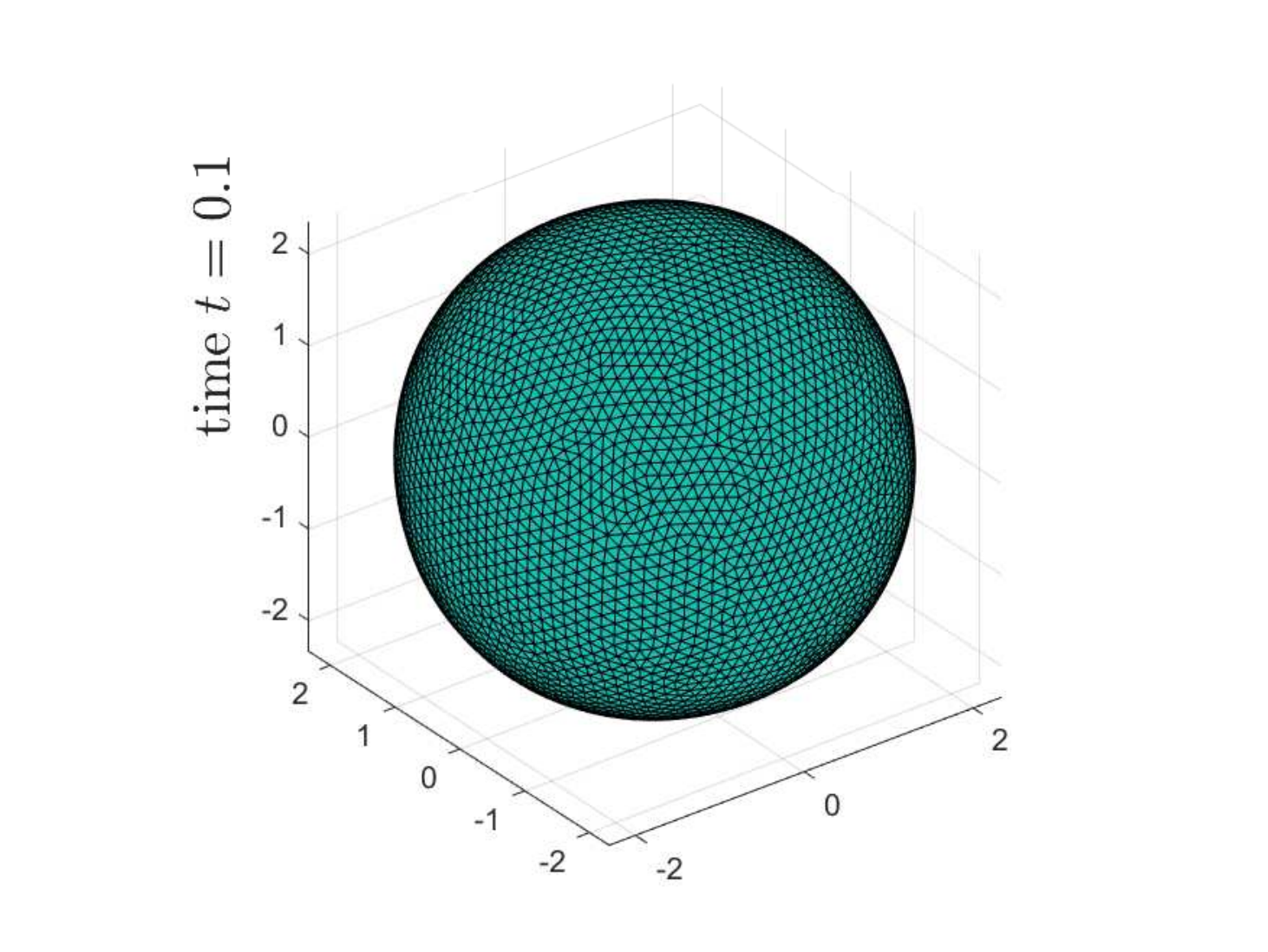} 
	\\
	\includegraphics[trim={30 20 50 50},clip,width=0.32\textwidth]{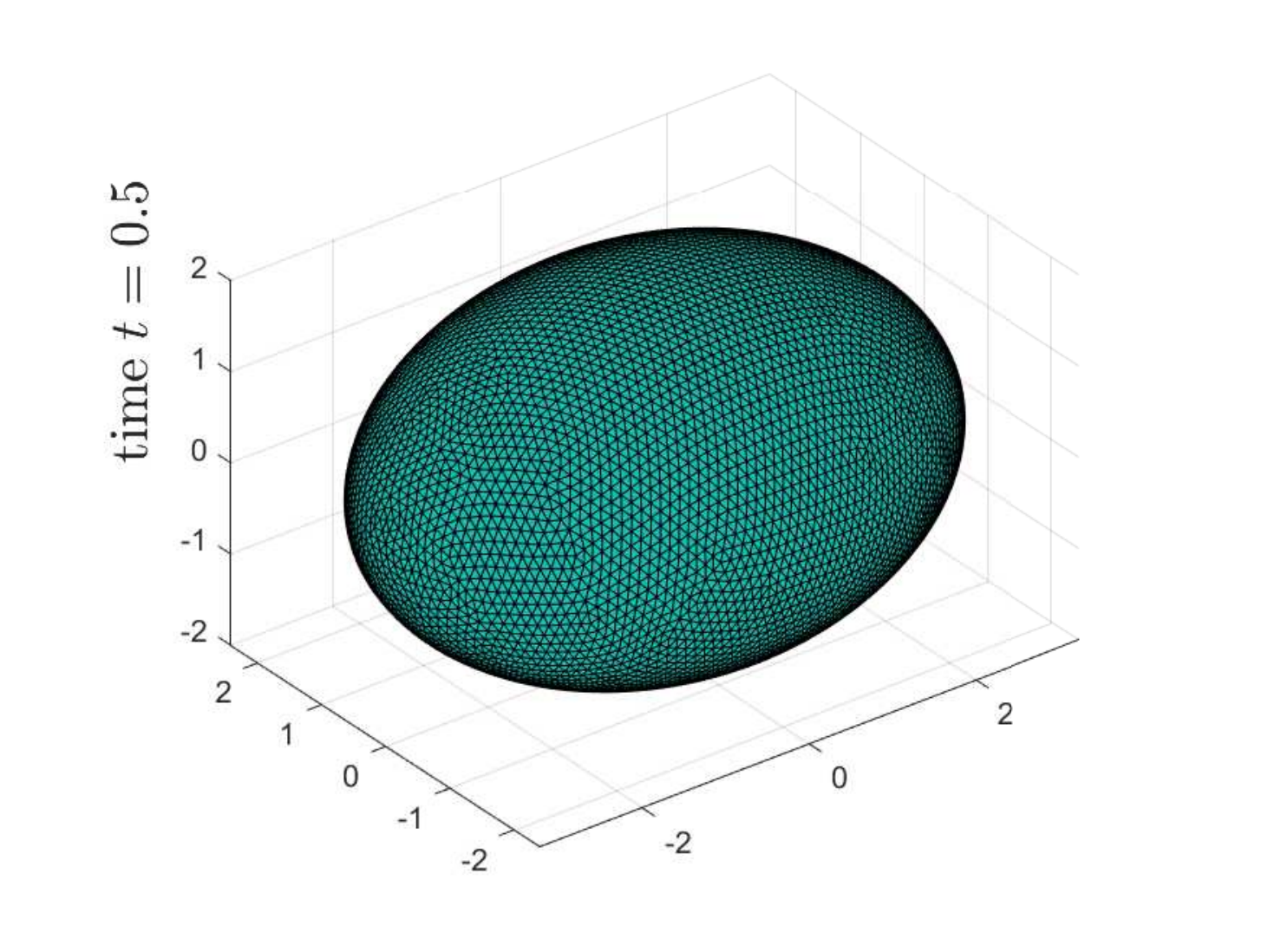} 
	\includegraphics[trim={15 27 35 60},clip,width=0.32\textwidth]{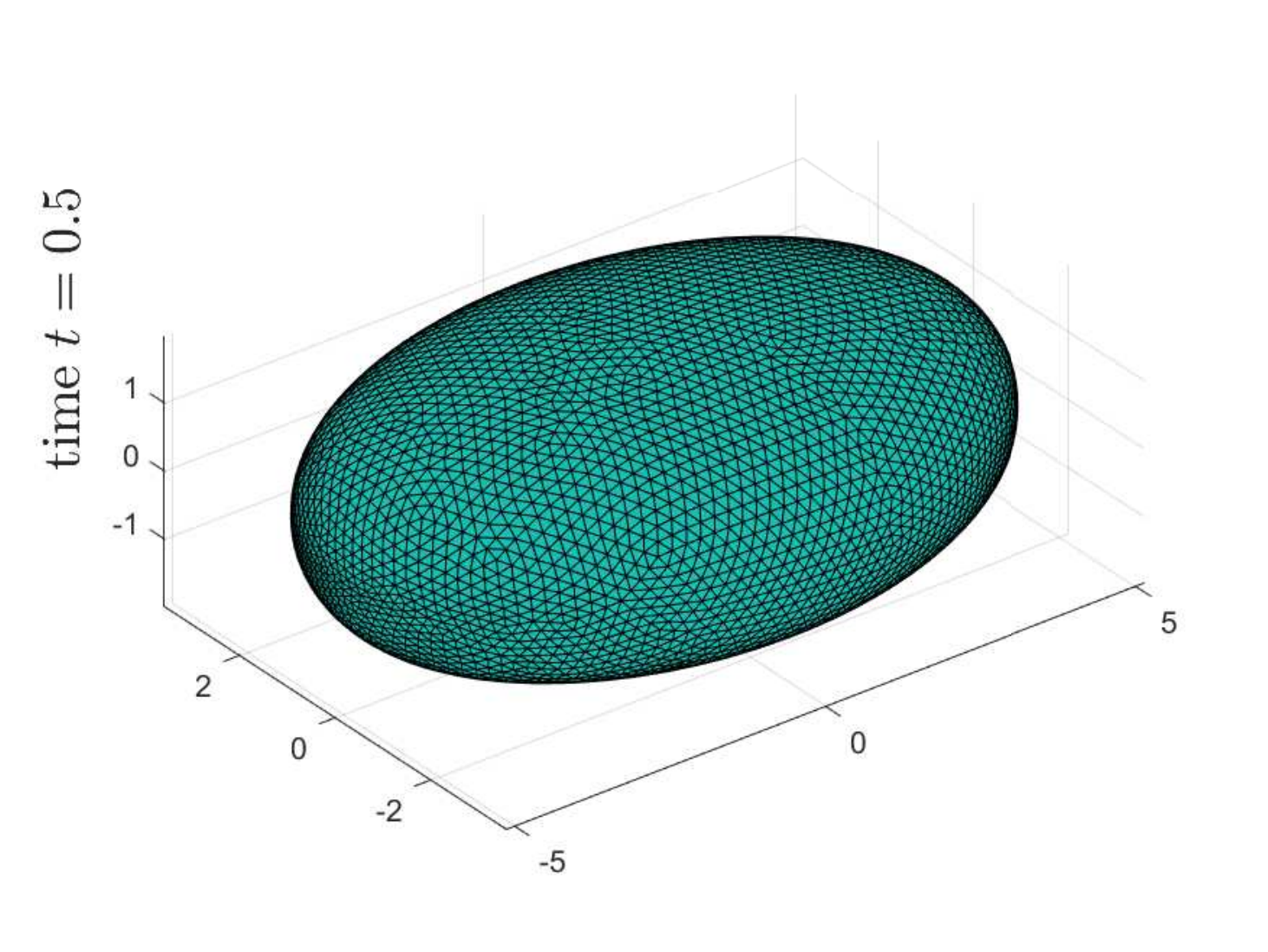} 
	\includegraphics[trim={50 20 50 50},clip,width=0.32\textwidth]{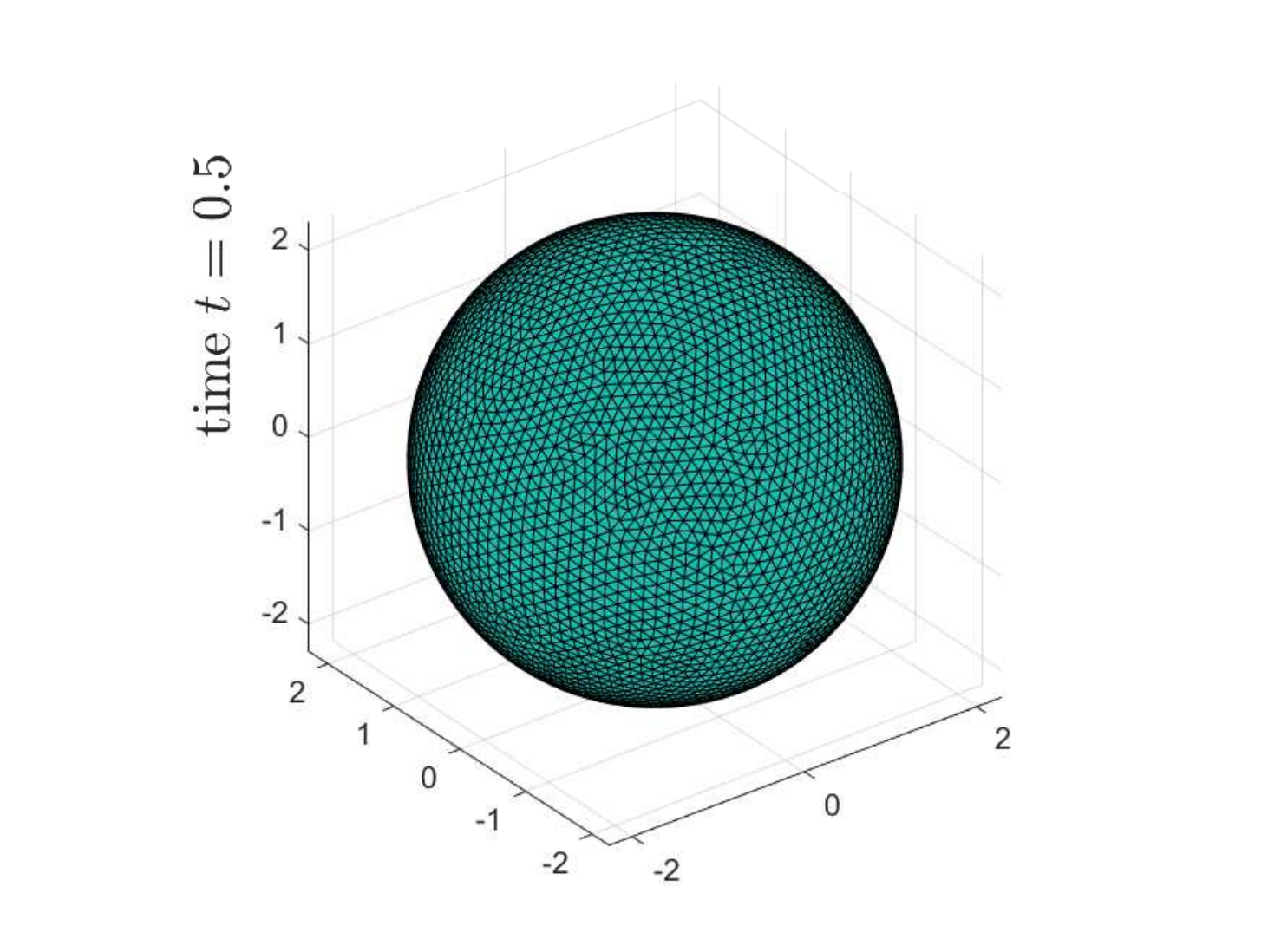} 
	\\
	\includegraphics[trim={50 20 50 50},clip,width=0.32\textwidth]{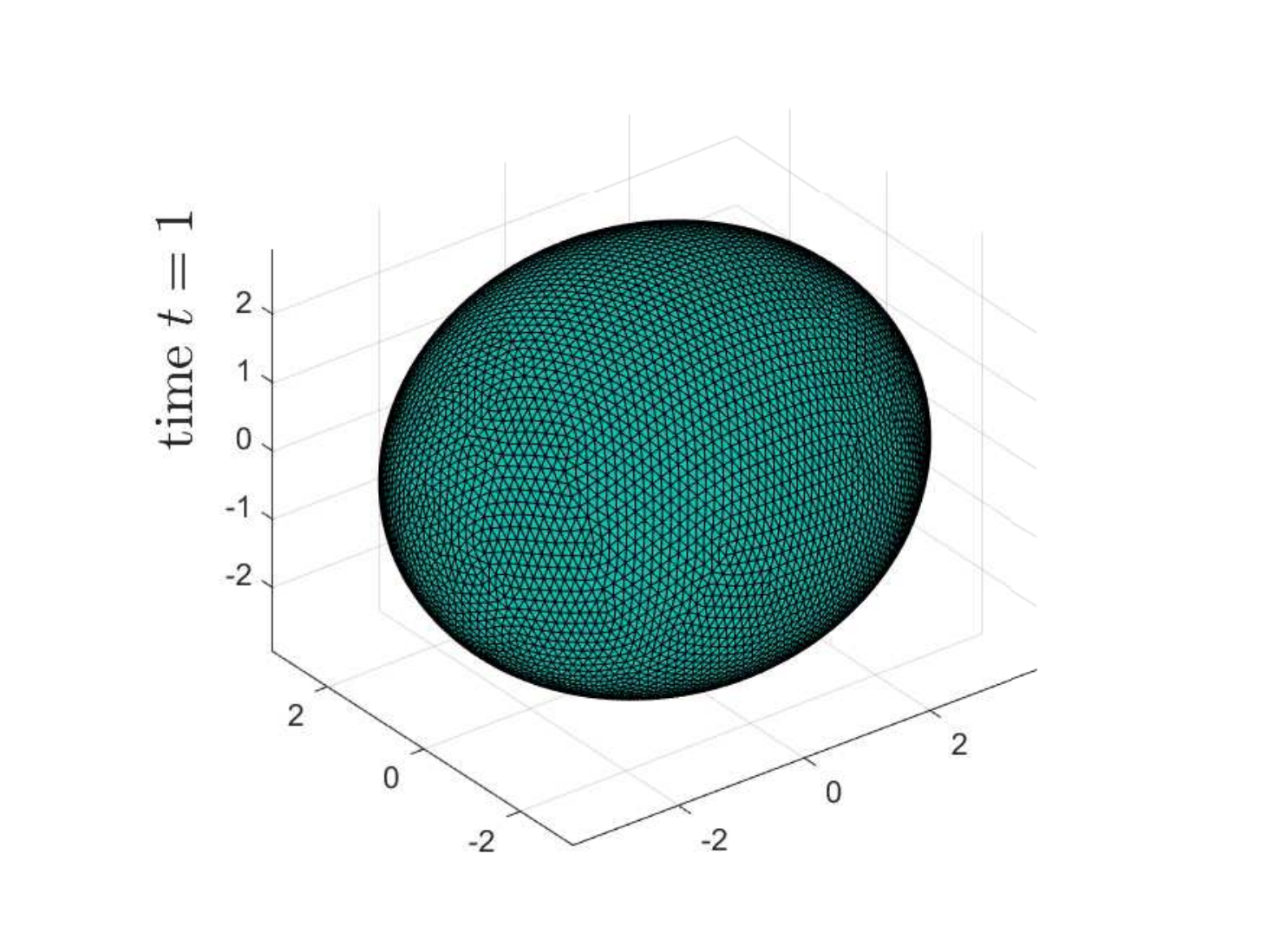} 
	\includegraphics[trim={15 20 35 70},clip,width=0.32\textwidth]{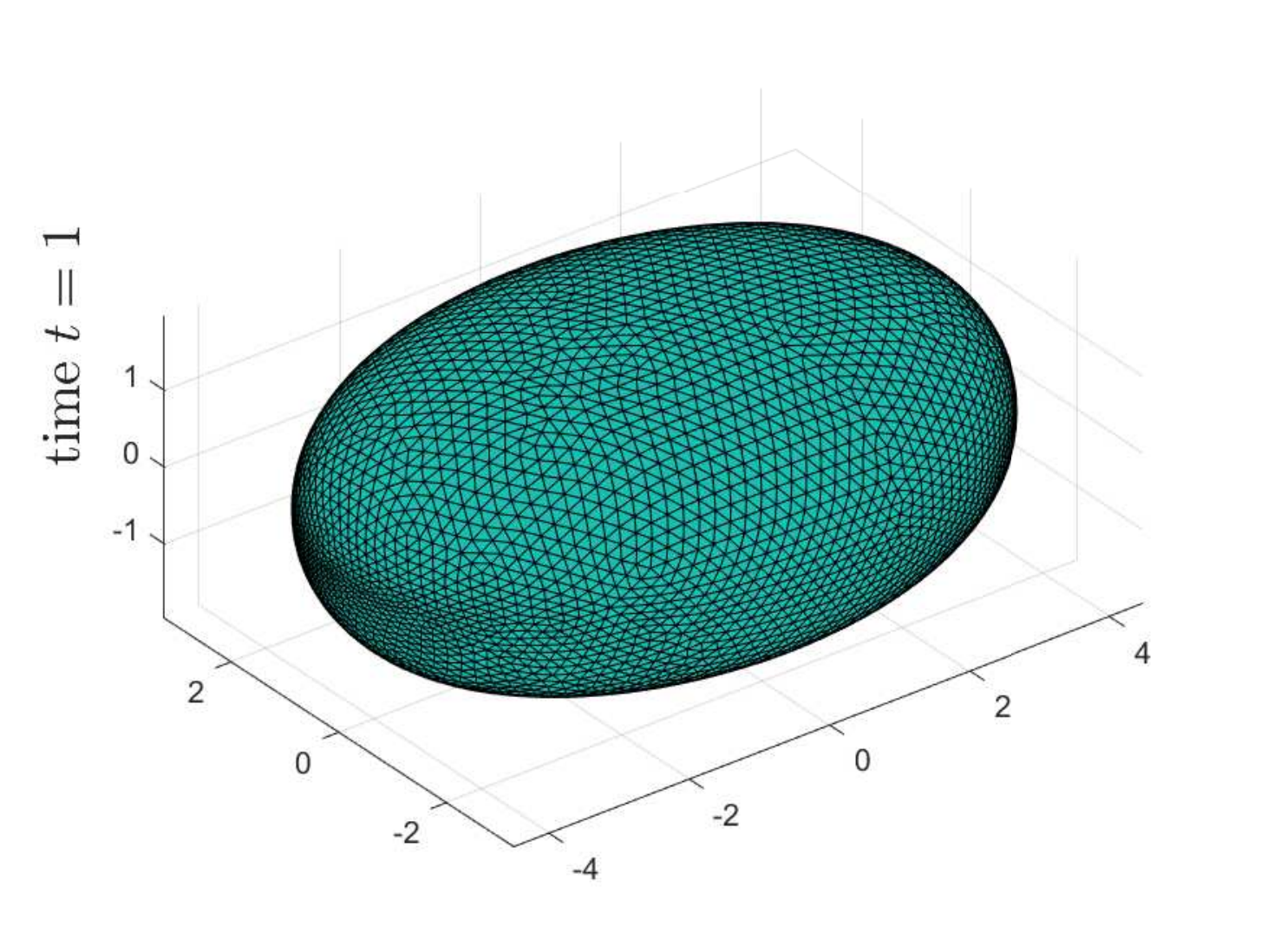} 
	\includegraphics[trim={50 20 50 50},clip,width=0.32\textwidth]{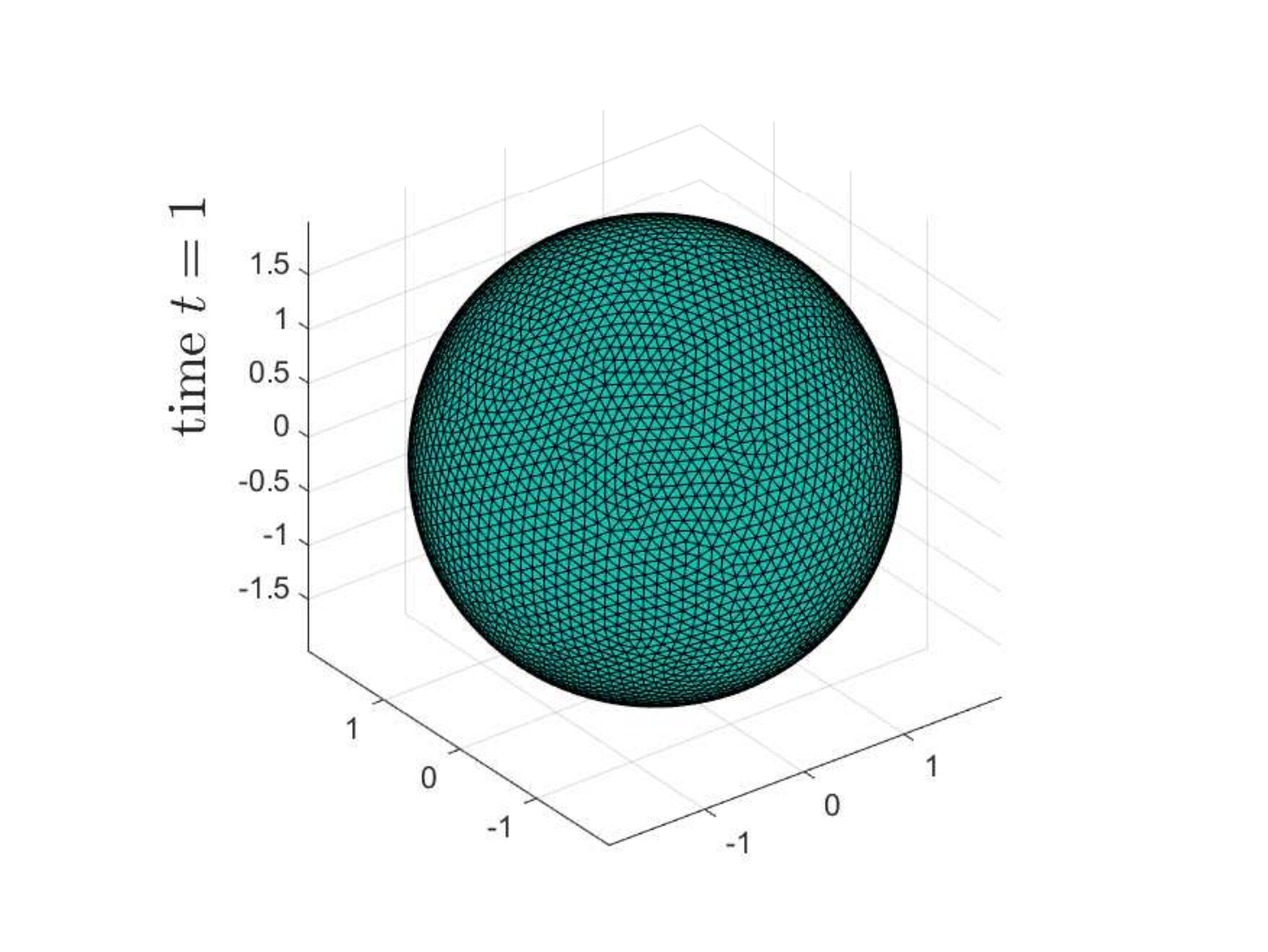} 
	\\
	\includegraphics[width=0.3\textwidth]{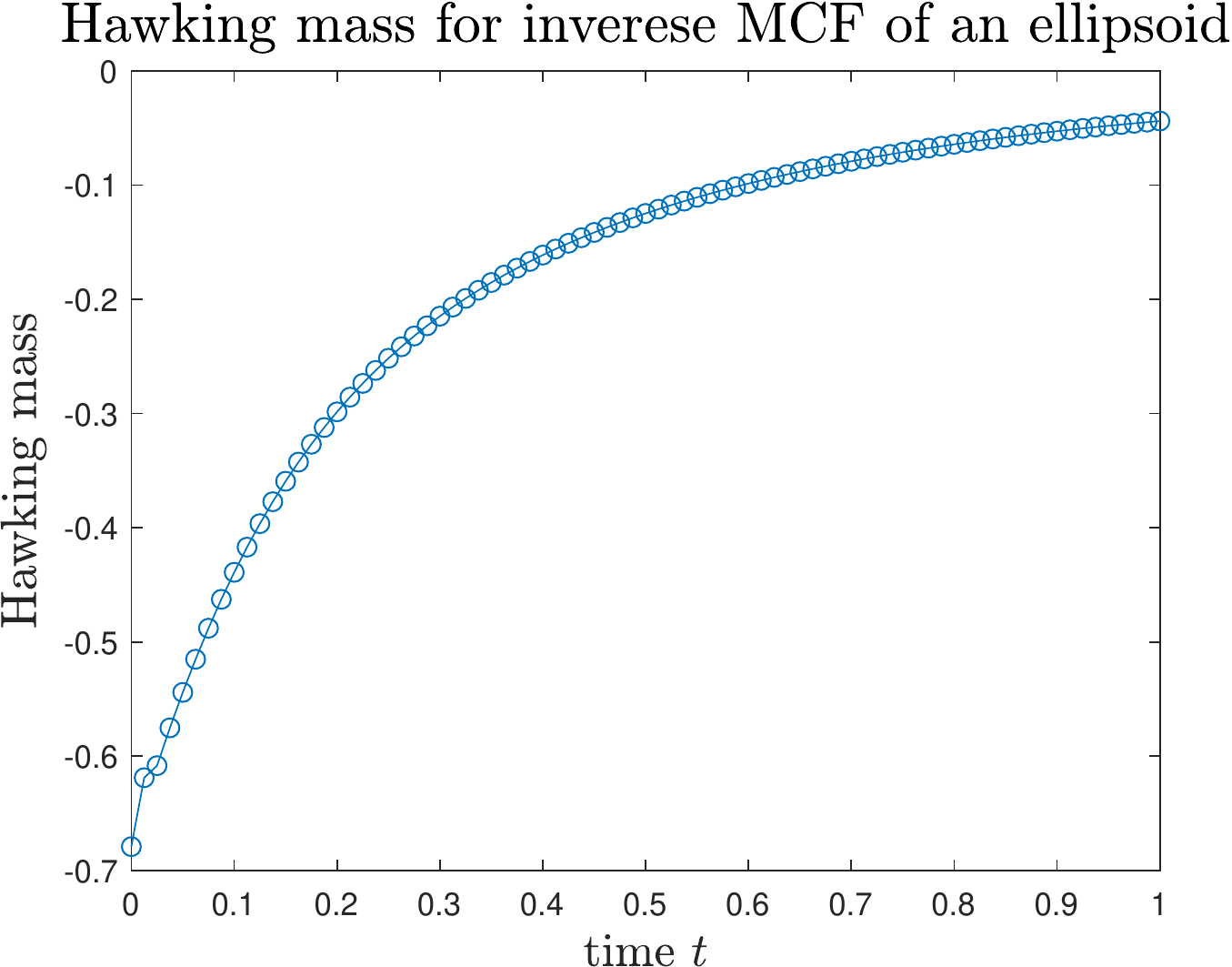}
	\includegraphics[width=0.3\textwidth]{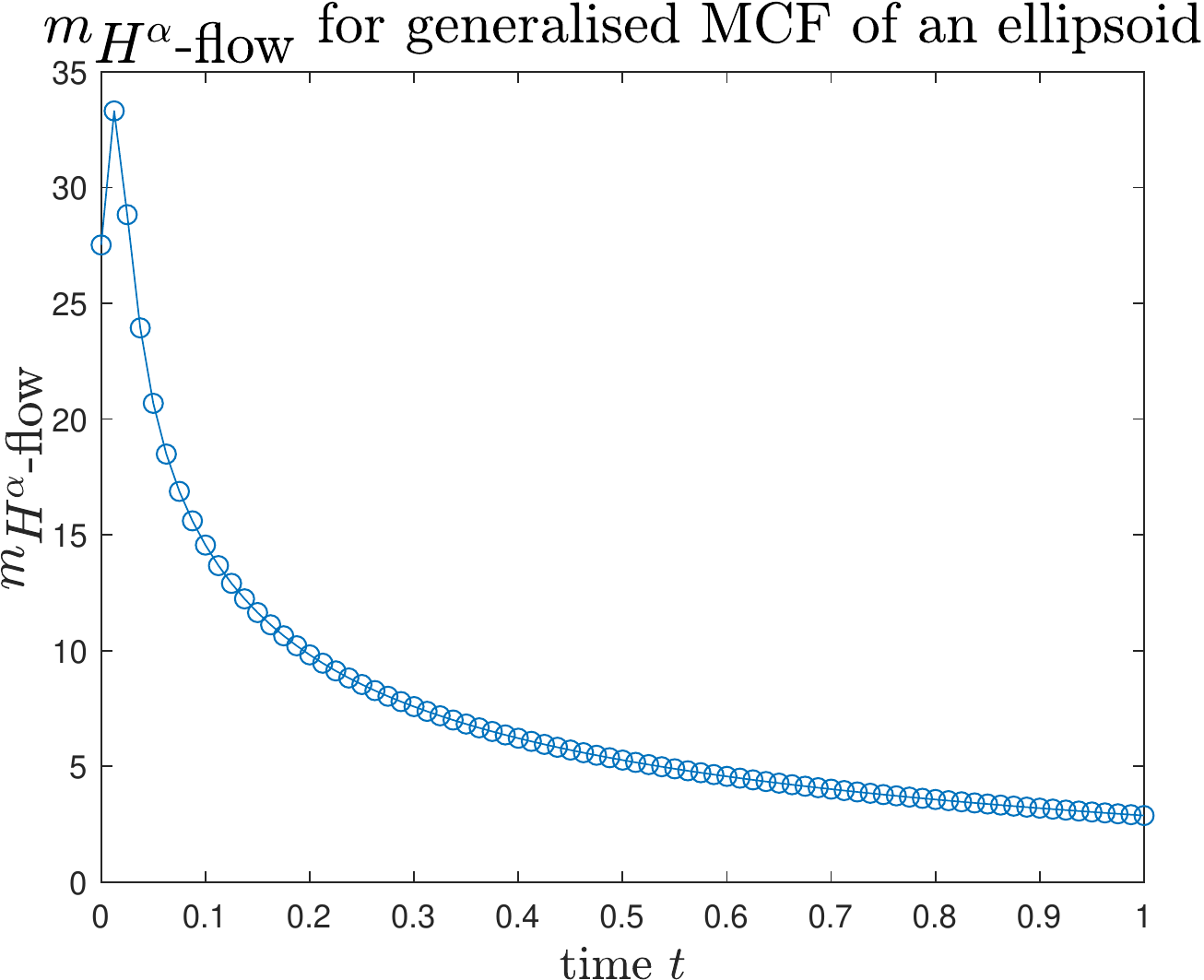}
	\includegraphics[width=0.3\textwidth]{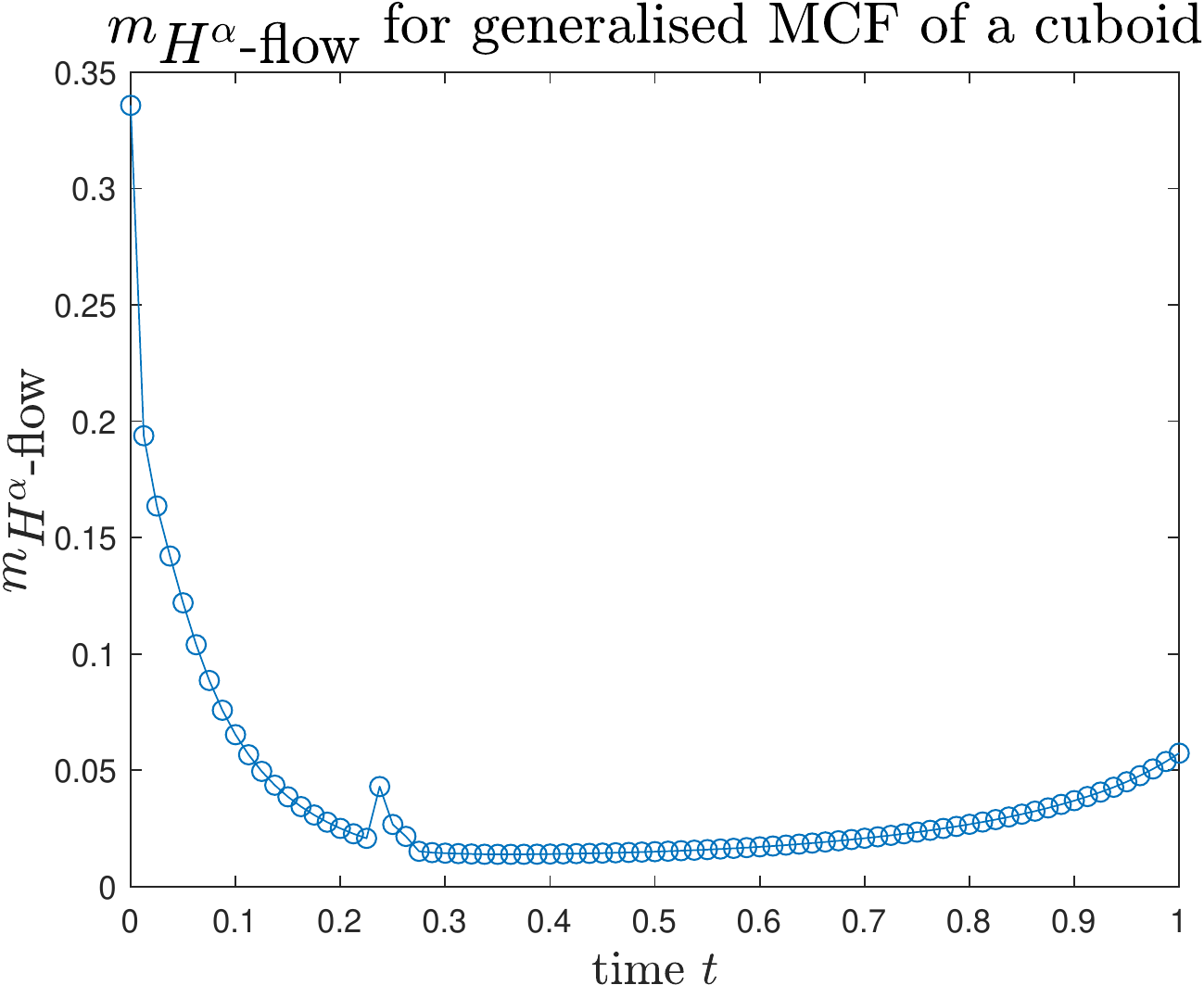}
	\caption{Surface evolutions (from left to right, inverse mean curvature flow, $H^2$-flow, and $H^6$-flow) at different times (from top to bottom) and corresponding monotone quantities (bottom row)}
	\label{fig:monotone}
\end{figure}

\subsection{Generalized mean curvature flows of non-convex surfaces}

To complement our theoretical results, a numerical experiment is presented here for generalized mean curvature flows of two \emph{non-convex} initial surfaces $\Ga^0 \subset \R^3$, given by
\begin{equation}
	\begin{aligned}
		\Ga^0 = &\ \Big\{ x \in \R^3 \mid x_1^2 + x_2^2 + 2 x_3^2 \big(x_3^2 - 159/200\big)  - 0.5 \bbk = 0 \ebk \Big\} , \qquad \text{cf.~\cite{ElliottStyles_ALEnumerics}}, \\
		\text{and } \qquad \Ga^0 = &\ \Big\{ x \in \R^3 \mid (x_1^2-1)^2 + (x_2^2-1)^2 + (x_3^2-1)^2 - 1.05  \bbk = 0 \ebk \Big\} .
	\end{aligned}
\end{equation}
These initial surfaces have regions with both negative and positive mean curvature, and hence of a class not covered by our theorems. Nevertheless, our algorithm can be still used to compute numerical solutions of generalized flows.

In Figure~\ref{fig:iMCF_dumbbell}, \ref{fig:MCFgen_dumbbell}, and \ref{fig:iMCFgen_genus5}, we respectively report on the numerical solution to the inverse mean curvature flow and for the $H^\alpha$-flow and $-H^{-\alpha}$-flow (with $\alpha = 2$). For the first two experiments we have used a dumbbell shaped two-dimensional surface with $3538$ nodes, while for the third experiment the genus 5 surface with $11424$ nodes. For all experiments we have used a time step size $\tau = 0.0015625$. 

In the case of inverse mean curvature flow, Figure~\ref{fig:iMCF_dumbbell}, some slight surface distortions can be observed around the neck (where mean curvature switches sign), this is however not observable for the $-H^{-\alpha}$-flow in Figure~\ref{fig:iMCFgen_genus5}. The algorithm performs rather robust for the $H^\alpha$-flow, rapidly shrinking towards a point. Note that the different flows are integrated until different final times, see the figures. In the case of Figure~\ref{fig:MCFgen_dumbbell} also note the rapidly shrinking surface.

\begin{figure}[htbp]
	\includegraphics[trim={60 20 90 20},clip,width=0.31\textwidth]{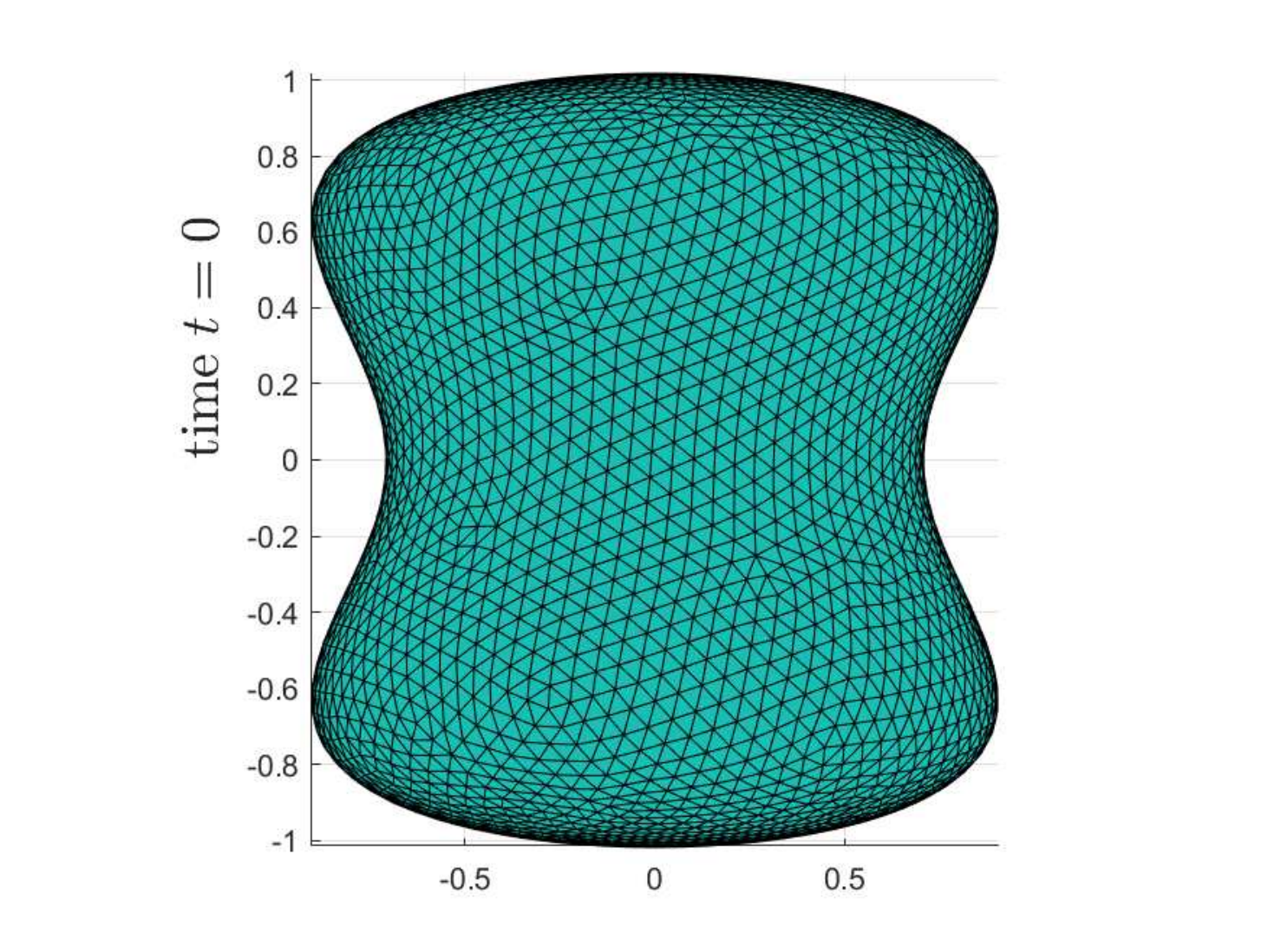} 
	\includegraphics[trim={60 20 90 20},clip,width=0.32\textwidth]{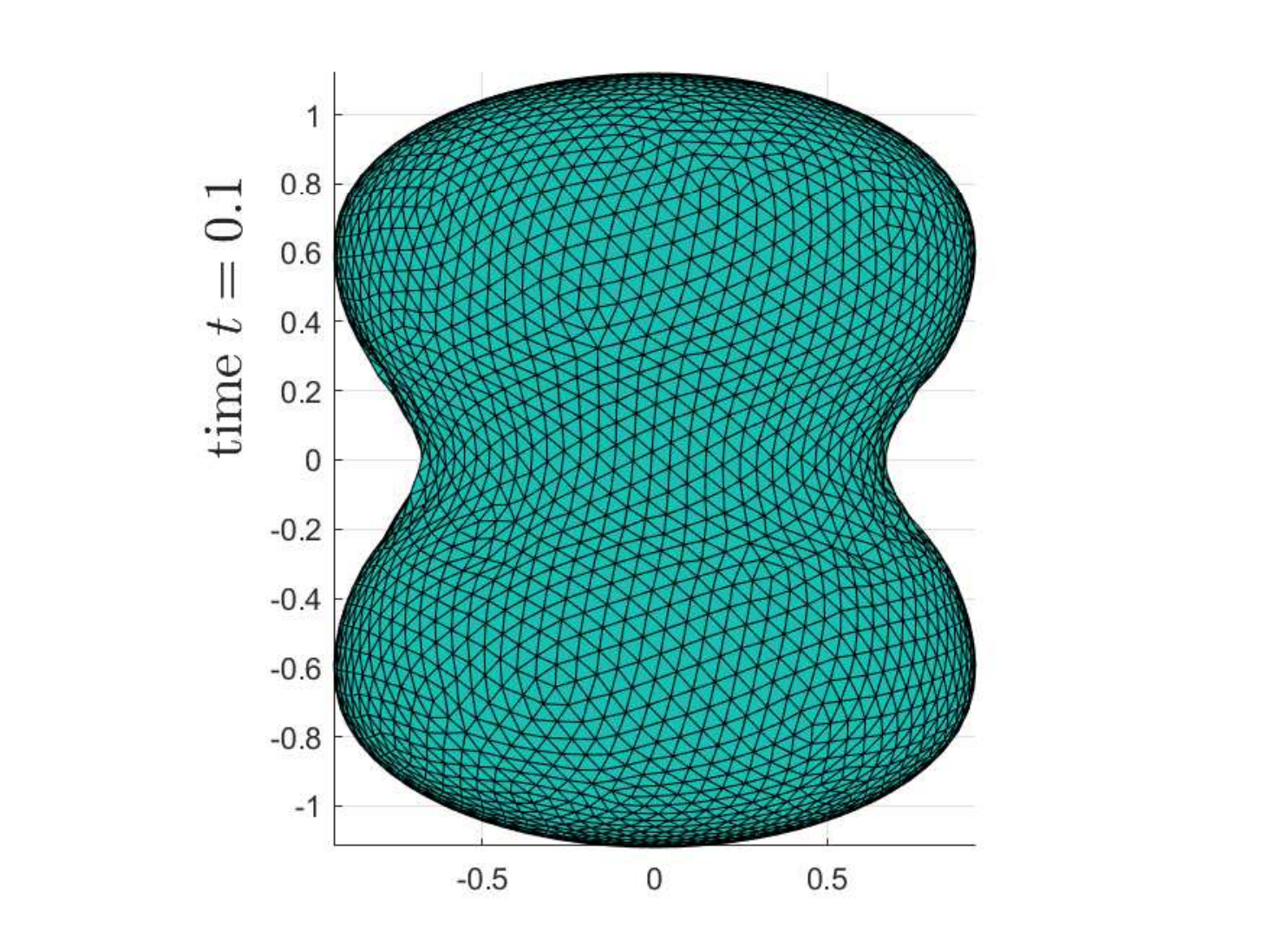} 
	\includegraphics[trim={60 20 90 20},clip,width=0.32\textwidth]{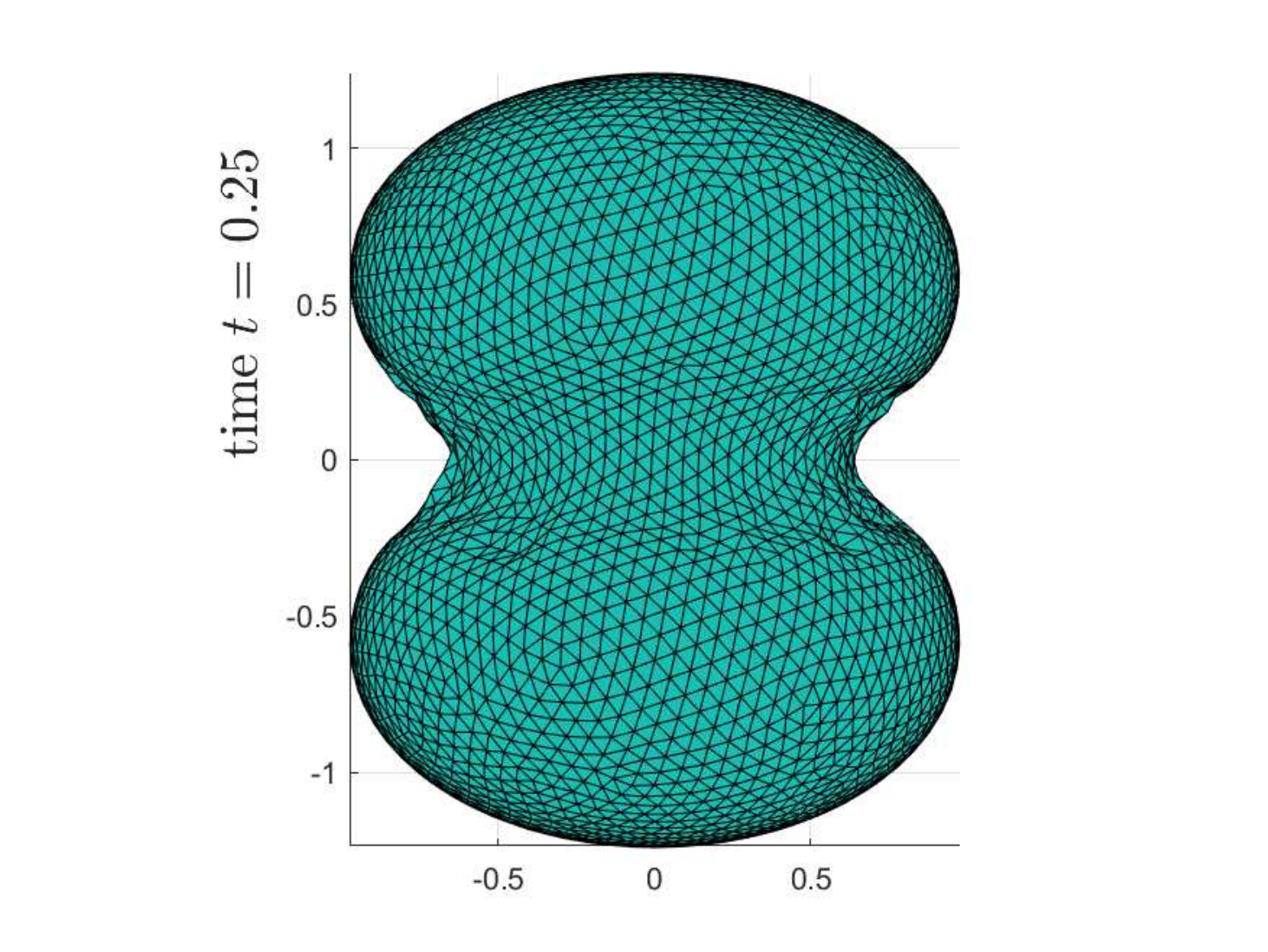} 
	\\
	\includegraphics[trim={60 20 100 20},clip,width=0.32\textwidth]{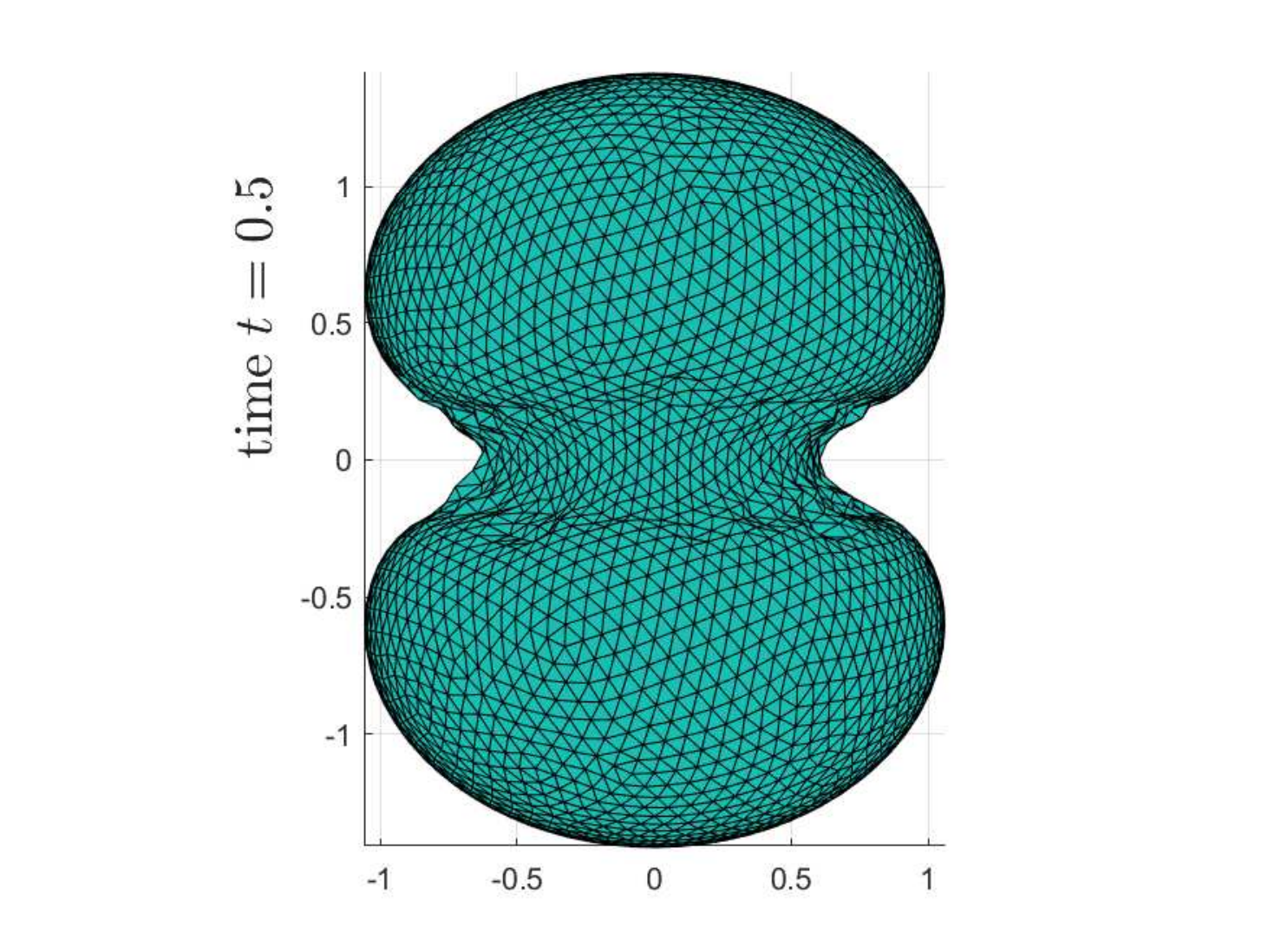} 
	\includegraphics[trim={60 20 110 20},clip,width=0.32\textwidth]{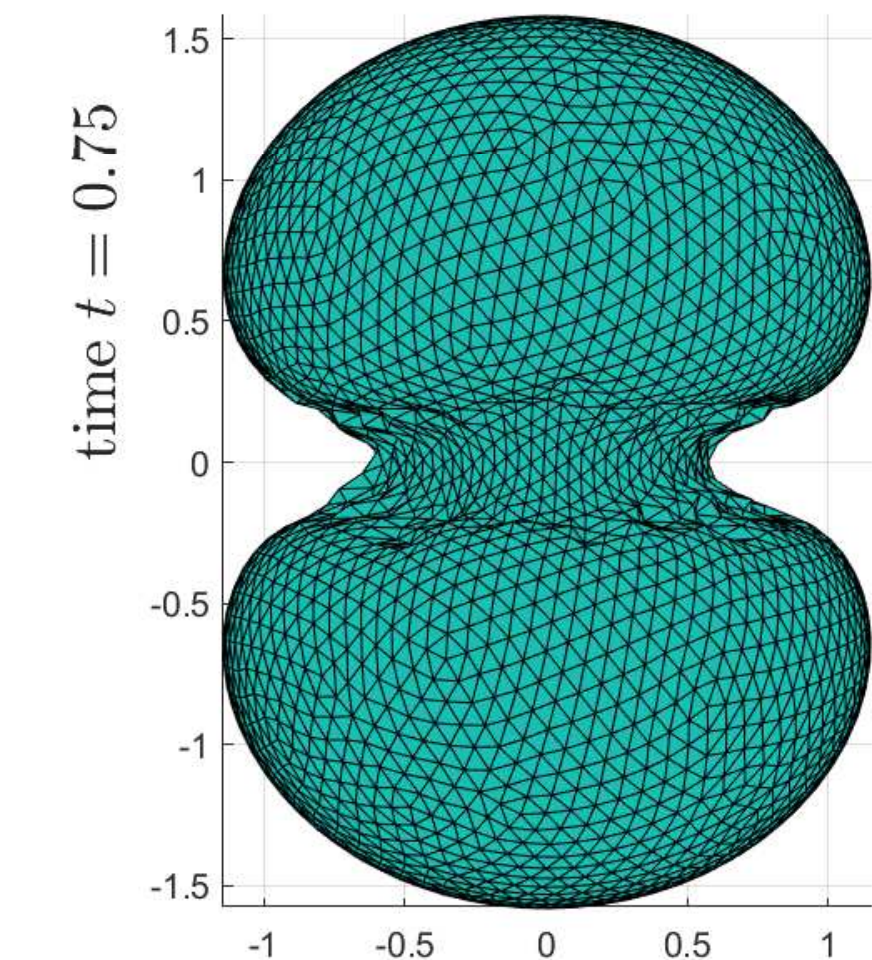} 
	\includegraphics[trim={60 20 110 20},clip,width=0.32\textwidth]{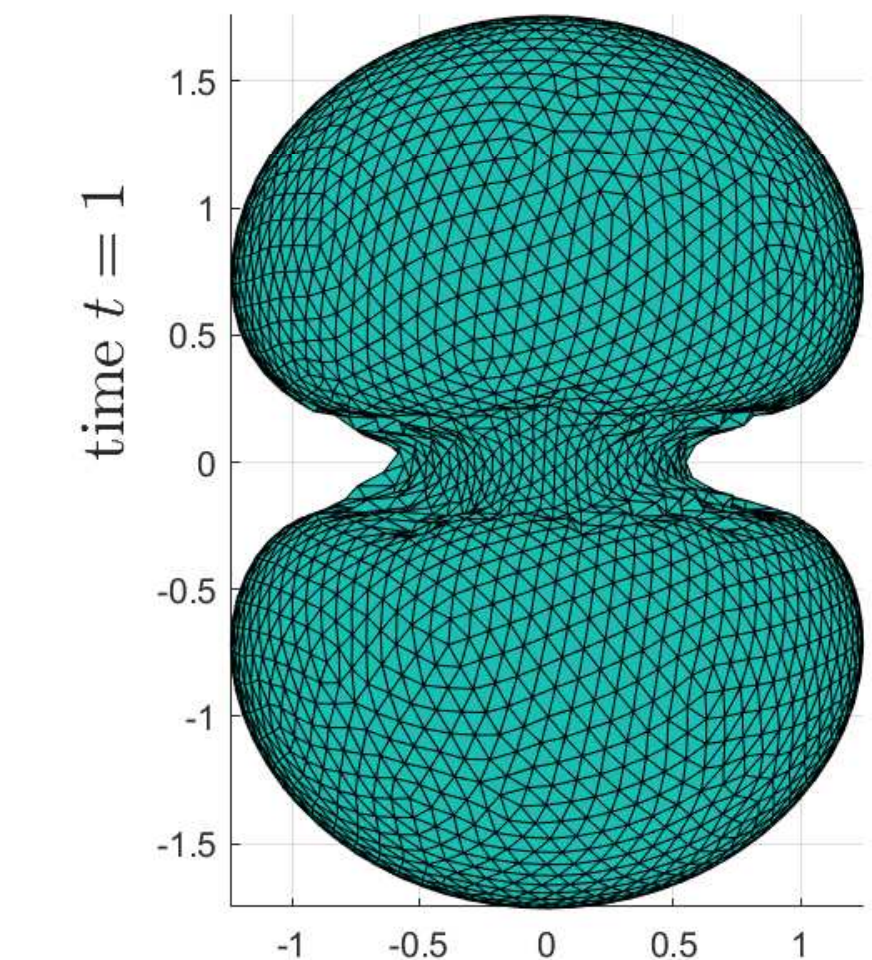} 
	\caption{Inverse mean curvature flow (side view) of a dumbbell in $\R^3$ at different times in $[0,1]$.}
	\label{fig:iMCF_dumbbell}
\end{figure}
	
\begin{figure}[htbp]
	\includegraphics[trim={60 20 90 20},clip,width=0.31\textwidth]{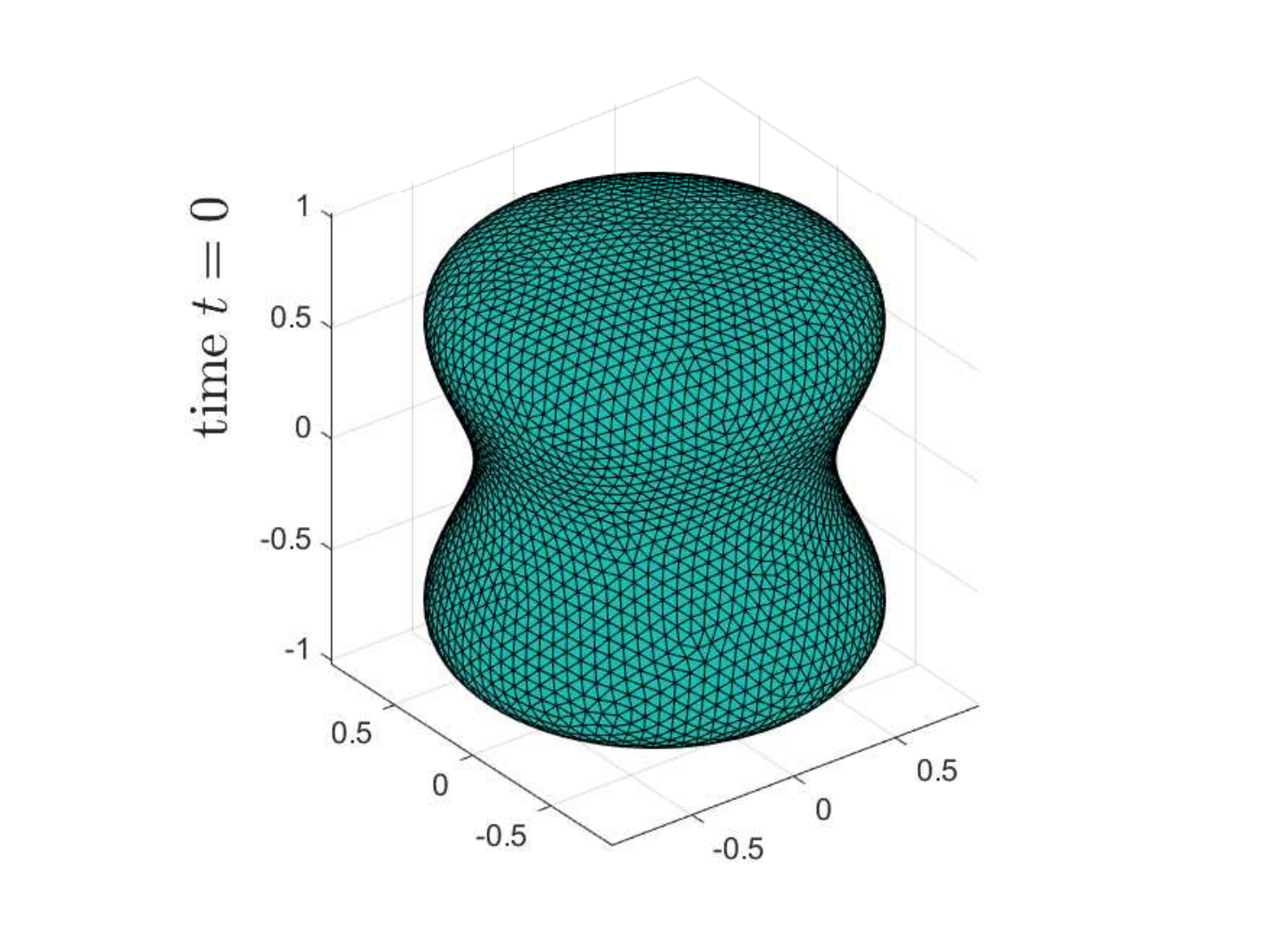} 
	\includegraphics[trim={60 20 90 15},clip,width=0.32\textwidth]{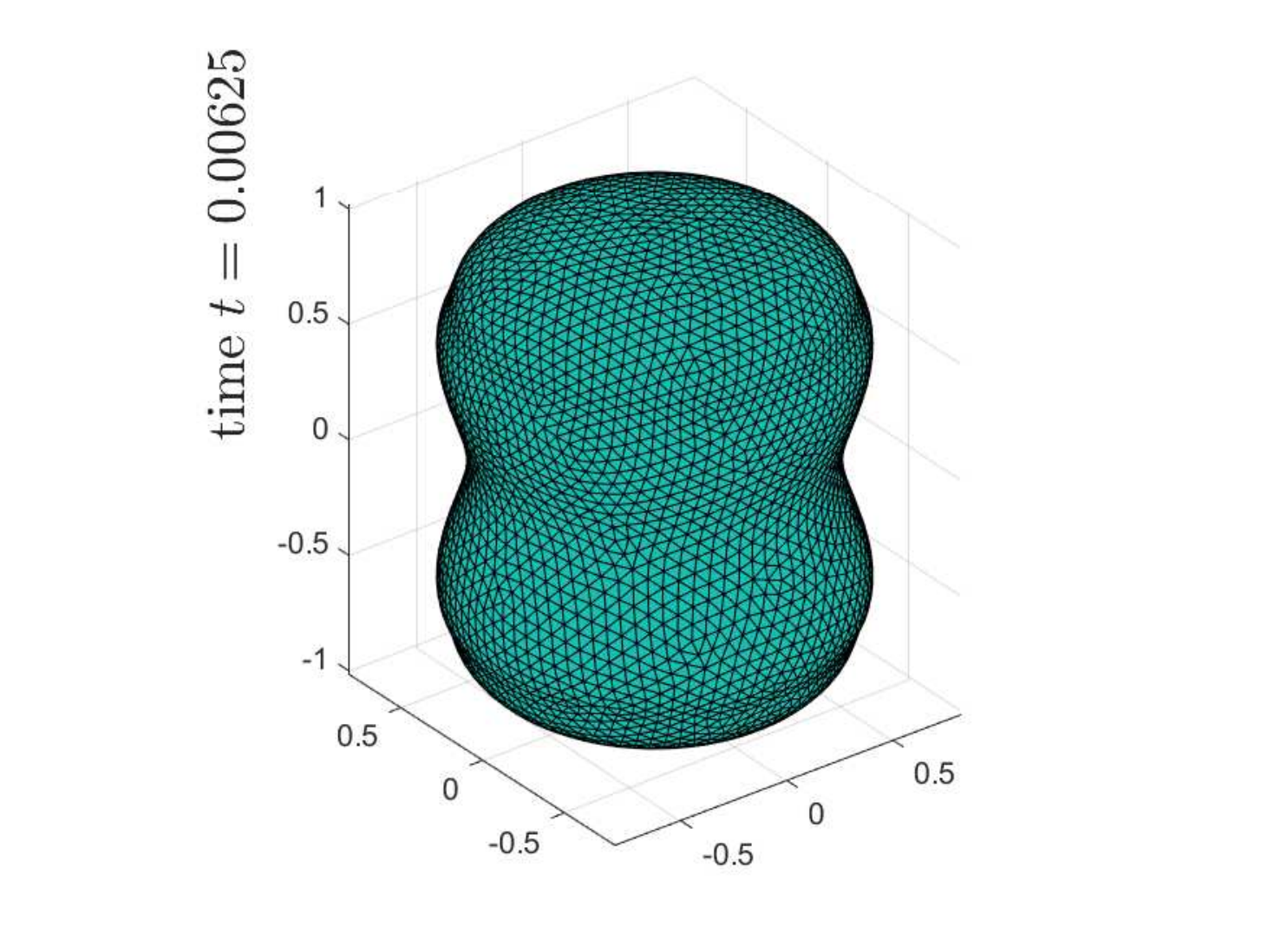}
	\includegraphics[trim={60 20 90 20},clip,width=0.32\textwidth]{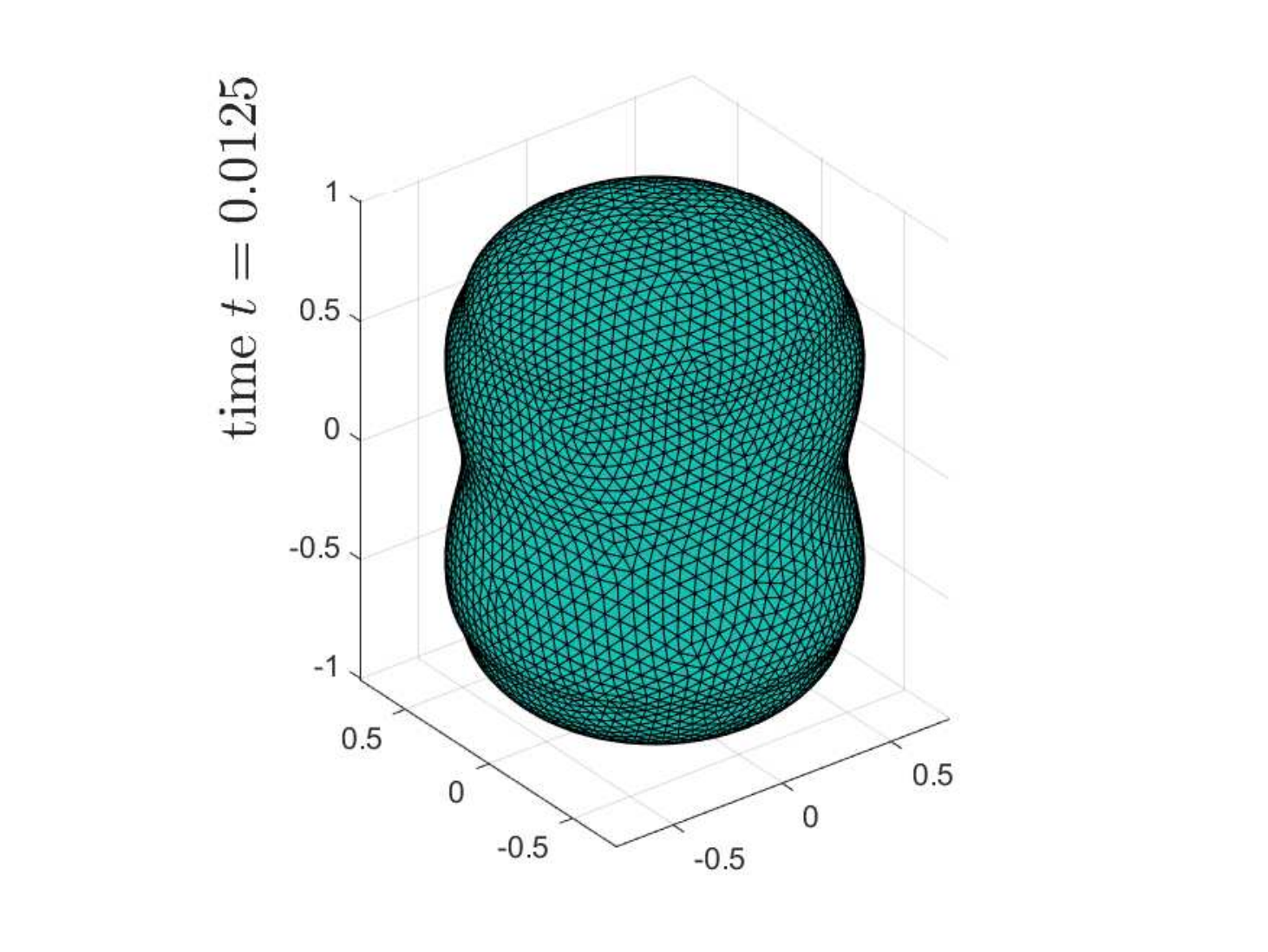}
	\\
	\includegraphics[trim={60 20 90 20},clip,width=0.32\textwidth]{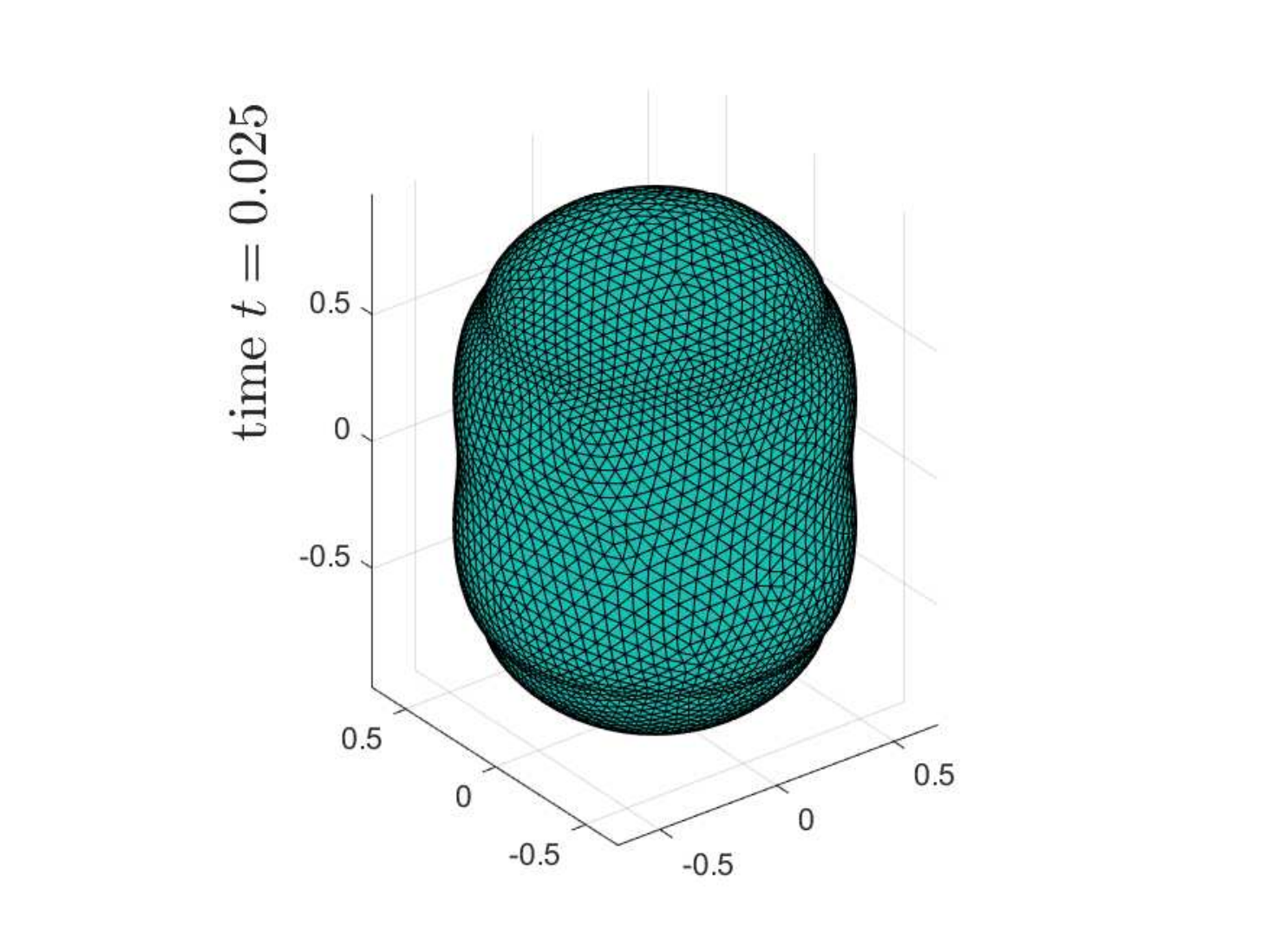}
	\includegraphics[trim={60 20 90 20},clip,width=0.32\textwidth]{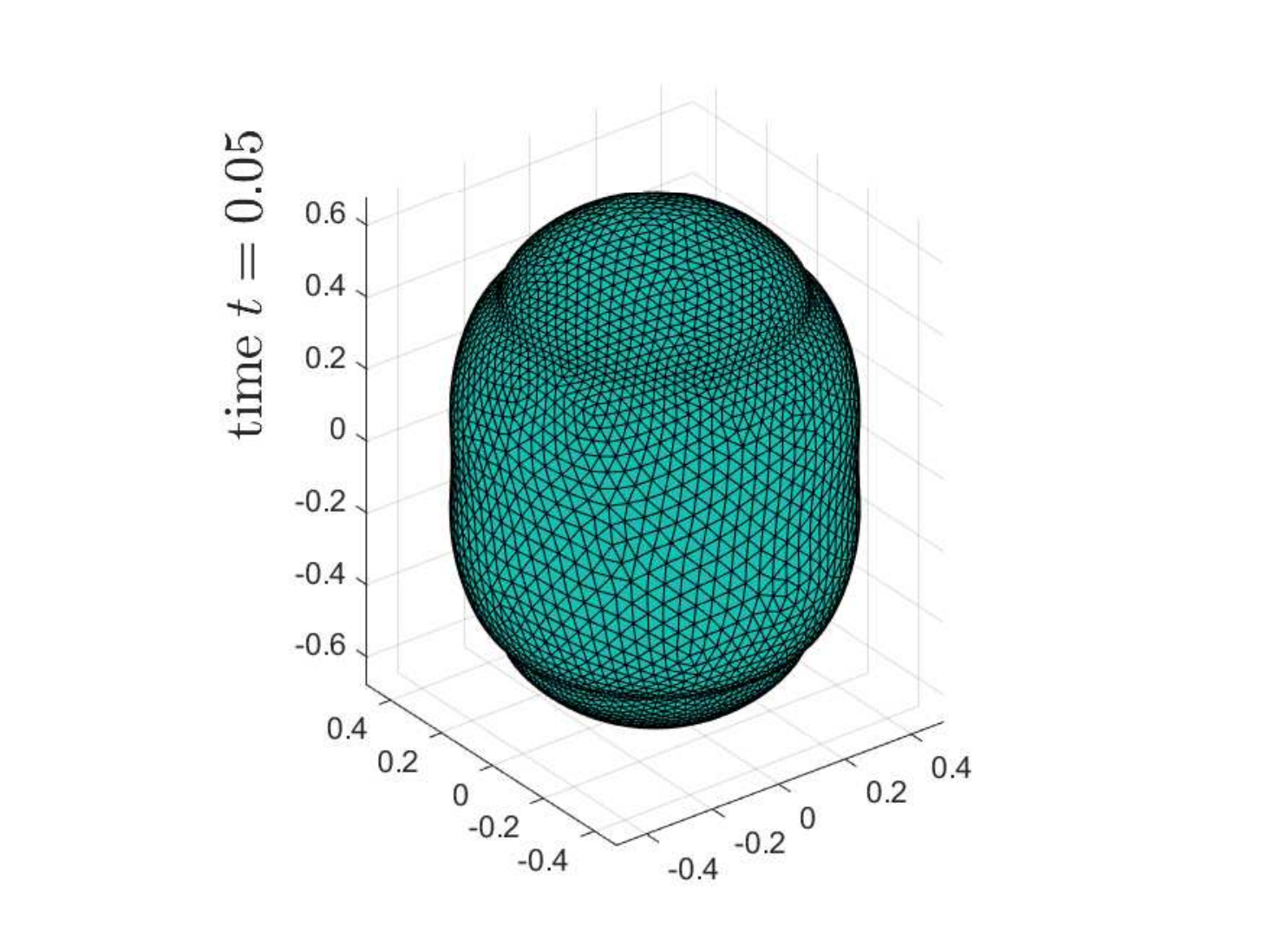}
	\includegraphics[trim={60 20 90 20},clip,width=0.32\textwidth]{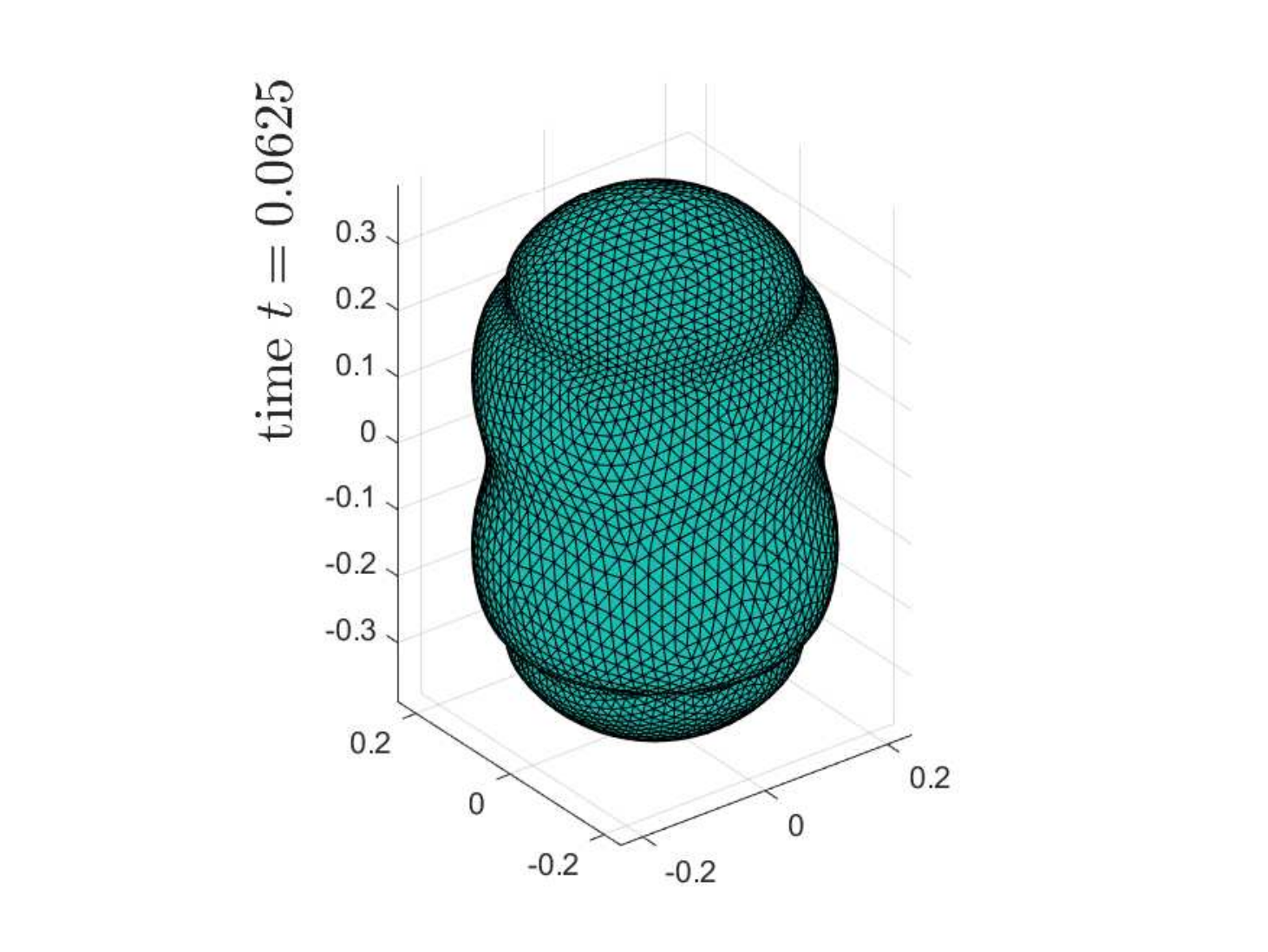}
	\caption{Powers of mean curvature flow with $\alpha = 2$ of a dumbbell in $\R^3$ at different times in $[0,0.0625]$.}
	\label{fig:MCFgen_dumbbell}
\end{figure}


\begin{figure}[htp]
	\includegraphics[trim={50 20 90 20},clip,width=0.31\textwidth]{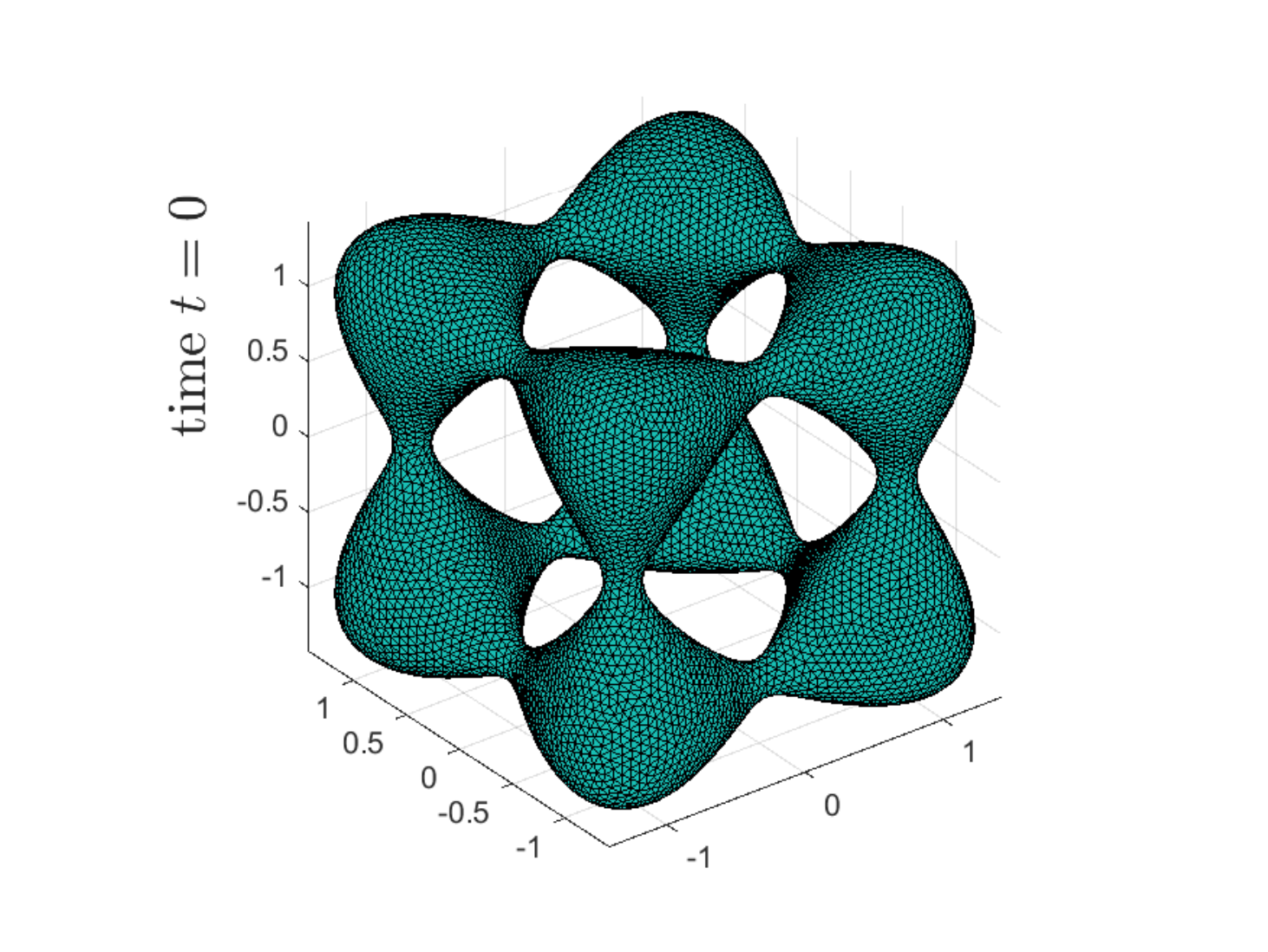} 
	\includegraphics[trim={50 20 90 20},clip,width=0.32\textwidth]{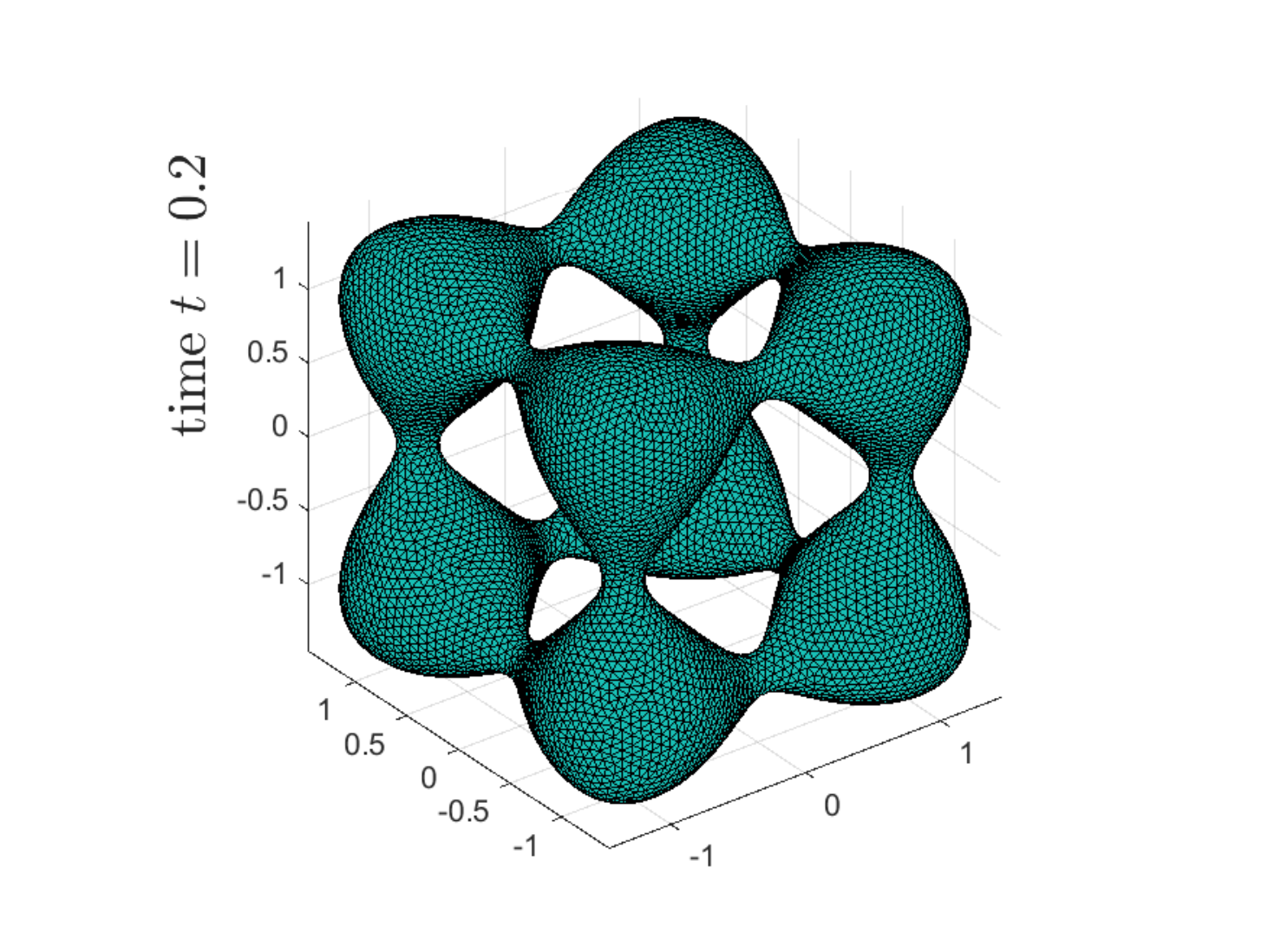} 
	\includegraphics[trim={50 20 90 20},clip,width=0.32\textwidth]{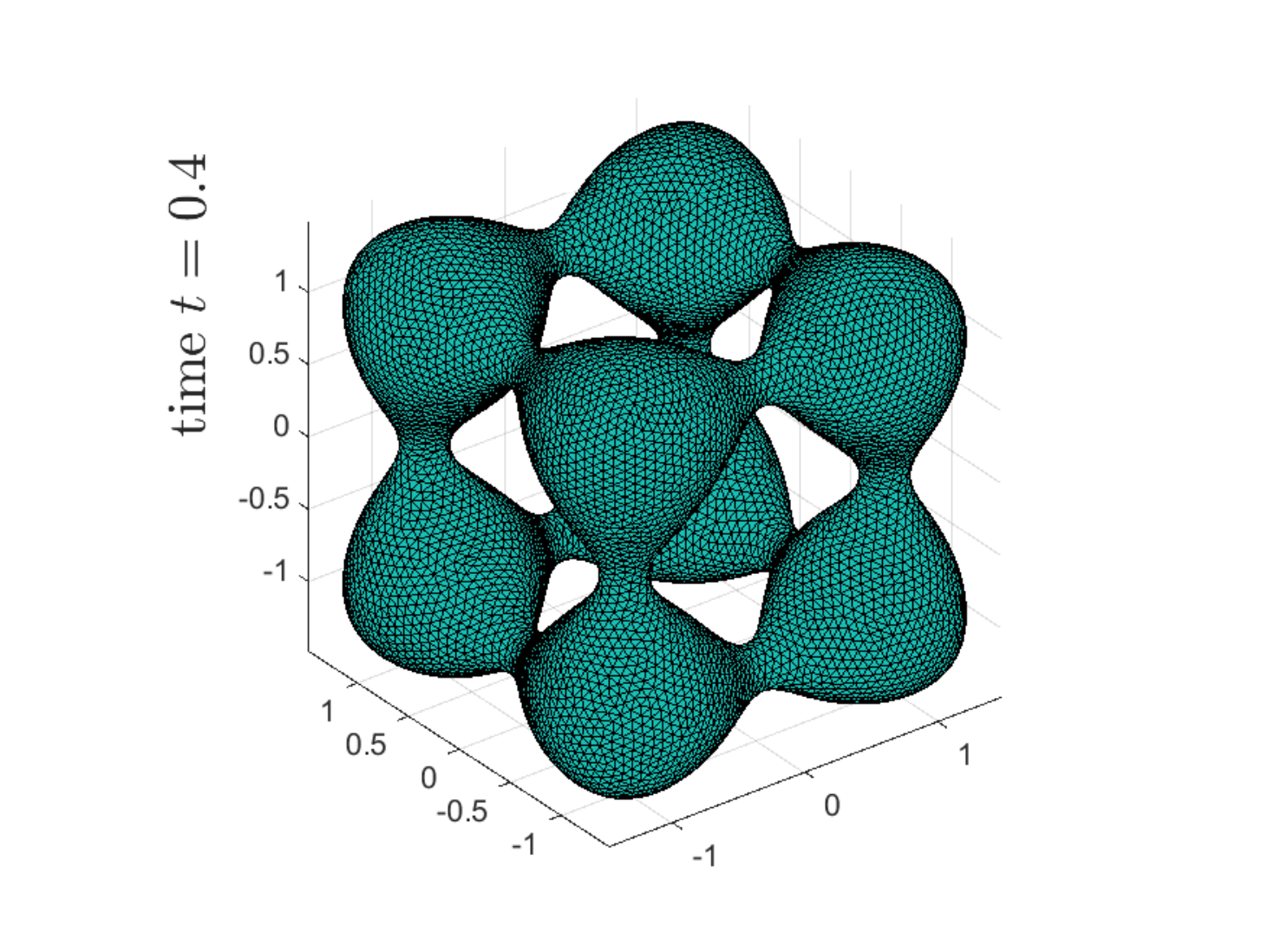} 
	\\
	\includegraphics[trim={50 20 90 20},clip,width=0.32\textwidth]{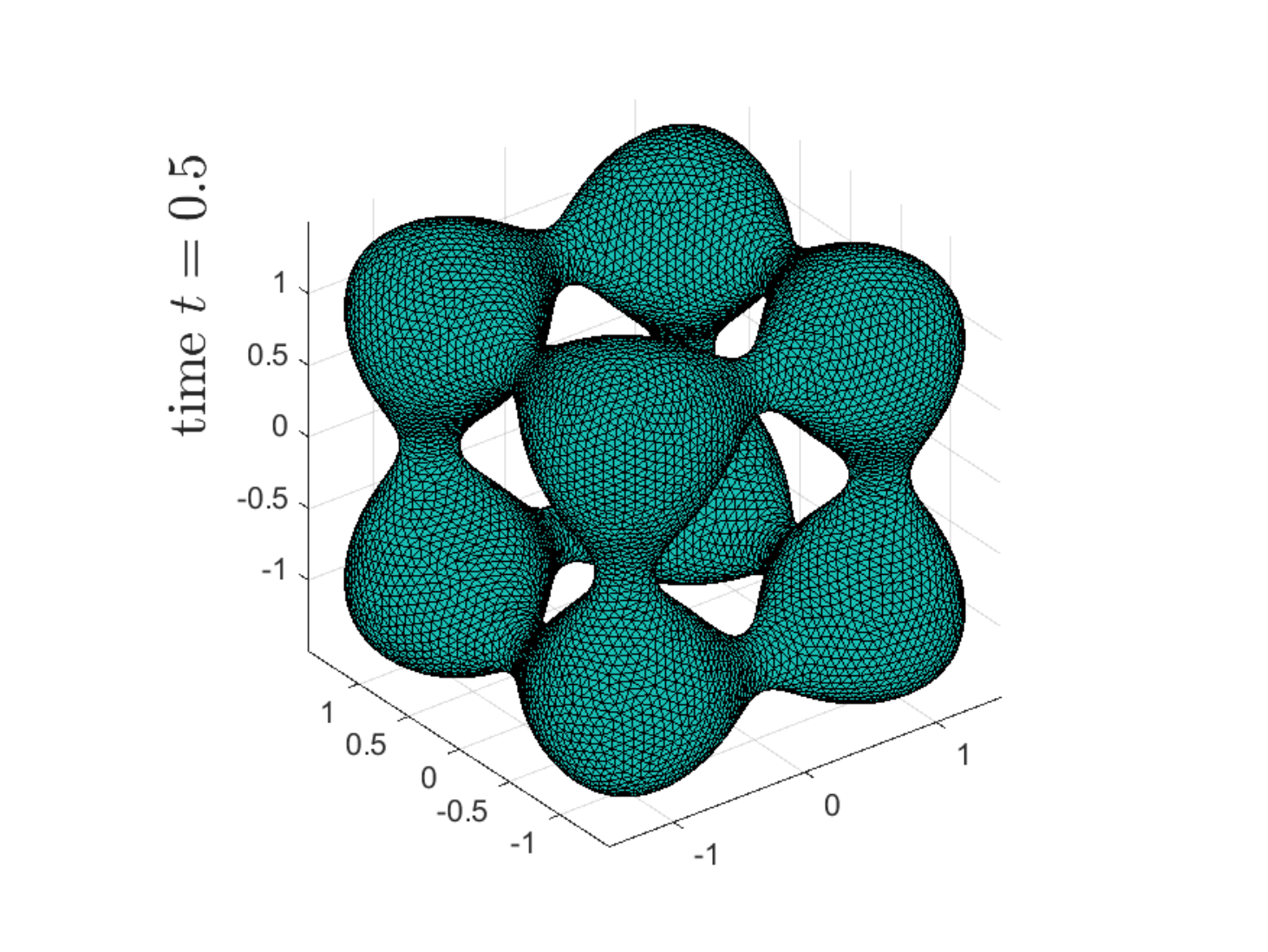} 
	\includegraphics[trim={50 20 90 20},clip,width=0.32\textwidth]{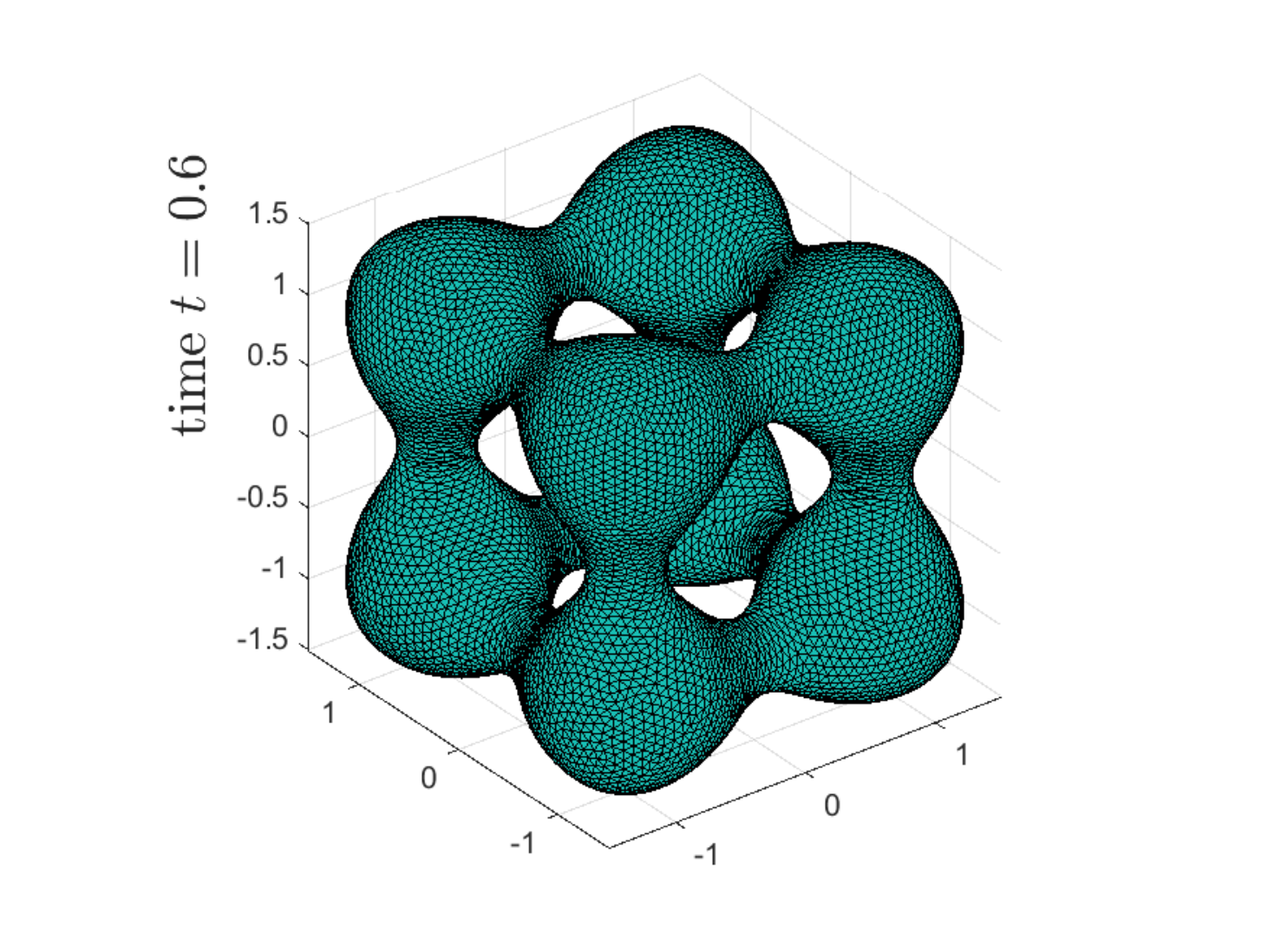} 
	\includegraphics[trim={50 20 90 20},clip,width=0.32\textwidth]{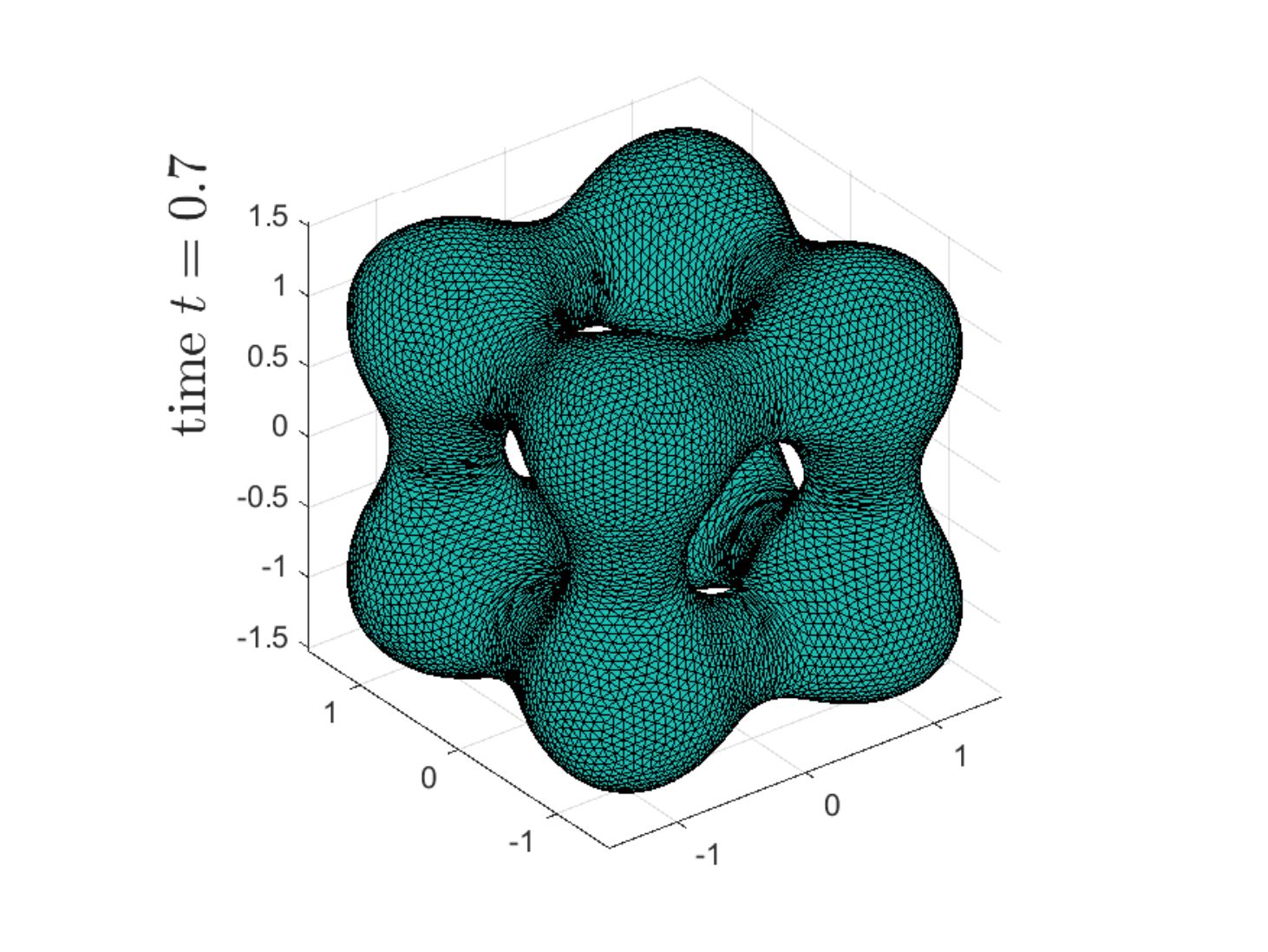} 
	\caption{Powers of inverse mean curvature flow with $\alpha = 2$ of a genus 5 surface in $\R^3$ at different times in $[0,0.7]$.}
	\label{fig:iMCFgen_genus5}
\end{figure}

\newpage

\section*{Acknowledgement}
We thank Simon Brendle and Christian Lubich for our inspiring discussions on the topic.

A significant portion of the manuscript was written when both authors were employed at the University of T\"{u}bingen. We gratefully acknowledge their support.

The work of Bal\'azs Kov\'acs is supported by the Deutsche Forschungsgemeinschaft (DFG, German Research Foundation) -- Project-ID 258734477 -- SFB 1173, and by the Heisenberg Programme of the Deutsche Forschungsgemeinschaft -- Project-ID 446431602.

\bibliographystyle{IMANUM-BIB}
\bibliography{MCF_literature}

\begin{thebibliography}{}

\bibitem[Akrivis {\em et~al.}(2017)Akrivis, Li, \&
  Lubich]{AkrivisLiLubich_quasilinBDF}
{\sc Akrivis, G., Li, B. \& Lubich, C.} (2017)
\newblock Combining maximal regularity and energy estimates for time
  discretizations of quasilinear parabolic equations.
\newblock {\em Math. Comp.}, {\bf 86}, 1527--1552.

\bibitem[Akrivis {\em et~al.}(2019)Akrivis, Feischl, Kov{\'a}cs, \&
  Lubich]{LLG}
{\sc Akrivis, G., Feischl, M., Kov{\'a}cs, B. \& Lubich, C.} (2019)
\newblock Higher-order linearly implicit full discretization of the
  {Landau--Lifshitz--Gilbert} equation.
\newblock arXiv:1903.05415.

\bibitem[Akrivis {\em et~al.}(2020)Akrivis, Chen, Yu, \&
  Zhou]{Akrivis_etal_BDF6}
{\sc Akrivis, G., Chen, M., Yu, F. \& Zhou, Z.} (2020)
\newblock The energy technique for the six-step {BDF} method.
\newblock {\em arXiv:2007.08924\/}.

\bibitem[Akrivis \& Lubich(2015)Akrivis \& Lubich]{AkrivisLubich_quasilinBDF}
{\sc Akrivis, G. \& Lubich, C.} (2015)
\newblock Fully implicit, linearly implicit and implicit--explicit backward
  difference formulae for quasi-linear parabolic equations.
\newblock {\em Numer. Math.}, {\bf 131}, 713--735.

\bibitem[Alessandroni \& Sinestrari(2010)Alessandroni \&
  Sinestrari]{AlessandroniSinestrari_nhMCF}
{\sc Alessandroni, R. \& Sinestrari, C.} (2010)
\newblock Convexity estimates for a nonhomogeneous mean curvature flow.
\newblock {\em Math. Z.}, {\bf 266}, 65--82.

\bibitem[Alvarez {\em et~al.}(1993)Alvarez, Guichard, Lions, \&
  Morel]{AlvarezGuichardLionsMorel}
{\sc Alvarez, L., Guichard, F., Lions, P.-L. \& Morel, J.-M.} (1993)
\newblock Axioms and fundamental equations of image processing.
\newblock {\em Arch. Rational Mech. Anal.}, {\bf 123}, 199--257.

\bibitem[Angenent {\em et~al.}(1998)Angenent, Sapiro, \&
  Tannenbaum]{AngenentSapiroTannenbaum}
{\sc Angenent, S., Sapiro, G. \& Tannenbaum, A.} (1998)
\newblock On the affine heat equation for non-convex curves.
\newblock {\em J. Amer. Math. Soc.}, {\bf 11}, 601--634.

\bibitem[Barrett {\em et~al.}(2007)Barrett, Garcke, \& N{\"u}rnberg]{BGN2007}
{\sc Barrett, J., Garcke, H. \& N{\"u}rnberg, R.} (2007)
\newblock On the variational approximation of combined second and fourth order
  geometric evolution equations.
\newblock {\em SIAM J. Sci. Comput.}, {\bf 29}, 1006--1041.

\bibitem[Barrett {\em et~al.}(2008)Barrett, Garcke, \& N{\"u}rnberg]{BGN2008}
{\sc Barrett, J., Garcke, H. \& N{\"u}rnberg, R.} (2008)
\newblock On the parametric finite element approximation of evolving
  hypersurfaces in {$\R^3$}.
\newblock {\em J. Comput. Phys.}, {\bf 227}, 4281--4307.

\bibitem[Barrett {\em et~al.}(2019)Barrett, Garcke, \&
  N{\"u}rnberg]{BGN_survey}
{\sc Barrett, J., Garcke, H. \& N{\"u}rnberg, R.} (2019)
\newblock Parametric finite element approximations of curvature driven
  interface evolutions.
\newblock {\em arXiv:1903.09462v1\/}.

\bibitem[Brenner \& Scott(2008)Brenner \& Scott]{BrennerScott}
{\sc Brenner, S.~C. \& Scott, R.} (2008)
\newblock {\em The mathematical theory of finite element methods\/},  vol.
  Texts in Applied Mathematics, 15.
\newblock Springer, New York.

\bibitem[Dahlquist(1978)Dahlquist]{Dahlquist}
{\sc Dahlquist, G.} (1978)
\newblock {G}-stability is equivalent to {A}-stability.
\newblock {\em BIT\/}, {\bf 18}, 384--401.

\bibitem[Deckelnick {\em et~al.}(2005)Deckelnick, Dziuk, \&
  Elliott]{DeckelnickDE2005}
{\sc Deckelnick, K., Dziuk, G. \& Elliott, C.} (2005)
\newblock Computation of geometric partial differential equations and mean
  curvature flow.
\newblock {\em Acta Numerica\/}, {\bf 14}, 139--232.

\bibitem[Demlow(2009)Demlow]{Demlow2009}
{\sc Demlow, A.} (2009)
\newblock Higher--order finite element methods and pointwise error estimates
  for elliptic problems on surfaces.
\newblock {\em SIAM J. Numer. Anal.}, {\bf 47}, 805--807.

\bibitem[Dziuk(1988)Dziuk]{Dziuk88}
{\sc Dziuk, G.} (1988)
\newblock Finite elements for the {B}eltrami operator on arbitrary surfaces.
\newblock {\em Partial differential equations and calculus of variations,
  Lecture Notes in Math., 1357, Springer, Berlin\/}, 142--155.

\bibitem[Dziuk {\em et~al.}(2012)Dziuk, Lubich, \&
  Mansour]{DziukLubichMansour_rksurf}
{\sc Dziuk, G., Lubich, C. \& Mansour, D.} (2012)
\newblock {R}unge--{K}utta time discretization of parabolic differential
  equations on evolving surfaces.
\newblock {\em IMA J. Numer. Anal.}, {\bf 32}, 394--416.

\bibitem[Dziuk \& Elliott(2007)Dziuk \& Elliott]{DziukElliott_ESFEM}
{\sc Dziuk, G. \& Elliott, C.} (2007)
\newblock Finite elements on evolving surfaces.
\newblock {\em IMA J. Numer. Anal.}, {\bf 27}, 262--292.

\bibitem[Ecker(2012)Ecker]{Ecker2012}
{\sc Ecker, K.} (2012)
\newblock {\em Regularity theory for mean curvature flow\/}.
\newblock Birkh\"{a}user, Boston.

\bibitem[Elliott \& Styles(2012)Elliott \& Styles]{ElliottStyles_ALEnumerics}
{\sc Elliott, C. \& Styles, V.} (2012)
\newblock An {ALE} {ESFEM} for solving {PDE}s on evolving surfaces.
\newblock {\em Milan J. Math.}, {\bf 80}, 469--501.

\bibitem[Espin(2020)Espin]{Espin_nhMCF}
{\sc Espin, T.} (2020)
\newblock A pinching estimate for convex hypersurfaces evolving under a
  nonhomogeneous variant of mean curvature flow.
\newblock {\em arXiv:2001.02546\/}.

\bibitem[Feng {\em et~al.}(2007)Feng, Neilan, \& Prohl]{FengNeilanProhl}
{\sc Feng, X., Neilan, M. \& Prohl, A.} (2007)
\newblock Error analysis of finite element approximations of the inverse mean
  curvature flow arising from the general relativity.
\newblock {\em Numer. Math.}, {\bf 108}, 93--119.

\bibitem[Frittelli {\em et~al.}(2018)Frittelli, Madzvamuse, Sgura, \&
  Venkataraman]{ESFEM_DMP}
{\sc Frittelli, M., Madzvamuse, A., Sgura, I. \& Venkataraman, C.} (2018)
\newblock Numerical preservation of velocity induced invariant regions for
  reaction-diffusion systems on evolving surfaces.
\newblock {\em J. Sci. Comput.}, {\bf 77}, 971--1000.

\bibitem[Gautschi(1997)Gautschi]{Gautschi}
{\sc Gautschi, W.} (1997)
\newblock {\em Numerical analysis\/}.
\newblock Birkh\"{a}user, Boston.
\newblock An introduction.

\bibitem[Gerhardt(2014)Gerhardt]{Gerhardt_pIMCF}
{\sc Gerhardt, C.} (2014)
\newblock Non-scale-invariant inverse curvature flows in {E}uclidean space.
\newblock {\em Calc. Var. Partial Differential Equations\/}, {\bf 49},
  471--489.

\bibitem[Hairer \& Wanner(1996)Hairer \& Wanner]{HairerWannerII}
{\sc Hairer, E. \& Wanner, G.} (1996)
\newblock {\em Solving Ordinary Differential Equations II. Stiff and
  Differential--Algebraic Problems\/}, {S}econd edition edn.
\newblock Springer, Berlin.

\bibitem[Hawking(1968)Hawking]{Hawking}
{\sc Hawking, S.~W.} (1968)
\newblock Gravitational radiation in an expanding universe.
\newblock {\em J. Mathematical Phys.}, {\bf 9}, 598--604.

\bibitem[Huisken(1984)Huisken]{Huisken1984}
{\sc Huisken, G.} (1984)
\newblock Flow by mean curvature of convex surfaces into spheres.
\newblock {\em J. Differential Geometry\/}, {\bf 20}, 237--266.

\bibitem[Huisken \& Ilmanen(2001)Huisken \& Ilmanen]{HuiskenIlmanen}
{\sc Huisken, G. \& Ilmanen, T.} (2001)
\newblock The inverse mean curvature flow and the {R}iemannian {P}enrose
  inequality.
\newblock {\em J. Differential Geom.}, {\bf 59}, 353--437.

\bibitem[Huisken \& Ilmanen(2008)Huisken \& Ilmanen]{HuiskenIlmanan_highreg}
{\sc Huisken, G. \& Ilmanen, T.} (2008)
\newblock Higher regularity of the inverse mean curvature flow.
\newblock {\em J. Differential Geom.}, {\bf 80}, 433--451.

\bibitem[Huisken \& Polden(1999)Huisken \& Polden]{HuiskenPolden}
{\sc Huisken, G. \& Polden, A.} (1999)
\newblock Geometric evolution equations for hypersurfaces.
\newblock {\em Calculus of variations and geometric evolution problems
  ({C}etraro, 1996)\/}. Lecture Notes in Math., vol. 1713.
\newblock Springer, Berlin, pp. 45--84.

\bibitem[Kov\'{a}cs {\em et~al.}(2017)Kov\'{a}cs, Li, Lubich, \& {Power
  Guerra}]{KLLP2017}
{\sc Kov\'{a}cs, B., Li, B., Lubich, C. \& {Power Guerra}, C.} (2017)
\newblock Convergence of finite elements on an evolving surface driven by
  diffusion on the surface.
\newblock {\em Numer. Math.}, {\bf 137}, 643--689.

\bibitem[Kov\'{a}cs(2018)Kov\'{a}cs]{Kovacs2018}
{\sc Kov\'{a}cs, B.} (2018)
\newblock High-order evolving surface finite element method for parabolic
  problems on evolving surfaces.
\newblock {\em IMA J. Numer. Anal.}, {\bf 38}, 430--459.

\bibitem[Kov{\'a}cs {\em et~al.}(2019)Kov{\'a}cs, Li, \& Lubich]{MCF}
{\sc Kov{\'a}cs, B., Li, B. \& Lubich, C.} (2019)
\newblock A convergent evolving finite element algorithm for mean curvature
  flow of closed surfaces.
\newblock {\em Numer. Math.}, {\bf 143}, 797--853.

\bibitem[Kov{\'a}cs {\em et~al.}(2020)Kov{\'a}cs, Li, \& Lubich]{Willmore}
{\sc Kov{\'a}cs, B., Li, B. \& Lubich, C.} (2020)
\newblock A convergent evolving finite element algorithm for willmore flow of
  closed surfaces.
\newblock {\em arXiv:2007.15257\/}.

\bibitem[Kov\'acs \& Lubich(2018)Kov\'acs \& Lubich]{KL2018}
{\sc Kov\'acs, B. \& Lubich, C.} (2018)
\newblock Linearly implicit full discretization of surface evolution.
\newblock {\em Numer. Math.}, {\bf 140}, 121--152.

\bibitem[Kov\'{a}cs \& {Power Guerra}(2016)Kov\'{a}cs \& {Power
  Guerra}]{KovacsPower_quasilinear}
{\sc Kov\'{a}cs, B. \& {Power Guerra}, C.} (2016)
\newblock Error analysis for full discretizations of quasilinear parabolic
  problems on evolving surfaces.
\newblock {\em NMPDE\/}, {\bf 32}, 1200--1231.

\bibitem[Kr\"{o}ner {\em et~al.}(2018)Kr\"{o}ner, Kr\"{o}ner, \&
  Kr\"{o}ner]{KroenerKroenerKroener_levelsetFEM_pMCF}
{\sc Kr\"{o}ner, A., Kr\"{o}ner, E. \& Kr\"{o}ner, H.} (2018)
\newblock Finite element approximation of level set motion by powers of the
  mean curvature.
\newblock {\em SIAM J. Sci. Comput.}, {\bf 40}, A4158--A4183.

\bibitem[Kr{\"{o}}ner(2013)Kr{\"{o}}ner]{Kroener_levelsetFEM_pMCF_arxiv}
{\sc Kr{\"{o}}ner, H.} (2013)
\newblock Finite element approximation of power mean curvature flow.
\newblock {\em arXiv:1308.2392\/}.

\bibitem[Kr{\"{o}}ner(2017)Kr{\"{o}}ner]{Kroener_levelset_pMCF}
{\sc Kr{\"{o}}ner, H.} (2017)
\newblock Approximation rates for regularized level set power mean curvature
  flow.
\newblock {\em Port. Math.}, {\bf 74}, 115--126.

\bibitem[Kr{\"{o}}ner(2019)Kr{\"{o}}ner]{Kroener_iMCf_pMCF}
{\sc Kr{\"{o}}ner, H.} (2019)
\newblock Analysis of constants in error estimates for the finite element
  approximation of regularized nonlinear geometric evolution equations.
\newblock {\em SIAM J. Numer. Anal.}, {\bf 57}, 2413--2435.

\bibitem[Lubich {\em et~al.}(2013)Lubich, Mansour, \&
  Venkataraman]{LubichMansourVenkataraman_bdsurf}
{\sc Lubich, C., Mansour, D. \& Venkataraman, C.} (2013)
\newblock Backward difference time discretization of parabolic differential
  equations on evolving surfaces.
\newblock {\em IMA J. Numer. Anal.}, {\bf 33}, 1365--1385.

\bibitem[Malladi \& Sethian(1995)Malladi \& Sethian]{MalladiSethian}
{\sc Malladi, R. \& Sethian, J.~A.} (1995)
\newblock Image processing via level set curvature flow.
\newblock {\em Proc. Nat. Acad. Sci. U.S.A.}, {\bf 92}, 7046--7050.

\bibitem[Nevanlinna \& Odeh(1981)Nevanlinna \& Odeh]{NevanlinnaOdeh}
{\sc Nevanlinna, O. \& Odeh, F.} (1981)
\newblock Multiplier techniques for linear multistep methods.
\newblock {\em Numer. Funct. Anal. Optim.}, {\bf 3}, 377--423.

\bibitem[Pasch(1998)Pasch]{diss_Pasch}
{\sc Pasch, E.} (1998)
\newblock {N}umerische {V}erfahren zur {B}erechnung von {K}r\"ummungsfl\"ussen.
\newblock {\em {PhD} thesis\/}, Universit\"{a}t T\"{u}bingen, Germany.

\bibitem[Persson \& Strang(2004)Persson \& Strang]{distmesh}
{\sc Persson, P.-O. \& Strang, G.} (2004)
\newblock A simple mesh generator in {MATLAB}.
\newblock {\em SIAM Review\/}, {\bf 46}, 329--345.

\bibitem[Sapiro \& Tannenbaum(1994)Sapiro \& Tannenbaum]{SapiroTannenbaum}
{\sc Sapiro, G. \& Tannenbaum, A.} (1994)
\newblock On affine plane curve evolution.
\newblock {\em J. Funct. Anal.}, {\bf 119}, 79--120.

\bibitem[Scheuer(2016)Scheuer]{Scheuer_pIMCF}
{\sc Scheuer, J.} (2016)
\newblock Pinching and asymptotical roundness for inverse curvature flows in
  {E}uclidean space.
\newblock {\em J. Geom. Anal.}, {\bf 26}, 2265--2281.

\bibitem[Schn\"{u}rer(2005)Schn\"{u}rer]{Schnurer}
{\sc Schn\"{u}rer, O.~C.} (2005)
\newblock Surfaces contracting with speed {$|A|^2$}.
\newblock {\em J. Differential Geom.}, {\bf 71}, 347--363.

\bibitem[Schoen \& Yau(1979)Schoen \& Yau]{SchoenYau}
{\sc Schoen, R. \& Yau, S.~T.} (1979)
\newblock On the proof of the positive mass conjecture in general relativity.
\newblock {\em Comm. Math. Phys.}, {\bf 65}, 45--76.

\bibitem[Schulze(2002)Schulze]{Schulze_diss}
{\sc Schulze, F.} (2002)
\newblock Nichtlineare evolution von hyperfl{\"a}chen entlang ihrer mittleren
  kr{\"u}mmung.
\newblock {\em {PhD} thesis\/}, University of T\"ubingen, Germany.
\newblock
  \url{https://publikationen.uni-tuebingen.de/xmlui/bitstream/handle/10900/48388/pdf/diss.pdf}.

\bibitem[Schulze(2005)Schulze]{Schulze_1}
{\sc Schulze, F.} (2005)
\newblock Evolution of convex hypersurfaces by powers of the mean curvature.
\newblock {\em Math. Z.}, {\bf 251}, 721--733.

\bibitem[Schulze(2006)Schulze]{Schulze_2}
{\sc Schulze, F.} (2006)
\newblock Convexity estimates for flows by powers of the mean curvature.
\newblock {\em Ann. Sc. Norm. Super. Pisa Cl. Sci. (5)\/}, {\bf 5}, 261--277.

\bibitem[Schulze(2008)Schulze]{Schulze_3}
{\sc Schulze, F.} (2008)
\newblock Nonlinear evolution by mean curvature and isoperimetric inequalities.
\newblock {\em J. Differential Geom.}, {\bf 79}, 197--241.

\bibitem[Walker(2015)Walker]{Walker2015}
{\sc Walker, S.~W.} (2015)
\newblock {\em The shape of things: a practical guide to differential geometry
  and the shape derivative\/}.
\newblock SIAM, Philadelphia.

\end{thebibliography}

%
%
%
%
%

\end{document}